%% file: article.tex
\documentclass[a4paper, 10pt]{article}
\usepackage[margin=1in]{geometry}
\usepackage{srcltx,graphicx,epstopdf}
\usepackage{amsmath, amssymb, amsthm}
\usepackage{color}
\usepackage{multirow}
\usepackage{overpic}
\usepackage{subfig} 
\usepackage{hyperref}
\usepackage{booktabs}
\usepackage{diagbox}

\usepackage{cleveref}
\crefname{section}{Section}{Sections}
\crefname{subsection}{Subsection}{Subsections}
\crefname{appendix}{Appendix}{Appendix}
\crefname{figure}{Figure}{Figures}
\crefname{table}{Table}{Tables}
\crefname{property}{Property}{Properties}
\crefname{theorem}{Theorem}{Theorem}
\crefname{criterion}{criterion}{criteria}

\graphicspath{{images/}}

\newtheorem{theorem}{Theorem}
\newtheorem{lemma}[theorem]{Lemma}

\newtheorem{property}[theorem]{Property}
\newtheorem{corollary}[theorem]{Corollary}
\newtheorem{definition}[theorem]{Definition}

\newtheorem{remark}{Remark}
\newtheorem{criterion}{Criterion}

\newcommand\bA{{\bf A}}
\newcommand\bB{{\bf B}}
\newcommand\bD{{\bf D}}
\newcommand\bDn{{\tilde{\bD}}}
\newcommand\bJ{{\bf J}}
\newcommand\bM{{\bf M}}
\newcommand\bMn{{\tilde{\bM}}}
\newcommand\bR{{\bf R}}

\newcommand\bLambda{{\bf \Lambda}}
\newcommand\bLambdan{{\tilde{\bLambda}}}

\newcommand\bbR{\mathbb{R}}
\newcommand\bbN{\mathbb{N}}

\newcommand\bbH{\mathbb{H}}

\newcommand\bg{\boldsymbol{g}}
\newcommand\br{\boldsymbol{r}}
\newcommand\bw{\boldsymbol{w}}
\newcommand\bF{{\boldsymbol{F}}}
\newcommand\bC{{\boldsymbol{C}}}
\newcommand\bS{{\boldsymbol{S}}}
\newcommand\bSn{{\tilde{\bS}}}
\newcommand\bU{{\boldsymbol{U}}}

\newcommand\dd{\,\mathrm{d}}

\newcommand\ctend{c\,t_{\mathrm{end}}}
\newcommand\cell{\mathrm{cell}}

\newcommand\PN{{$P_N$ }}
\newcommand\MN{{$M_N$ }}
\newcommand\MP[1]{{${M\!P}_{#1}$ }}
\newcommand\MPN{{\MP{\!N}}}
\newcommand\Mone{{$M_1$ }}
\newcommand\MPtwo{{\MP{2}}}
\newcommand\HMP[1]{{${H\!M\!P}_{#1}$ }}
\newcommand\HMPN{{\HMP{\!N}}}

\newcommand\mS{{\mathcal{S}}}
\newcommand\mR{{\mathcal{R}}}
\newcommand\mP{\mathcal{P}}
\newcommand\mPn{\tilde{\mathcal{P}}}
\newcommand\mK{{\mathcal{K}}}
\newcommand\mKn{{\tilde{\mK}}}

\newcommand\cs[2]{{\lambda^{(#1)}_{#2}}}

\newcommand\weight{\omega^{[\alpha]}}
\newcommand\pl{\phi^{[\alpha]}}
\newcommand\Pl{\Phi^{[\alpha]}}
\newcommand\Pdl{\Psi^{[\alpha]}}
\newcommand\weightn{\tilde{\omega}^{[\alpha]}}
\newcommand\Pln{\tilde{\Phi}^{[\alpha]}}
\newcommand\pln{\tilde{\phi}^{[\alpha]}}
\newcommand\spaceH{\bbH^{[\alpha]}}
\newcommand\spaceHn{\tilde{\bbH}^{[\alpha]}}

\newcommand\pd[2]{\dfrac{\partial {#1}}{\partial {#2}}}
\newcommand\od[2]{\dfrac{\dd {#1}}{\dd {#2}}}

\newcommand\inner[2]{\left\langle{#1},{#2}\right\rangle_{\spaceH_N}}
\newcommand\innern[2]{\left\langle{#1},{#2}\right\rangle_{\spaceHn_N}}
\newcommand\moment[1]{{\langle#1\rangle}}

\numberwithin{equation}{section}

\newcommand\delete[1]{}

\newcommand\add[1]{#1}

\title{A Nonlinear Hyperbolic Model for Radiative Transfer Equation in
  Slab Geometry}

\author{ 
Yuwei Fan\thanks{Department of Mathematics, Stanford University,
    Stanford, CA 94305, email: {\tt ywfan@stanford.edu}},~~
Ruo Li\thanks{HEDPS \& CAPT, LMAM \& School of Mathematical
    Sciences, Peking University, Beijing, China, email: {\tt rli@math.pku.edu.cn}}, 
  ~~and
Lingchao Zheng\thanks{School of Mathematical Sciences, Peking
    University, Beijing, China, email: {\tt lczheng@pku.edu.cn}} 
}

\begin{document}
\maketitle

\input{intro}
\input{RTE}
\input{hyper}
\input{num}
\input{conclusion}

\bibliographystyle{abbrv}
\bibliography{../../article,../references}
\end{document}

%% file: intro.tex
\begin{abstract}

  Linear models for the radiative transfer equation have been well
  developed, while nonlinear models are seldom investigated even for
  slab geometry due to some essential difficulties. We have proposed a
  moment model in \cite{MPN} for slab geometry, which combines the
  ideas of the classical \PN and \MN model. Though the model is far
  from perfect, it was demonstrated to be quite efficient in
  numerically approximating the solution of the radiative transfer
  equation, that we are motivated to improve this model further. 
  Consequently, we propose in this paper a new model following
  the chartmap in \cite{MPN} with some significant theoretic
  progresses. The new model is derived with global hyperbolicity, and
  meanwhile some necessary physical properties are preserved. We give
  a complete analysis of the characteristic structure and propose a
  numerical scheme for the new model. Numerical examples are presented
  to demonstrate the numerical performance of the new model.

  \vspace*{4mm}
\end{abstract}

\section{Introduction}
In kinetic theory, the radiative transfer equation (RTE), which is the
evolution equation of the specific intensity, describes the motion of
photons and their interaction with the background medium. In the past
decades, it has many applications in different fields, for instance,
radiation astronomy \cite{mihalas1978stellar}, reactor physics
\cite{pomraning1973equations, duderstadt1979transport}, atmospheric
radiative transfer \cite{marshak20053d}, and optical imaging
\cite{klose2002optical,tarvainen2005hybrid}.
The RTE is a high-dimensional integro-differential kinetic equation,
so how to develop effective numerical methods for RTE is an important
issue. The common numerical methods can be classified into two
categories: the probabilistic methods like the direct simulation Monte
Carlo (DSMC) methods \cite{Bird, hayakawa2007coupled,
  densmore2012hybrid}, and the deterministic schemes
\cite{broadwell1964study, larsen2010advances,
  Stamnes1988Electromagnetic, jeans1917stars, Davison1960on,
  dubroca1999theoretical, minerbo1978maximum,
  alldredge2016approximating, fan2018fast, MPN}, such as the discrete
ordinates method ($S_N$) \cite{broadwell1964study, larsen2010advances,
  Stamnes1988Electromagnetic}, the moment methods
\cite{jeans1917stars, Davison1960on, dubroca1999theoretical,
  alldredge2016approximating, MPN} and etc.

The DSMC method, introduced by Bird in \cite{Bird}, follows a
representative set of photons as they interact with background and
move in physical space. So far, this method has made remarkable
successes in solving the RTE, but the statistical scatter (or
statistical noise) is the main issue for its accuracy.  In order to
improve accuracy, one needs to increase the number of photons, which
significantly increases both the computational and memory
requirements.

The discrete ordinates method ($S_N$), which is one of the most
popular deterministic methods, solves the transport equation along
with a discrete set of angular directions from a given quadrature set.
However, the $S_N$ model is based on the assumption that the particles
can only move along the directions in the quadrature set, which
results in numerical artifacts, known as \emph{ray effects}
\cite{larsen2010advances}.

The moment method studies the evolution of a finite number of moments
of the specific intensity. Typically, the evolution equation of a
lower order moment depends on higher order moments. Hence one has to
introduce the called \emph{moment closure} to close the moment model.
A common method for the moment closure is to construct an ansatz to
approximate the specific intensity and the two most popular moment
methods are the spherical harmonics method ($P_N$)
\cite{pomraning1973equations} and the maximum entropy method ($M_N$)
\cite{levermore1996moment, dubroca1999theoretical,
  minerbo1978maximum}.  The \PN model constructs the ansatz by
expanding the specific intensity around the equilibrium in terms of
spherical harmonics in the velocity direction. However, the resulting
model may lead to nonphysical oscillations, or even worse, negative
particle concentration \cite{brunner2001one, brunner2005two,
  mcclarren2008solutions}. The \MN model constructs the ansatz based
on the principle of maximum entropy \cite{levermore1996moment,
  dubroca1999theoretical}. However, no algebraic expression of the
moment closure is known for the case $N\geq 2$, and one has to solve
an ill-conditioned optimization problem to obtain the moment closure
in the implementation, which strongly limits the application of the
\MN model.

Recently, a nonlinear moment model (called the \MPN model) was
proposed in \cite{MPN}. This model takes the ansatz of the \Mone model
(the first order of the \MN model) as the weight function, then
constructs the ansatz by expanding the specific intensity around the
weight function in terms of orthogonal polynomials in the velocity
direction. The \MPN model is a nonlinear model since the weight
function contains the energy flux of the intensity. Numerical tests in
\cite{MPN} demonstrate its numerical efficiency and show that the \MPN
model produces an improved approximation of the intensity in
comparison of the \PN model in \cite{MPN}. Moreover, it was proved
that the \MPN model with $N=2$ is globally hyperbolic in the
realizable domain.

In spite of its numerous progress, the \MPN model is, however, far from
perfect. The theoretical investigation shows that the \MPN model with
$N\geq 3$ loses its hyperbolicity when the specific intensity is far
away from the equilibrium. For the case $N=2$, the \MPN model might
give unphysical characteristic speeds. Precisely, the characteristic
speeds could be faster than the speed of light. Detailed discussion is
presented in \cref{sec:defects}. These defects limit the application
of the \MPN model on the strong non-equilibrium problems and
time-dependent problems. Encouraged by the improved numerical
performance of the \MPN model, we are motivated to study further
following this spirit to get rid of these limits. The top object is to
study how to gain hyperbolicity of the \MPN model.

Although the hyperbolicity is a critical issue for the moment model,
there are not too many works on the hyperbolic regularization till
now. The well-known entropy-based \MN model \cite{levermore1996moment}
is globally hyperbolic because its ansatz is an exponential function,
which makes the model symmetric hyperbolic. However, it does not provide
clues for other models to gain hyperbolicity. The first globally
hyperbolic regularization was proposed in \cite{Fan, Fan_new}, where
the authors study the regularization by investigating the coefficient
matrix of the reduced model. This work was extended to a general
framework on deriving a hyperbolic reduced model for generic kinetic
equations in \cite{framework, Fan2015} based on the operator
projection (or truncation). Many follow-up works were proposed after
that, for instance \cite{Tiao, DFT, Koellermeier, di2016quantum,
  kuang2017globally}. We refer readers to \cite{julian2017diagram} and
references therein for more details. Hence, a natural idea is to apply
the hyperbolic regularization framework in \cite{framework,Fan2015} on
the \MPN model to yield a globally hyperbolic model. However, though
being hyperbolic, the resulting model is not satisfied since it
changes the model even for $N=1$, in which case the \MPN model is
precisely the \Mone model. This indicates the resulting model may not
able to yield a correct high-order Eddington approximation. We are
obliged to develop additional techniques to attain a satisfied model.

In this paper, we first discuss some natural criteria to improve the
\MPN model by hyperbolic regularization:
\begin{enumerate}
\item the regularized model is globally hyperbolic;
\item the characteristic speeds of the regularized model cannot be
  faster than the speed of light;
\item the regularization vanishes for the case $N=1$;
\item only the evolution equations that are closed by the moment
  closure can be changed.
\end{enumerate}
The first criterion is our goal, and the second one is a natural
physical constraint. The third one is to guarantee the correctness of
the high-order Eddington approximation, and the last one is to attain
high efficiency. More discussion is presented in \cref{c1,c2,c3,c4}.

Taking these criteria into account, we notice that the key idea of the
framework in \cite{framework, Fan2015} is that in the convection term,
the spatial derivative operator $\pd{\cdot}{z}$ and the multiplying
velocity operator $\mu\cdot$ are not coupled in the space defined by
the specific intensity but they are coupled in the linear space
defined by the ansatz, and then the authors decoupled these two
operators to gain the hyperbolicity. Keeping such an idea in mind, we are
inspired to propose a modified hyperbolic regularization. Making use
of the weight function and the ansatz of the \MPN model, we introduce
a new space that is defined by the derivative of the weight function
with respect to the parameter in the weight function. This makes us
decouple the spatial derivative operator and the multiplying velocity
operator in this new space. Consequently, the resulting moment model
satisfies all the criteria, saying that the new model is not only
globally hyperbolic but also retains some physical properties of the
RTE. The characteristic structure of the new model is well
studied. Moreover, the new hyperbolic regularization generalizes the
framework in \cite{framework, Fan2015}, extends its application range
and also takes properties of the kinetic equation into account of the
regularization.

To develop a numerical scheme for the new model, we adopt the DLM
theory \cite{Maso} to deal with the non-conservative part by
introducing a generalized Rankine-Hugoniot condition. Then the
numerical scheme in \cite{Rhebergen}, which can be treated as a
non-conservative version of the HLL, is applied to discretize the
non-conservative system. Numerical simulations are performed to
demonstrate the numerical efficiency of the new model. Thanks to the
hyperbolic regularization, the new model works well for the case the
\MPN model fails. The simulations on benchmark problems show that the
new model has good agreement with the reference solution.


The rest of this paper is arranged as follows. \Cref{sec:RTE} briefly
introduces the RTE and the \MPN model. Particularly, we try to discuss
the defects of the \MPN model in detail to clarify the improvements
in the new model. In \cref{sec:regularization}, we point out the
failure of the hyperbolic regularization framework in
\cite{framework,Fan2015} and propose the generalized hyperbolic
regularization method for the \MPN model to yield the new
model. Numerical scheme and numerical results for the new model are
presented in \cref{sec:num}. The paper ends with a conclusion in
\cref{sec:conclusion}.


%% file: RTE.tex
\section{\MPN Model for Radiative Transfer Equation}\label{sec:RTE}
The time-dependent radiative transfer equation (RTE) for a grey medium in the slab geometry has the
form
\begin{equation}\label{eq:RTE}
    \frac{1}{c}\pd{I}{t}+\mu\pd{I}{z}=\mS(I),
\end{equation}
where $I=I(z,t,\mu)$ is the specific intensity of radiation and $c$ is the speed of light.
The variable $\mu\in[-1,1]$ is the cosine of the angle between the photon velocity and the positive
$z$-axis.
The right hand side $\mS(I)$ denotes the actions by the background medium on the photons, and it
usually contains a scattering term, an absorption term, and an emission term.

\subsection{Moment method}
Denote the $k$-th moment of the specific intensity by 
\begin{equation}
    \moment{I}_k \triangleq \int_{-1}^1\mu^k I(\mu)\dd\mu,\quad k\in \bbN,
\end{equation}
then multiplying \eqref{eq:RTE} by $\mu^k$ and integrating it with respect to $\mu$ over $[-1,1]$
yields the moment equations
\begin{equation}
    \frac{1}{c}\pd{\moment{I}_k}{t} + \pd{\moment{I}_{k+1}}{z}
    = \moment{\mS(I)}_k,\quad k\in\bbN.
\end{equation}
Notice that the governing equation of $\moment{I}_k$ depends on the $(k+1)$-th moment
$\moment{I}_{k+1}$, which indicates that the full system contains infinite number of
equations. Thus, in order to derive a reduced model for \eqref{eq:RTE}, we choose a positive integer
$N$ and discard all the governing equation of $\moment{I}_k$, $k>N$. Clearly, the truncated 
system is not closed due to its dependence on $\moment{I}_{N+1}$, so we need to provide a so-called
\emph{moment closure} for the model.
A common strategy of the moment closure is to construct an ansatz for the specific intensity.
Precisely, let $E_k$, $k=0,\dots,N$, be the $k$-th known moments for a certain unknown specific
intensity $I$. One can propose an approximation, also called \emph{ansatz}, $\hat{I}(E_0,\dots,E_N;
\mu)$ such that
\begin{equation}\label{eq:constraints}
    \moment{\hat{I}(E_0,\dots,E_N;\cdot)}_k=E_k,\quad k=0,\dots,N,
\end{equation}
and $\hat{I}$ is uniquely determined by \eqref{eq:constraints}. Then the moment closure is given by
\begin{equation}\label{eq:closure}
    E_{N+1} = \moment{\hat{I}(E_0,\dots,E_N;\cdot)}_{N+1},
\end{equation}
and the moment model is proposed to be the system
\begin{equation}\label{eq:momentsystem}
    \frac{1}{c}\pd{E_k}{t}+\pd{E_{k+1}}{z} = \moment{\mS(\hat{I}(E_0,\dots,E_N;\mu))}_k,
    \quad k=0,\dots,N.
\end{equation}

Based on the moment closure strategy, many existing models are developed in the literature,
for example, the \PN model \cite{jeans1917stars}, the \MN model \cite{levermore1996moment,
dubroca1999theoretical}, the positive \PN model \cite{hauck2010positive}, the $B_2$ model
\cite{alldredge2016approximating}, and the \MPN model \cite{MPN}. The \MPN model proposed in
\cite{MPN} shows good numerical results for some standard benchmarks. In this paper, we will restudy
this model and point out its defects in both theoretical analysis and limitation on numerical
simulations, and then propose a novel regularization for this model.

\subsection{\MPN model}\label{sec:MPNmodel}
The \MPN model starts from introducing the weight function
\begin{equation}\label{eq:weight}
    \weight(\mu) = \frac{1}{(1+\alpha\mu)^4},\quad \alpha\in(-1,1).
\end{equation}
Using the Gram-Schmidt orthogonalization, one can directly define a series of monic orthogonal
polynomials in the interval $[-1,1]$ with respect to the weight function $\weight(\mu)$ recursively
as
\begin{align}\label{eq:Gram-Schmidt}
  \pl_0(\mu) = 1, \quad \pl_j(\mu) = \mu^j - \sum_{k=0}^{j-1}
  \frac{\mK_{j,k}} {\mK_{k,k}} \pl_k(\mu), \quad j\geq 1,
\end{align}
where  the coefficients $\mK_{j,k}$ is given by 
\begin{align}
  \mK_{j,k} = \int_{-1}^1\mu^j\pl_k(\mu)\weight(\mu)\dd\mu.
\end{align}
The orthogonality of $\pl_k$ yields
\begin{equation}
    \mK_{j,k} = 0, \text{ if } j<k,\quad \mK_{k,k}=\int_{-1}^1(\pl_k(\mu))^2\weight\dd\mu>0.
\end{equation}
The ansatz of the \MPN model is defined as
\begin{equation}\label{eq:ansatz}
    \hat{I}(E_0, \dots, E_N;\mu) \triangleq \sum_{i=0}^N f_i \Pl_i(\mu), 
\end{equation}
where $\Pl_i(\mu)=\pl_i(\mu)\weight(\mu), i = 0,1,\dots,N$ 
are the basis functions, and $f_i$ are the expansion coefficients 
to be determined by the moment
constraints \eqref{eq:constraints}. Thanks to the orthogonality of $\pl_i$,
we have
\[
    f_i = \frac{1}{\mK_{i,i}} \int_{-1}^1 \pl_i(\mu)\hat{I}(\mu)\dd\mu,
    \quad i=0,\dots,N.
\]
Substituting the recursive relationship \eqref{eq:Gram-Schmidt} into the upper
equation yields the following recursive formulation for $f_i$, which are
functions dependent on $E_i$,
\begin{equation}\label{eq:fk_Ek}
    f_i = \frac{1}{\mK_{i,i}} \left(E_i - \sum_{j=0}^{i-1}
    \mK_{i,j} f_j\right),\quad 0\leq i\leq N.
\end{equation}
The moment closure is then given by 
\begin{equation}\label{eq:MPNclosure}
    E_{N+1} = \sum_{k=0}^N \mK_{N+1,k} f_k.
\end{equation}
For the \MPN model, the parameter $\alpha$ is set as $\alpha=
-\frac{3E_1/E_0}{2+\sqrt{4-3(E_1/E_0)^2}}$. In this case, direct calculations yield
\begin{equation}\label{eq:f1}
    f_1=0.
\end{equation}

\subsection{Defects of the \MPN model}\label{sec:defects}
The \MPN model has been well studied in \cite{MPN}. It was shown that the ansatz \eqref{eq:ansatz}
had a better approximation to the specific intensity than the \PN model, and the \MPN model was
numerically demonstrated to be effective in approximating the RTE. As a particular case, the \MPtwo
model was well studied, including its hyperbolicity and characteristic field for the Riemann
problem. Nevertheless, the \MPN model is far from perfect, and it has some defects in the theoretical
analysis, which limits its application in numerical simulations.

There are a lot of criteria to judge a reduced model. Among them, the following criteria are
basic conditions for a physical model:
\begin{criterion}\label{c1}
    The reduced model is globally hyperbolic.
\end{criterion}
\begin{criterion}\label{c2}
    The characteristic speeds of the reduced model lie in $[-c, c]$.
\end{criterion}

The \cref{c1} uses the following definition.
\begin{definition}[Global hyperbolicity] \label{def:hyperbolicity}
    A system of first order quasi-linear partial differential equations
    \begin{equation*}
        \pd{\bw}{t}+\bA(\bw)\pd{\bw}{z}=0, \quad \bw\in\Omega
    \end{equation*}
    is called \emph{hyperbolic} at the point $\bw_0\in\Omega$ if the matrix $\bA(\bw_0)$ is
    diagonalizable with real eigenvalues. The system is called \emph{globally hyperbolic} if it is
    hyperbolic at each point $\bw\in\Omega$.
\end{definition}
Since the left hand side of the RTE \eqref{eq:RTE} is an advection part, the \cref{c1} is the
necessary condition for the existence of the solution. Thus, the hyperbolicity is a critical
mathematical constraint on the reduced model. The \cref{c2} is a basic physical property of the
reduced model, which can be interpreted as that the information can not travel faster than the speed
of light. However, as will be shown, the \MPN model fails to satisfy these criteria.

\subsubsection{Loss of global hyperbolicity}
In \cite{MPN}, the \MPtwo model was proved to be globally hyperbolic in its realizability domain.
However, the global hyperbolicity fails to be preserved by the \MPN model with $N>2$. In the
following, we take the \MP{3}model as an example to show that the \MPN model fails to satisfy the
\cref{c1}.

Denote the characteristic polynomial of the \MP{3}model as $p_3(\lambda)$, then it depends on
$E_{k}/E_0$, $k=1,2,3$, i.e., $p_3(\lambda) =p_3(E_1/E_0, E_2/E_0,E_3/E_0; \lambda)$. That all the
zeros of $p_3(\lambda)$ are real is a necessary condition for the hyperbolicity.
\Cref{fig:hyperbolicregion} plots the real region (the region that all the zeros of $p_3(\lambda)$
are real) of $p_3(\lambda)$ with some given $E_3 / E_0$. Clearly, the zeros of $p_3(\lambda)$ are
not always real; thus the \MP{3}model is not globally hyperbolic.

\begin{figure}[ht]
  \centering 
  \subfloat[$E_3/E_0$=0]{
  \includegraphics[width=0.33\textwidth,trim={1mm 1mm 14mm 7mm}, clip]{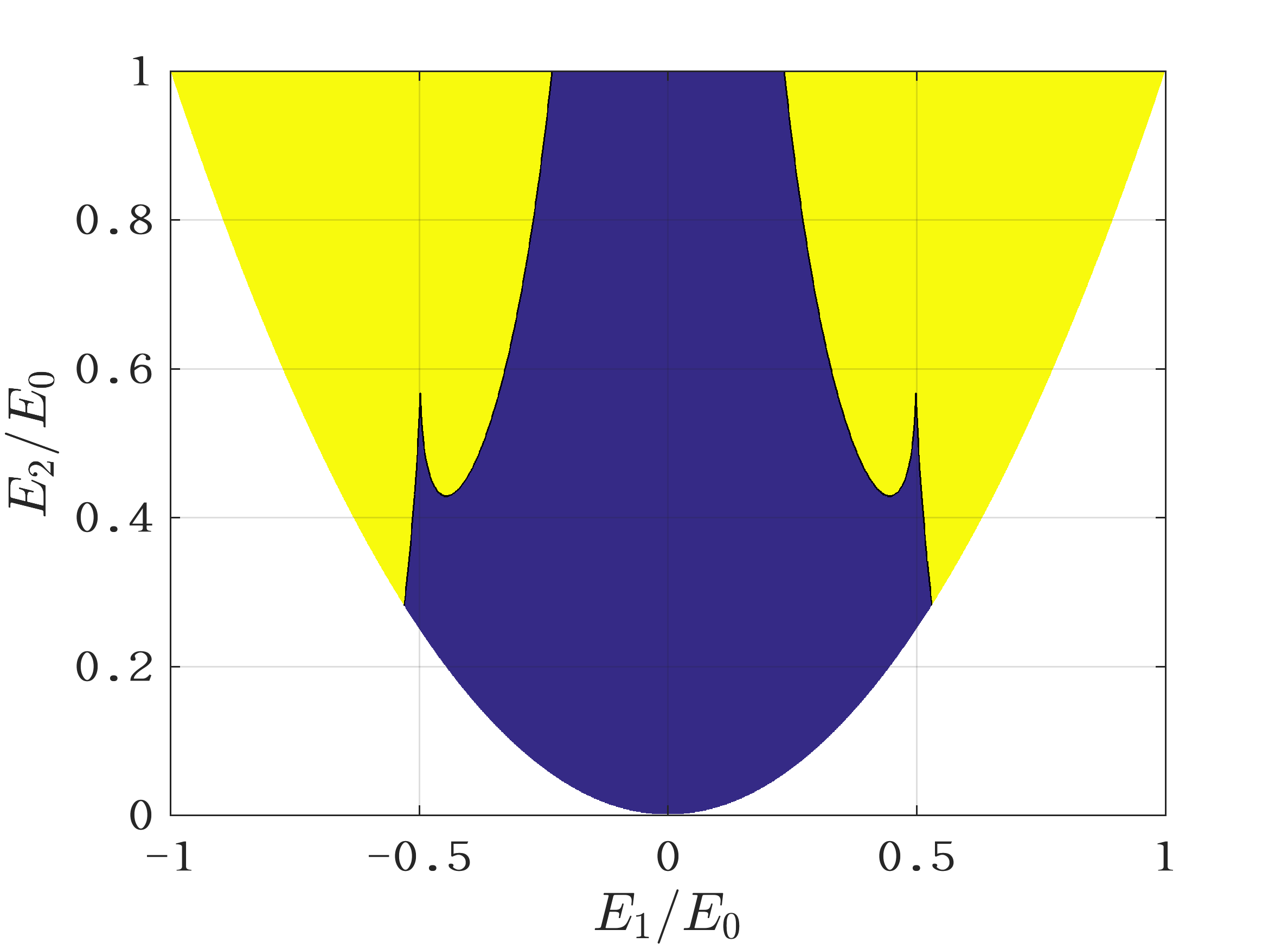}}
  \subfloat[$E_3/E_0$=1/5]{
  \includegraphics[width=0.33\textwidth,trim={1mm 1mm 14mm 7mm}, clip]{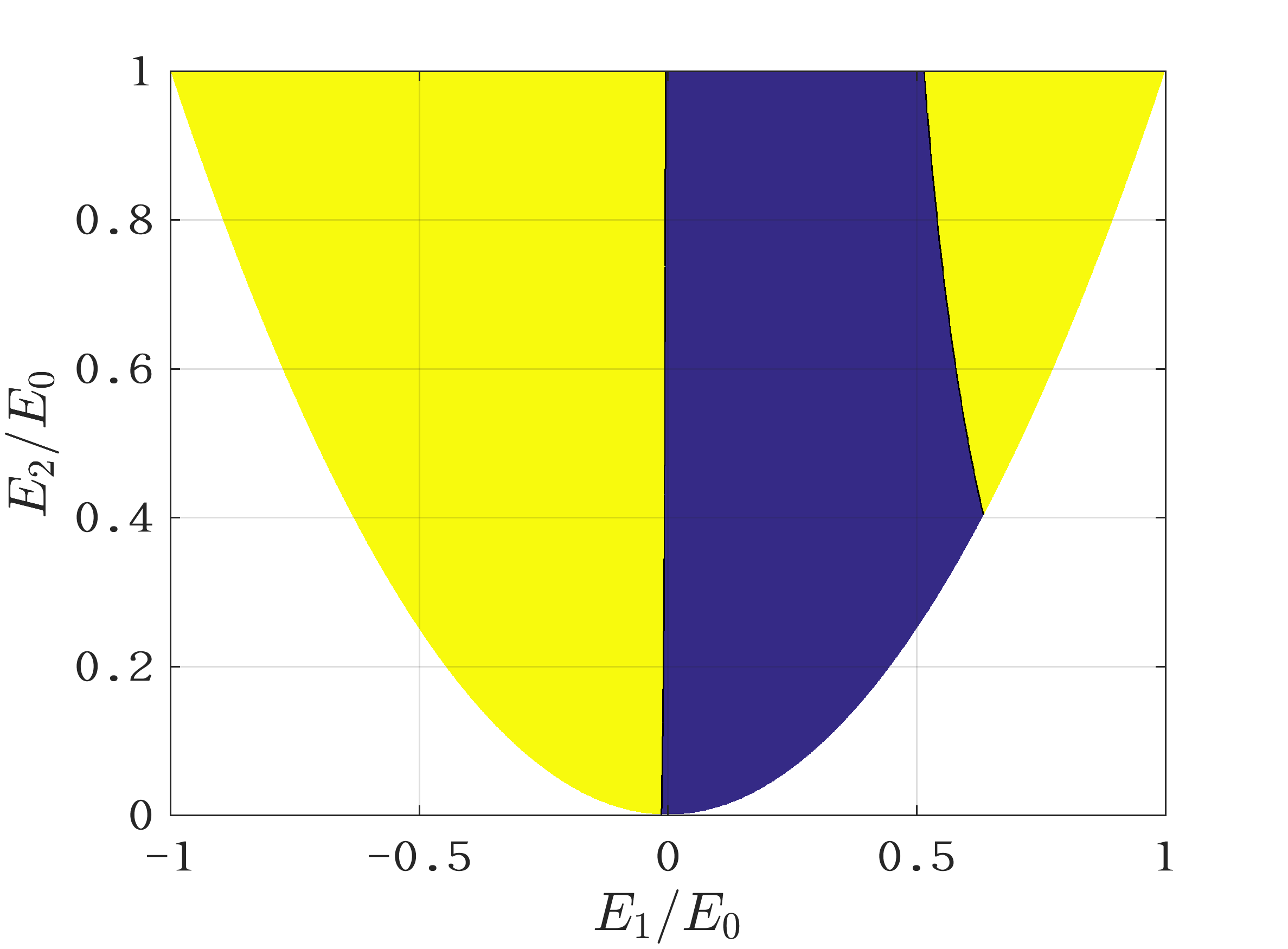}}
  \subfloat[$E_3/E_0$=1/2]{
  \includegraphics[width=0.33\textwidth,trim={1mm 1mm 14mm 7mm}, clip]{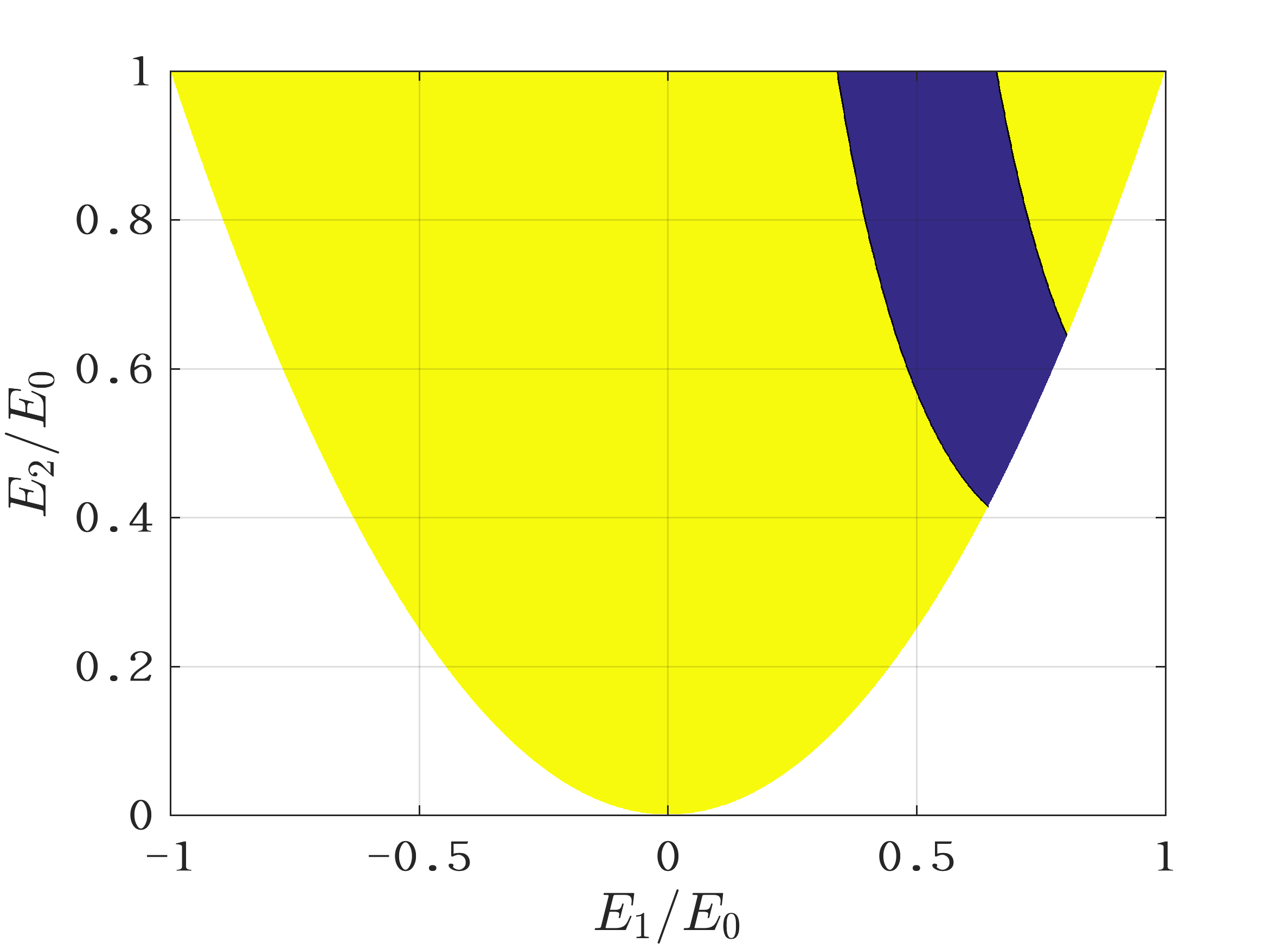}}
  \caption{Real region of the \MP{3}model with respect to $(E_1/E_0, E_2/E_0)$ with given $E_3 /
  E_0$. The blue region is the real region where all the zeros of $p_3(\lambda)$ are real, while the
  yellow region is non-real region where at least one zero of $p_3(\lambda)$ is not real.}
  \label{fig:hyperbolicregion}
\end{figure}

\subsubsection{Unphysical characteristic speed}\label{sec:speed}
We take the \MPtwo model as an example to show that the \MPN model fails to satisfy the \cref{c2}.
Denote the characteristic polynomial of the \MPtwo model by $p_2(\lambda)$.  Since the \MPtwo model
is strictly hyperbolic \cite{MPN}, all the zeros of $p_2(\lambda)$ are real and distinct. We denote
the zeros of $p_2(\lambda)$ by $\lambda_k$, $k=1,2,3$ with $\lambda_1<\lambda_2<\lambda_3$. Clearly,
the characteristic speeds $\lambda_k$ are determined by $E_1 / E_0$ and $E_2 / E_0$, i.e.,
$\lambda_k=\lambda_k(E_1 / E_0, E_2 / E_0)$.  \Cref{fig:characteristicspeedofMPtwo} presents the
profile of $\lambda_k$. One can observe that there is a region for $\lambda_1$ and $\lambda_3$ where
the characteristic speed does not lie in $[-c,c]$.

\begin{figure}[ht]
  \centering 
  \subfloat[$\lambda_1/c$]{
  \includegraphics[width=0.33\textwidth,trim={1mm 1mm 14mm 7mm}, clip]{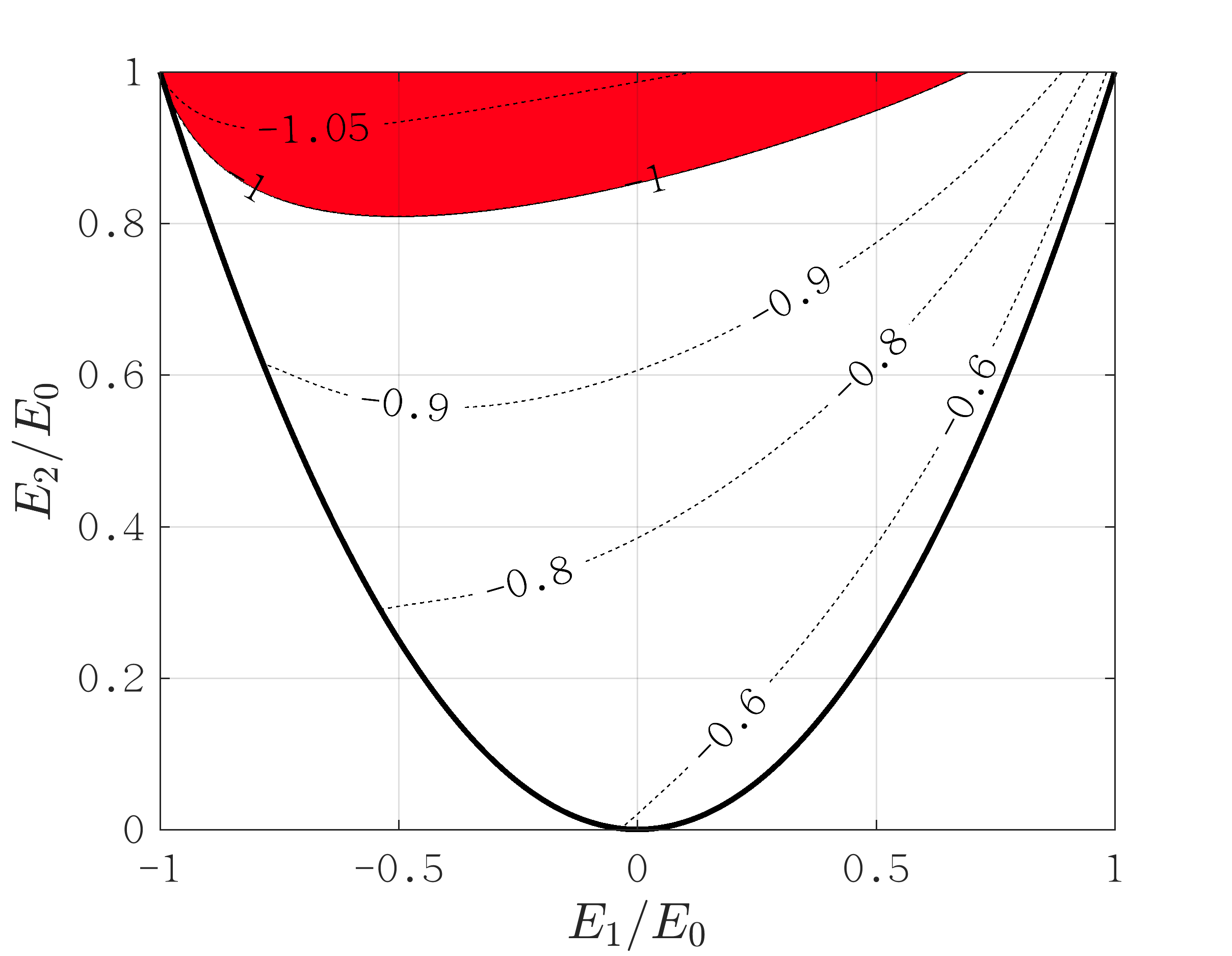}}
  \subfloat[$\lambda_2/c$]{
  \includegraphics[width=0.33\textwidth,trim={1mm 1mm 14mm 7mm}, clip]{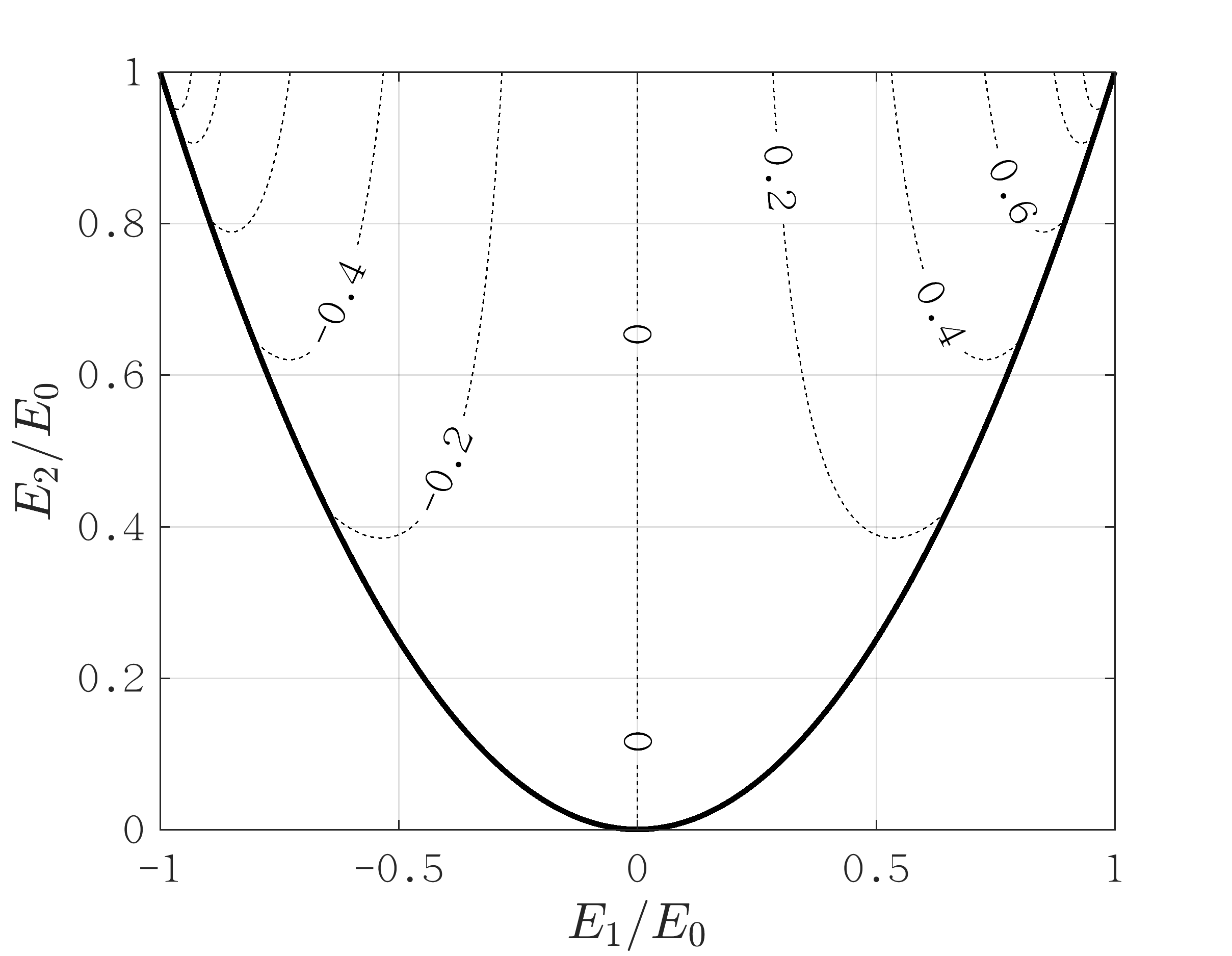}}
  \subfloat[$\lambda_3/c$]{
  \includegraphics[width=0.33\textwidth,trim={1mm 1mm 14mm 7mm}, clip]{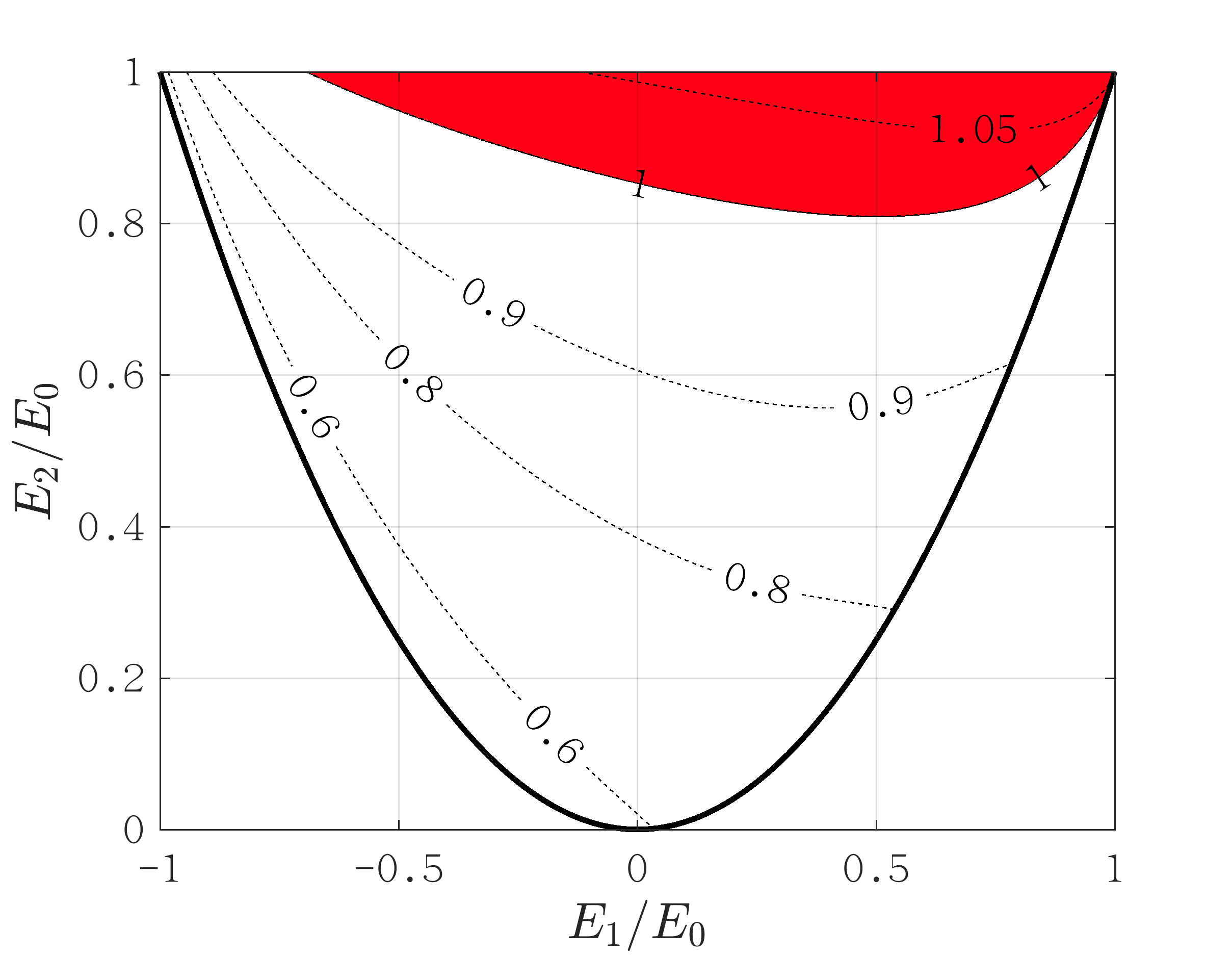}}
  \caption{Contour of characteristic speeds $\lambda_{k}$, $k=1,2,3$ of the \MPtwo model in the
  realizability domain.  In the red region, the characteristic speeds lie beyond $[-c,c]$.}
  \label{fig:characteristicspeedofMPtwo}
\end{figure}

%% file: hyper.tex
\section{Hyperbolic Regularization}\label{sec:regularization}\label{sec:ms}
In \cref{sec:defects}, it is pointed out that the \MPN model is not globally hyperbolic and its
characteristic speeds can be faster than the speed of light.
This motivates us to explore a regularization for the \MPN model to eliminate such defects.
As is discussed in the introduction, even though the hyperbolicity is a critical property for the 
reduced model, only a few works on the hyperbolic regularization are proposed. The candidate is 
the moment model reduction method in \cite{framework, Fan2015}, which provides a general framework
on deriving hyperbolic reduced models from kinetic equations.


On the other hand, one must follow some criteria in the regularization.
As is discussed in \cite{MPN}, when $N=1$, the \MPN model is the classical $M_1$ model, which
satisfies the \cref{c1,c2}. It is nature to require that the regularization does not change the
\MP{1}model, i.e., 
\begin{criterion}\label{c3}
    The regularization vanishes for the case $N=1$.
\end{criterion}
\begin{remark}
    Basically, we should require that the local linearization of the regularized reduced model 
    around the weight function is the same as that of the \MPN model. Precisely, we assume
    $\hat{I}=\weight + \epsilon g$, where $\epsilon$ is a small quantity, then we use this decomposition
    to linearize the \MPN model and the regularized reduced model by discarding high order terms. The
    resulting two models should be the same. This constraint is a critical condition for the
    regularization.  Otherwise, even $\hat{I}-\weight$ is quite small, the solution of the \MPN model
    and the regularized reduced model can be qualitatively different. Moreover, if one applies the
    Chapman-Enskog expansion \cite{Cowling} on the regularized reduced model with $N=2$ to derive the
    high-order Eddington approximation, the resulting approximation may be not correct
    (Ref. \cite{di2016quantum} presents an example for the quantum gas).
    Although the \cref{c3} is a necessary but not sufficient condition for such constraint, this 
    criterion is easy to check and is enough for verifying the regularization in this paper.
\end{remark}
In the derivation of the reduced model \eqref{eq:momentsystem}, the governing equation of $E_N$ is
the only unclosed equation. Thus it is the only equation one can modify in the regularization.
Precisely, we have the following criterion.
\begin{criterion}\label{c4}
    To regularize the reduced model \eqref{eq:momentsystem}, the governing equation of $E_k$,
    $k=0,\dots,N-1$ can not be changed.
\end{criterion}

In this section, we first show that the regularization framework in \cite{framework, Fan2015} to
regularize the \MPN model to be globally hyperbolic can not fulfill all the constraints of
\cref{c1,c2,c3,c4}. This investigation inspires us to propose a novel regularization by using the
special structure of the weight function and the RTE. Then the novel regularization is proposed, and
the regularized reduced model is well studied.

\subsection{Hyperbolic regularization framework}
We first try to apply the regularization framework in \cite{framework, Fan2015} on the \MPN model
and show that the resulting model dissatisfies the \cref{c1,c2,c3,c4}. It indicates that a globally
hyperbolic regularization on the \MPN model is not trivial.
\subsubsection{Reformulation of reduced model}
In order to apply the regularization framework in \cite{framework, Fan2015}, we introduce some
notations and reformulate the reduced model. Denote $\bbH$ by the space of all the admissible
specific intensity for the RTE and define the Hilbert space 
\begin{equation}
    \spaceH_N := \mathrm{span}\left\{ \Pl_i, i=0,\dots,N\right\}
\end{equation}
with the inner product
\begin{equation}
    \inner{\Phi}{\Psi}=\int_{-1}^1\Phi(\mu)\Psi(\mu)/\weight(\mu)\dd\mu.
\end{equation}
Then for each intensity $I\in\bbH$, the corresponding ansatz $\hat{I}$ \eqref{eq:ansatz} for the
\MPN model is in the space $\spaceH_N$. We define a projection from the space $\bbH$ to the space
$\spaceH_N$ as 
\begin{equation}\label{eq:ansatz2}
    \begin{aligned}
        \mP: I&\rightarrow \hat{I}=\sum_{i=0}^Nf_i\Pl_i,\quad
        f_i= \frac{\inner{I}{\Pl_i}}{\inner{\Pl_i}{\Pl_i}} = 
        \frac{\int_{-1}^1I\pl_i\dd\mu}{\mK_{i,i}},
    \end{aligned}
\end{equation}
where $\mK_{i,i}$ is defined in \eqref{eq:Gram-Schmidt}. Since the basis function is the product of
the weight function and the orthogonal polynomial, the projection $\mP$ is an orthogonal projection.

With the upper notations, the reduced model \eqref{eq:momentsystem} can be equivalently expressed as
\begin{equation}\label{eq:ms0}
    \inner{\mu^k\weight}{\frac{1}{c}\pd{\mP I}{t} +\mu\pd{\mP I}{z}} = \inner{\mu^k\weight}{\mS(\mP
    I)},\quad k=0,\dots,N.
\end{equation}
Since the projection $\mP$ is an orthogonal projection, we have
\begin{equation}\label{eq:ms}
    \inner{\mu^k\weight}{\frac{1}{c}\mP\pd{\mP I}{t} + \mP\mu\pd{\mP I}{z}} =
    \inner{\mu^k\weight}{\mP\mS(\mP I)},\quad k=0,\dots,N.
\end{equation}
For the simplicity of notations, we write the \MPN model \eqref{eq:ms} formally as
\begin{equation}\label{eq:P_mpn}
    \frac{1}{c}\mP\pd{\mP I}{t} + \mP\mu\pd{\mP I}{z} = \mP\mS(\mP I).
\end{equation}
Moreover, noticing that $\{\mu^k\weight, k=0,\dots,N\}$ and $\{\Pl_k, k=0,\dots,N\}$ are two sets of
basis function of $\spaceH$, we can also rewrite the reduced model \eqref{eq:ms} equivalently as
\begin{equation}\label{eq:ms_basis}
    \inner{\Pl_k}{\frac{1}{c}\pd{\mP I}{t} +\mu\pd{\mP I}{z}} = 
    \inner{\Pl_k}{\mS(\mP I)},\quad k=0,\dots,N.
\end{equation}

We note that all of \eqref{eq:ms0}, \eqref{eq:P_mpn} and \eqref{eq:ms_basis} are different forms of a
same model. 
The form \eqref{eq:P_mpn} is an abbreviation of \eqref{eq:ms} and is convenient to investigate the
hyperbolicity regularization and the differences between models, the form \eqref{eq:ms0} is
beneficial to analysis the reduced model in the form \eqref{eq:momentsystem}, and the form
\eqref{eq:ms_basis} is good to study the hyperbolicity of the model.

Next, we derive the reduced model in the quasi-linear form based on \eqref{eq:ms_basis}.
In the ansatz of the \MPN model \eqref{eq:ansatz}, there are many variables, including $f_i$,
$i=0,\dots,N$ and $\alpha$. Noticing \eqref{eq:f1}, the unknown variables are $f_0, \alpha, f_2,
\dots,f_N$. We arrange them into a vector and denote it by $\bw=(f_0,\alpha,f_2,\dots,f_N)^T$.
Since the basis $\Pl_i$, $i=0,\dots,N$ only depend on the variables $\alpha$, we let
\begin{equation*}
    \Pdl_i(\mu) := \od{\Pl_i(\mu)}{\alpha}.
\end{equation*}
Here we do not care about the concrete form of $\Pdl_i$, so its expression is omitted.
The derivation part can be written as
\begin{equation}\label{eq:derivative}
    \pd{\mP I}{s} = \sum_{i=0}^N\left( \pd{f_i}{s}\Pl_i+f_i\pd{\Pl_i}{s} \right)
    = \sum_{i=0}^N\left( \pd{f_i}{s}\Pl_i+f_i\pd{\alpha}{s}\Pdl_i \right), \quad s=t, z.
\end{equation}
Direct calculations yield
\begin{equation}\label{eq:p2d}
    \mP\pd{\mP I}{t} = (\boldsymbol{\Pl})^T\bD\pd{\bw}{t},\quad
    \mP\mu\pd{\mP I}{z} = (\boldsymbol{\Pl})^T\bB\pd{\bw}{z},
\end{equation}
where $\boldsymbol{\Pl}=(\Pl_0,\dots,\Pl_N)^T$ and the matrices
$\bD=(D_{i,j})_{i,j=0,\dots,N}\in\bbR^{(N+1)\times (N+1)}$ and
$\bB=(B_{i,j})_{i,j=0,\dots,N}\in\bbR^{(N+1)\times (N+1)}$ are defined by
\begin{align}
  \label{eq:matrixDelements_MP2}
    D_{i,j} &= \begin{cases}
        \delta_{ij},    & j\neq 1,\\
        \sum_{k=0}^N\inner{\Pl_i}{\Pdl_k}f_k / \mK_{i,i},   &   j=1,
    \end{cases}\\
  \label{eq:matrixBelements_MP2}
    B_{i,j} &= \begin{cases}
        \inner{\Pl_i}{\mu\Pl_j}/\mK_{i,i},   &   j\neq 1,\\
        \sum_{k=0}^N\inner{\Pl_i}{\mu\Pdl_k}f_k / \mK_{i,i},   &   j=1.
    \end{cases}
\end{align}
Let $\bS=\left(\inner{\Pl_i}{\mS(\mP I)}/\mK_{i,i}\right)_{i=0,\dots,N}$, then 
the reduced model \eqref{eq:ms_basis} can be reformulated as
\begin{equation}\label{eq:ms_quasi}
    \frac{1}{c}\bD\pd{\bw}{t}+\bB\pd{\bw}{z} = \bS.
\end{equation}

Clearly, \eqref{eq:ms_quasi} is the quasi-linear form of \eqref{eq:momentsystem}.
The variables vector $\bw$ can be uniquely determined by the ansatz \eqref{eq:ansatz2} and vice
versa. Thus we can treat $\bw$ as the vector representation of $\mP I$ in the space $\spaceH_N$. By
noticing \eqref{eq:p2d}, $\bD\pd{\cdot}{t}$ and $\bB\pd{\cdot}{z}$ are the matrix representation of
the operators
$\mP\pd{\cdot}{t}$ and $\mP\mu\pd{\cdot}{z}$ in \eqref{eq:P_mpn} in the space $\spaceH_N$ with the
basis $\Pl_k$.
We emphasis these relationships using the following formula
\begin{equation}\label{eq:relationship_dp}
    \bw\leftrightarrow \mP I,\quad
    \bD\pd{\cdot}{t}\leftrightarrow \mP\pd{\cdot}{t},\quad
    \bB\pd{\cdot}{z}\leftrightarrow \mP\mu\pd{\cdot}{z},
    \quad \text{in } \spaceH_N \text{ with basis } \Pl_k.
\end{equation}

\subsubsection{Hyperbolic model reduction framework}\label{sec:framework}
The hyperbolic model reduction framework in \cite{framework, Fan2015} suggests to add a more
projection between the operators $\mu\cdot$ and $\pd{\cdot}{z}$ to regularize the \MPN model to be
globally hyperbolic, and the resulting model is
\begin{equation}\label{eq:ms_framework}
    \frac{1}{c}\mP\pd{\mP I}{t} + \mP\mu\mP\pd{\mP I}{z} = \mP\mS(\mP I).
\end{equation}
Noticing \eqref{eq:p2d}, we can obtain that
\begin{equation*}
    \begin{aligned}
        \inner{\boldsymbol{\Pl}}{\mP\pd{\mP I}{t}} &= 
        \inner{\boldsymbol{\Pl}}{(\boldsymbol{\Pl})^T}\bD\pd{\bw}{t}, \\
        \inner{\boldsymbol{\Pl}}{\mu\mP\pd{\mP I}{z}} &= 
        \inner{\mu\boldsymbol{\Pl}}{(\boldsymbol{\Pl})^T\bD\pd{\bw}{z}}
        = \inner{\mu\boldsymbol{\Pl}}{(\boldsymbol{\Pl})^T}\bD\pd{\bw}{z}.
    \end{aligned}
\end{equation*}
Let $\bLambda=\mathrm{diag}(\mK_{0,0},\dots,\mK_{N,N})\in\bbR^{(N+1)\times (N+1)}$ and
$\bM=\bLambda^{-1} \inner{\mu\boldsymbol{\Pl}}{(\boldsymbol{\Pl})^T}$, then the regularized reduced
model corresponding to \eqref{eq:ms_framework} can be written as
\begin{equation}\label{eq:ms_framework_quasi}
    \frac{1}{c}\bD\pd{\bw}{t}+\bM\bD\pd{\bw}{z} = \bS.
\end{equation}
Similarly as \eqref{eq:relationship_dp}, $\bM$ is the matrix representation of the operator
$\mP\mu\cdot$ in the space $\spaceH_N$ with the basis $\Pl_k$, i.e.,
\begin{equation}
    \bM\leftrightarrow \mP\mu\cdot, \quad \text{in } \spaceH_N \text{ with basis } \Pl_k.
\end{equation}
Since the matrix $\inner{\mu\boldsymbol{\Pl}}{(\boldsymbol{\Pl})^T}$ is symmetric and the matrix
$\bLambda$ is symmetric positive definite, the matrix $\bM$ is real diagonalizable.
Hence, the model \eqref{eq:ms_framework_quasi} is globally hyperbolic, i.e., satisfying the
\cref{c1}.  We also claim that the model satisfies the \cref{c2}. Actually, since
$\Pl_i=\weight\pl_i$ and $\pl_i$ is orthogonal polynomials, one can obtain that the characteristic
polynomial of $\bM$ is $\pl_{N+1}$, whose zeros lie in $[-1,1]$. We leave more details in
\cref{sec:new_reg}.

The hyperbolic reduced model reduction framework in \cite{framework,Fan2015} indeed regularizes
the \MPN model to be hyperbolic, however, it fails to satisfy the \cref{c3}.
For exampe, $N=1$, we have $\mP I=f_0\weight$, then one only need to check whether 
\begin{equation}
    \inner{\weight\mu^k}{\mP\mu\mP\pd{\mP I}{z} - \mP\mu\pd{\mP I}{z}}=
    \int_{-1}^1\mu^{k+1}\left(\mP\pd{\mP I}{z} - \pd{\mP I}{z}\right)\dd\mu,\quad 
    k=0,1
\end{equation}
are both zero for any $\alpha\in(-1,1)$.
Unfortunately, direct calculations yield
\begin{equation}
    \int_{-1}^1\mu^2 \left(\mP\pd{\weight}{z}-\pd{\weight}{z}\right)\dd\mu
    =\frac{(4\alpha^2-12)\ln\left( \frac{1+\alpha}{1-\alpha} \right)+24\alpha}
    {3\alpha^4 (1-\alpha^2)^2}\pd{\alpha}{z}\neq0.
\end{equation}

\subsection{Novel hyperbolic regularization}\label{eq:regularization}
The failure of the existing regularization methods indicates that it is not trivial to regularize the
\MPN model to be globally hyperbolic in the constraints of the \cref{c1,c2,c3,c4}. In the following, 
we aim to construct a novel hyperbolic regularization for the \MPN model.

\subsubsection{Reformulation of the reduced model}
Note that the derivation of the weight function \eqref{eq:weight} with respect to $\alpha$ is
\begin{equation*}
    \pd{\weight}{\alpha} = \frac{-4\mu}{(1+\alpha\mu)^5}.
\end{equation*}
We introduce a new weight function 
\begin{equation}
    \weightn=\frac{1}{(1+\alpha\mu)^5}
\end{equation}
and define a series of monic orthogonal polynomials in the interval $[-1,1]$ with respect to the
weight function $\weightn(\mu)$ recursively as
\begin{align*}
  \pln_0(\mu) = 1, \quad \pln_j(\mu) = \mu^j - \sum_{k=0}^{j-1}
  \frac{\mKn_{j,k}}{\mKn_{k,k}} \pln_k(\mu), \quad j\geq 1,
  \quad
  \mKn_{j,k} = \int_{-1}^1\mu^j\pln_k(\mu)\weight(\mu)\dd\mu.
\end{align*}
Hereafter all the analogous notations with respect to the weight $\weightn$ will be marked by
$\tilde{\cdot}$.
Let $\Pln_i(\mu)=\weightn(\mu)\pln_i(\mu)$ and define the Hilbert space 
\begin{equation}
    \spaceHn_N := \mathrm{span}\left\{ \Pln_i, i=0,\dots,N\right\}
\end{equation}
with the inner product
\begin{equation}
    \innern{\Phi}{\Psi}=\int_{-1}^1\Phi(\mu)\Psi(\mu)/\weightn(\mu)\dd\mu.
\end{equation}
The two spaces $\spaceH_N$ and $\spaceHn_N$ have the following relationship.
\begin{lemma} \label{lem:relationship_spaces}
  For any $\Phi\in\spaceH_N$, we have
  \begin{equation}\label{eq:subset}
      \Phi\in\spaceHn_{N+1} ~\text{ and }~ \pd{\Phi}{\alpha}\in\spaceHn_{N+1}.
  \end{equation}
\end{lemma}
\begin{proof}
    We only need to check that \eqref{eq:subset} holds for $\Pl_i$, $i=0,\dots,N$.
    Note that
    \begin{equation*}
        \weight = (1+\alpha \mu)\weightn,\quad 
        \pd{\weight}{\alpha} = -4\mu\weightn.
    \end{equation*}
    We have 
    \begin{align}
        \label{eq:Pl2Pln}
        \Pl_i&=\weight\pl_i=(1+\alpha\mu)\pl_i\weightn\in\spaceHn_{N+1}, \quad i=0,\dots,N,\\
        \pd{\Pl_i}{\alpha} &=\pl_i\pd{\weight}{\alpha}+\pd{\pl_i}{\alpha}\weight
        =-4\mu\pl_i\weightn+(1+\alpha\mu)\pd{\pl_i}{\alpha}\weightn, \quad i=0,\dots,N.
    \end{align}
    Since $\pl_i$ is a monic polynomial of degree $i$ with its coefficient dependent on $\alpha$, 
    $\pd{\pl_i}{\alpha}$ is a polynomial whose degree is no more than $i-1$. Thus we have 
    $\pd{\Pl_i}{\alpha}\in\spaceHn_{N+1}$, $i=0,\dots,N$.
    This completes the proof.
\end{proof}
By noticing \eqref{eq:derivative}, \cref{lem:relationship_spaces} indicates that $\pd{\mP
I}{t}\in\spaceHn_{N+1}$. This is an important property of the space $\spaceHn_{N}$.
In the later of this section, we will show that this property is essential for the \cref{c4} for the
our regularization.
We define a projection from  the space $\bbH$ to the space $\spaceHn_N$ as
\begin{equation}\label{eq:ansatzn}
    \begin{aligned}
        \mPn: I&\rightarrow \sum_{i=0}^Ng_i\Pln_i,\quad
        g_i= \frac{\innern{I}{\Pln_i}}{\innern{\Pln_i}{\Pln_i}} = 
        \frac{\int_{-1}^1I\pln_i\dd\mu}{\mKn_{i,i}}.
    \end{aligned}
\end{equation}
Analogously as the projection $\mP$, the projection $\mPn$ is also an orthogonal projection.
We point out that the inner products of the spaces $\spaceH_N$ and $\spaceHn_N$ satisfy the
relationship
\begin{equation}\label{eq:relationship_inner}
    \inner{\weight\mu^k}{I}=\int_{-1}^1\mu^k I\dd\mu
    =\innern{\weightn\mu^k}{I},\quad k=0,\dots,N, ~ \forall I\in\bbH.
\end{equation}
This relationship is fundamental to study the reduced model \eqref{eq:ms0}. Actually, we can rewrite 
\eqref{eq:ms0} as
\begin{equation}\label{eq:ms_n}
    \innern{\mu^k\weightn}{\frac{1}{c}\pd{\mP I}{t} +\mu\pd{\mP I}{z}} = 
    \innern{\mu^k\weightn}{\mS(\mP I)},\quad k=0,\dots,N,
\end{equation}
which can further be written as the following form by noticing $\mPn$ is an orthogonal projection
\begin{equation}\label{eq:ms_n_P}
    \innern{\mu^k\weightn}{\frac{1}{c}\mPn\pd{\mP I}{t} +\mPn\mu\pd{\mP I}{z}} =
    \innern{\mu^k\weightn}{\mPn\mS(\mP I)},\quad k=0,\dots,N,
\end{equation}
We abbreviate \eqref{eq:ms_n_P} as
\begin{equation}\label{eq:Pn_mpn}
    \frac{1}{c}\mPn\pd{\mP I}{t}+\mPn\mu\pd{\mP I}{z} = \mPn\mS(\mP I).
\end{equation}
\begin{remark}\label{rk:Galerkin}
It is worth to point out again that the system \eqref{eq:P_mpn} and the system \eqref{eq:Pn_mpn} are
exactly same. This can be understood in the viewpoint of the Galerkin method. For the system
\eqref{eq:P_mpn}, both the trial and test function spaces are $\spaceH_N$ with the inner product
$\inner{\cdot}{\cdot}$; while for the system \eqref{eq:Pn_mpn}, the trial function space is
$\spaceH_N$ and the test function space is $\spaceHn_N$ with the inner product
$\innern{\cdot}{\cdot}$. The two systems are same due to the relationship
\eqref{eq:relationship_inner}. Both these two methods are natural and clear. The advantage of the
space $\spaceHn_N$ is its good property \cref{lem:relationship_spaces}.
\end{remark}

\subsubsection{Hyperbolic regularization}\label{sec:new_reg}
In \cref{sec:framework}, a direct application of the framework in \cite{framework,Fan2015} fails to 
regularize the \MPN model to be hyperbolic in the constraints of \cref{c1,c2,c3,c4}. Here we
restudy the system in the space $\spaceHn_N$. By adding a more projection between the operators
$\mu\cdot$ and $\pd{\cdot}{z}$ in \eqref{eq:Pn_mpn}, we obtain
\begin{equation}\label{eq:new_regularization}
    \frac{1}{c}\mPn\pd{\mP I}{t}+\mPn\mu\mPn\pd{\mP I}{z} = \mPn\mS(\mP I).
\end{equation}
Next we study the regularized system \eqref{eq:new_regularization} and check the \cref{c1,c2,c3,c4}
one by one.
Firstly, we present the relationship between the two set of functions $\Pl_k$ and $\Pln_k$ in the
following lemma.
\begin{lemma}\label{lem:relationship_polynomials}
    The functions $\Pl_k$ can be represented by the function $\Pln_k$ by the following relationships
    \begin{align} 
        \label{eq:recurrence_polynomial} \Pl_k &= \alpha\Pln_{k+1} + \beta_k\Pln_k,\quad k\in\bbN,\\
        \label{eq:d_pl} \pd{\Pl_k}{\alpha} &= -4\Pln_{k+1} + \gamma_k\Pln_k,
        \quad k\in\bbN,
    \end{align}
    where $\beta_k=\dfrac{\mK_{k,k}}{\mKn_{k,k}}$ and
    $\gamma_k=\dfrac{1}{\mKn_{k,k}}\pd{\mK_{k,k}}{\alpha}$.
\end{lemma}
\begin{proof}
    The orthogonality of $\Pl_k$ and $\Pln_k$ indicates that
    \begin{equation}\label{eq:orthogonality}
        \int_{-1}^1\Pl_k\mu^j\dd\mu=\int_{-1}^1\Pln_k\mu^j\dd\mu=0,\quad j<k.
    \end{equation}
    \Cref{lem:relationship_spaces} tells that $\Pl_k\in\spaceHn_{k+1}$, thus there exists a set of
    coefficients $c_j$ such that
    \begin{equation*}
        \Pl_k=\sum_{j=0}^{k+1}c_j\Pln_j.
    \end{equation*}
    Using \eqref{eq:orthogonality}, we can directly obtain $c_j=0$, $j=0,\dots,k-1$ and
    $c_k=\dfrac{\mK_{k,k}}{\mKn_{k,k}}$. Since both $\pl_k$ and $\pln_k$ are monic polynomials,
    \eqref{eq:Pl2Pln} indicates that $c_{k+1}=\alpha$. Hence \eqref{eq:recurrence_polynomial} holds.

    \Cref{eq:orthogonality} indicates
    \begin{equation*}
        \int_{-1}^1\pd{\Pl_k}{\alpha}\mu^j\dd\mu=0,\quad j<k.
    \end{equation*}
    Then using the same technique in the proof of \eqref{eq:recurrence_polynomial}, one can
    directly prove \eqref{eq:d_pl}.
\end{proof}
Calculations using \eqref{eq:recurrence_polynomial} and \eqref{eq:d_pl} yield
\begin{equation}\label{eq:mpn_time}
    \begin{aligned}
        \pd{\mP I}{t} &= \sum_{i=0}^N\left( \pd{f_i}{t}\Pl_i+f_i\pd{\Pl_i}{t} \right)
        = \sum_{i=0}^N\left( \pd{f_i}{t}(\alpha\Pln_{i+1}+\beta_i\Pln_i)
        +f_i\pd{\alpha}{t}(-4\Pln_{i+1}+\gamma_i\Pln_i) \right)\\
        &= \sum_{i=0}^N\left( \pd{f_i}{t}\beta_i+\pd{f_{i-1}}{t}\alpha+f_i\pd{\alpha}{t} \gamma_i 
        -4f_{i-1}\pd{\alpha}{t}\right)\Pln_i
        +\left(\pd{f_N}{t}\alpha-4f_N\pd{\alpha}{t}\right)\Pln_{N+1},
    \end{aligned}
\end{equation}
where $f_{-1}\equiv0$. Then we have
\begin{equation}\label{eq:mpn_time_project}
    \mPn\pd{\mP I}{t} = \sum_{i=0}^N \left( \pd{f_i}{t}\beta_i+\pd{f_{i-1}}{t}\alpha
    +f_i\pd{\alpha}{t}\gamma_i -4f_{i-1}\pd{\alpha}{t}\right)\Pln_i,
\end{equation}
and the time derivative part can be written as
\begin{equation}
    \mPn\pd{\mP I}{t} = (\boldsymbol{\Pln})^T\bDn\pd{\bw}{t},
\end{equation}
where $\boldsymbol{\Pln}=(\Pln_0,\dots,\Pln_N)^T$ and
\begin{equation} \label{eq:matrixDn}
  \bDn = \begin{pmatrix}
    \beta_0 & \gamma_0f_0                  & 0       & 0       & 0       & \cdots & 0           & 0      \\
    \alpha  & -4f_0                        & 0       & 0       & 0       & \cdots & 0           & 0      \\
    0       & \gamma_2f_2                  & \beta_2 & 0       & 0       & \cdots & 0           & 0      \\
    0       & \gamma_3f_3-4f_2             & \alpha  & \beta_3 & 0       & \cdots & 0           & 0      \\
    0       & \gamma_4f_4-4f_3             & 0       & \alpha  & \beta_4 & \cdots & 0           & 0      \\
    \vdots  & \vdots                       & \vdots  & \vdots  & \vdots  & \ddots & \vdots      & \vdots \\
    0       & \gamma_{N-1}f_{N-1}-4f_{N-2} & 0       & 0       & 0       & \cdots & \beta_{N-1} & 0      \\
    0       & \gamma_Nf_{N}-4f_{N-1}       & 0       & 0       & 0       & \cdots & \alpha      & \beta_N
  \end{pmatrix}
   \in \bbR^{(N+1)\times(N+1)}.
\end{equation}
Then we have
\begin{align*}
    \innern{\boldsymbol{\Pln}}{\mPn\pd{\mP I}{t}} &= 
    \innern{\boldsymbol{\Pln}}{(\boldsymbol{\Pln})^T \bDn\pd{\bw}{t}}=
    \innern{\boldsymbol{\Pln}}{(\boldsymbol{\Pln})^T} \bDn\pd{\bw}{t},\\
    \innern{\boldsymbol{\Pln}}{\mu\mPn\pd{\mP I}{z}} &=
    \innern{\mu\boldsymbol{\Pln}}{(\boldsymbol{\Pln})^T\bDn\pd{\bw}{z}}=
    \innern{\mu\boldsymbol{\Pln}}{(\boldsymbol{\Pln})^T}\bDn\pd{\bw}{z}.
\end{align*}
Let $\bLambdan=\mathrm{diag}(\mKn_{0,0},\dots,\mKn_{N,N})\in\bbR^{(N+1)\times(N+1)}$ and
$\bMn=\bLambdan^{-1} \innern{\mu\boldsymbol{\Pln}}{(\boldsymbol{\Pln})^T}$, then the regularized
reduecd model corresponding to \eqref{eq:ms_n_P} can be written as
\begin{equation}\label{eq:ms_n_quasi}
    \frac{1}{c}\bDn\pd{\bw}{t} + \bMn\bDn\pd{\bw}{z} = \bSn,
\end{equation}
where $\bSn=\left(\innern{\Pln_i}{\mS(\mP I)}/\mKn_{i,i}\right)_{i=0,\dots,N}$.
Similarly as \eqref{eq:relationship_dp}, $\bDn\pd{\cdot}{t}$ and $\bMn$ are the matrix
representation of the operators $\mPn\pd{\cdot}{t}$ and $\mPn\mu\cdot$ respectively in the space
$\spaceHn_N$ with the basis $\Pln_k$, i.e.,
\begin{equation}\label{eq:relationship_dpn}
    \bw\leftrightarrow \mP I,\quad \text{ in } \spaceH_N,\quad
    \bDn\pd{\cdot}{t}\leftrightarrow \mPn\pd{\cdot}{t},\quad
    \bMn \leftrightarrow \mPn\mu\cdot,
    \quad \text{in } \spaceHn_N \text{ with basis } \Pln_k.
\end{equation}

For the regularized reduced model \eqref{eq:ms_n_quasi}, we claim that it is not only globally
hyperbolic, but is also strictly hyperbolic and symmetric hyperbolic.
\begin{theorem}\label{thm:regularization}
    The regularized reduced model \eqref{eq:ms_n_quasi} is strictly symmetric hyperbolic for any
    $\bw$ with $\alpha\in(-1,1)$.
\end{theorem}
Before the proof of the \cref{thm:regularization}, we list some useful properties of the orthogonal
polynomials. Its proof can be found in textbook on the orthogonal polynomials, for example
\cite{szeg1939orthogonal,gautschi2004orthogonal}.
\begin{lemma}\label{lem:polynomial}
    Given an interval $[x_l,x_r]$ and a weight function $\omega$ such that $\omega(x)>0$ and
    $\omega\in L^1(x_l,x_r)$, let $\{p_n\}$ is a sequence of monic orthogonal polynomial with
    respect to the inner product $\langle g,h\rangle=\int_{x_l}^{x_r} \omega gh\dd x$, then we have
    \begin{enumerate}
        \item the orthogonal polynomials can be generated by the three term recurrence:
            \begin{equation*}
                p_{n+1}=(x-a_{n+1})p_n - b_{n+1}p_{n-1}, n\in\bbN,\quad p_{-1}=0, \quad p_0=1;
            \end{equation*}
        \item\label{itm:zeros} the polynomial $p_n$ has $n$ real and simple zeros, and they all lie
            in $[x_l,x_r]$;
        \item\label{itm:intersection} let $x_j$, $j=1,\dots,n+1$ be zeros of $p_{n+1}$, then there exists one and only one 
            zero of $p_n$ in $(x_j, x_{j+1})$, $j=1,\dots,n$;
        \item\label{itm:J} let the Jacobian matrix be $\bJ=(J_{i,j})_{i,j=0,\dots,N}$ with
          $J_{i,j}=\frac{\int_{x_l}^{x_r}\omega x p_ip_j\dd x}{\int_{x_l}^{x_r}\omega p_i^2\dd x}$, then the
            characteristic polynomial of $\bJ$ is $p_{n+1}$.
    \end{enumerate}
\end{lemma}
\begin{proof}[Proof of \cref{thm:regularization}]
    Since $\bLambdan$ is symmetric positive definite and
    $\bLambdan\bMn=\innern{\mu\boldsymbol{\Pln}}{(\boldsymbol{\Pln})^T}$ is symmetric, we multiply
    \eqref{eq:ms_n_quasi} by $\bDn^T\bLambdan$ and obtain
    \begin{equation*}
        \frac{1}{c} \bDn^T\bLambdan\bDn\pd{\bw}{t}
        + \bDn^T\bLambdan\bMn\bDn\pd{\bw}{z}=\bDn^T\bLambdan\bSn,
    \end{equation*}
    where $\bDn^T\bLambdan\bDn$ is symmetric positive definite and $\bDn^T\bLambdan\bMn\bDn$ is
    symmetric. Thus the system \eqref{eq:ms_n_quasi} is symmetric hyperbolic.

    \Cref{lem:polynomial} \cref{itm:J} indicates that the characteristic polynomial of $\bM$ is
    $\pln_{N+1}$, whose zeros are all real and simple due to \cref{lem:polynomial} \cref{itm:zeros},
    thus the system \eqref{eq:ms_n_quasi} is strictly hyperbolic. This completes the proof.
\end{proof}

In the proof of \cref{thm:regularization}, we show that the characteristic polynomial of $\bM$ is
$\pln_{N+1}$. Using \cref{lem:polynomial} \cref{itm:zeros}, one can directly obtain the following
corollary.
\begin{corollary}\label{cor:speed}
    For any $1\leq N\in\bbN$, all the characteristic speeds are not faster than the speed of light.
\end{corollary}
\Cref{thm:regularization} and \cref{cor:speed} prove that the
regularized reduced model
\eqref{eq:ms_n_quasi} fulfils the \cref{c1,c2}.
Now we check the \cref{c3,c4}.
\Cref{eq:mpn_time,eq:mpn_time_project} show that 
\begin{equation}\label{eq:difference_mpn}
    \pd{\mP I}{z} - \mPn\pd{\mP I}{z} = 
        \left(\pd{f_N}{z}\alpha-4f_N\pd{\alpha}{z}\right)\Pln_{N+1}.
\end{equation}
Let 
\begin{equation}
  \label{eq:mR}
    \mR_k=\innern{\weightn\mu^{k+1}}{ \left(\pd{f_N}{z}\alpha-4f_N\pd{\alpha}{z}\right)\Pln_{N+1}}
    =\begin{cases}
        0,  &   k<N,\\
        \mKn_{N+1,N+1}\left(\pd{f_N}{z}\alpha-4f_N\pd{\alpha}{z}\right), & k=N.
    \end{cases}
\end{equation}
Since the convection part of the \MPN model is
\begin{equation*}
    \frac{1}{c}\innern{\weightn\mu^k}{\pd{\mP I}{t}}
    + \innern{\weightn\mu^k}{\mu\pd{\mP I}{z}}
    = \frac{1}{c}\pd{E_k}{t} + \pd{E_{k+1}}{z},\quad k=0,\dots,N,
\end{equation*}
the convection part of the regularized reduced model is
\begin{equation*}
    \frac{1}{c}\innern{\weightn\mu^k}{\pd{\mP I}{t}}
    + \innern{\weightn\mu^k}{\mu\mPn\pd{\mP I}{z}}
    = \frac{1}{c}\pd{E_k}{t} + \pd{E_{k+1}}{z}
    -\mR_k.
\end{equation*}
Thus the regularized reduced model is
\begin{equation}\label{eq:regularized_ms}
    \frac{1}{c}\pd{E_k}{t} + \pd{E_{k+1}}{z} -\mR_k = C_k,
\end{equation}
where $C_k=\int_{-1}^1\mu^k\mS(\mP I)\dd\mu$.
Since $\mR_k=0$, $k=0,\dots,N-1$, the only difference of \eqref{eq:regularized_ms} from
\eqref{eq:momentsystem} is the last equation, i.e., the \cref{c4} holds.

Particularly, when $N=1$, \eqref{eq:f1} shows that $f_1=0$. Thus the regularized term $\mR_k=0$,
$k=0,1$, i.e., the regularization vanishes, which indicates the \cref{c3} holds.

\begin{remark}
    The hyperbolic model reduction framework in \cite{framework,Fan2015} suggests that adding
    a more projection between the operators $\mu\cdot$ and $\pd{\cdot}{z}$ is able to regularize the
    reduced model to be hyperbolic. However, in that framework, all the procedures are done in the same
    space, which limits the freedom on the resulting system. The regularization in this subsection
    studies the moments in two spaces, where one is for the ansatz and the other one is for the
    operator. As discussed in \cref{rk:Galerkin}, in the viewpoint of the Galerkin method, the trial
    and test function spaces are different.
    Hence, the key point of the regularization proposed in this subsection is the specific choice of
    the space $\spaceHn_N$. Notice that
    \begin{displaymath}
        \pd{\mP I}{z} = \sum_{i=0}^N\left( \pd{f_i}{z}\Pl_i+f_i\pd{\Pl_i}{z} \right).
    \end{displaymath}
    It would be a good choice to select a subspace of 
    \begin{displaymath}
        \mathrm{span}\left\{ \Pl_i, \pd{\Pl_i}{\alpha}, i=0,\dots,N \right\}.
    \end{displaymath}
    For the \MPN model, \cref{lem:relationship_spaces} shows that the upper space is
    $\spaceHn_{N+1}$, which is the motivation of the novel regularization.
\end{remark}
Till now, we proposed a novel hyperbolic regularization for the \MPN model, and the resulting model
\eqref{eq:regularized_ms} satisfies all the \cref{c1,c2,c3,c4}. Next, we investigate the
characteristic structure of the regularized reduced model.
\subsubsection{Characteristic structure}
Denote the eigenvalues of $\bMn$ by $\cs{N}{k}$, $k=0,\dots,N$, which are zeros of $\pln_{N+1}$ and
$\cs{N}{0}<\cs{N}{1}<\dots<\cs{N}{N}$. We have the following properties for the characteristic speed
$\cs{N}{k}$, $k=0,\dots,N$.
\begin{property}\label{pro:speed}
    The characteristic speeds $\cs{N}{k}=\cs{N}{k}(\alpha)$, $k=0,\dots,N$ satisfy the following
    properties:
    \begin{enumerate}
        \item $\cs{N}{k}$, $k=0,\dots,N$ is strictly decreasing with respect to $\alpha$, i.e.,
            $\pd{\cs{N}{k}(\alpha)}{\alpha} < 0$, $\alpha\in(-1,1)$;
        \item\label{itm:intersection_lam} $\cs{N}{k} < \cs{N-1}{k} < \cs{N}{k+1}$, $k=0,\dots,N-1$;
        \item $\cs{N}{0} < \frac{E_1}{E_0} < \cs{N}{N}$, for any $N\geq1$.
    \end{enumerate}
\end{property}
To prove the \cref{pro:speed}, we list the follow lemma, whose proof can be found in
\cite[Section 3]{ismail1989monotonicity}.
\begin{lemma} \label{lem:zerosofpolynomial}
    Let $\{p^{[\alpha]}_n(x)\}$ be orthogonal polynomials with respect to weight function
    $\omega^{[\alpha]}(x)$ on the interval $[x_l, x_r]$ and assume $\weight(x)$ is positive and has
    a continuous first derivative with respect to $\alpha$ for $x\in [x_l,x_r]$ with
    $\alpha\in(\alpha_l,\alpha_r)$. Furthermore assume that 
    \begin{equation*}
        \int_{x_l}^{x_r} x^j \pd{\omega^{[\alpha]}(x)}{\alpha}\dd x, 
        \quad j=1,2,\dots,2n-1,
    \end{equation*}
    converge uniformly for $\alpha$ in every compact subinterval of
    $(\alpha_l, \alpha_r)$.
    Then the zeros of $p^{[\alpha]}_n$ are strictly increasing (decreasing) functions of $\alpha\in
    (\alpha_l,\alpha_r)$, if $\pd{\ln(\omega^{[\alpha]})}{\alpha}$ is a strictly increasing
    (decreasing) function of $x\in[x_l,x_r]$.
\end{lemma}
\begin{proof}[Proof of \cref{pro:speed}]
    We prove the conclusion one by one.
    \begin{enumerate}
        \item Notice that
            \[
            \pd{\ln(\weightn(\mu))}{\alpha} = -5\pd{\ln(1+\alpha\mu)}{\alpha} 
            =\frac{-5\mu}{1+\alpha\mu}
            \]
            is a decreasing function of $\mu\in (-1,1)$ for any $\alpha\in(-1,1)$. 
            For any $\alpha\in[\alpha_l,\alpha_r]\subset(-1,1)$, the weight function $\weightn$ is 
            bounded, so
            \[
            \int_{\mu_l}^{\mu_r} \mu^j \pd{\weightn(\mu)}{\alpha}\dd \mu, 
            \quad j=1,2,\dots,2n-1
            \]
            converge uniformly.
            According to \cref{lem:zerosofpolynomial}, we have that $\pd{\cs{N}{k}}{\alpha}<0$
            for any $\alpha\in(-1,1)$ and $k=0,\dots,N$.
        \item
            It is a direct corollary of \cref{lem:polynomial} \cref{itm:intersection}.
        \item Direct calculation yields $\pln_1(\mu) = \mu-\frac{E_1}{E_0}$, thus $\lambda_{1,0} =
            \frac{E_1}{E_0}$. Using the \cref{lem:polynomial} \cref{itm:intersection_lam}, we complete the proof.
    \end{enumerate}
\end{proof}

Riemann problem is of fundamental importance for the hyperbolic reduced model.
The solution structure of the Riemann problem is instructional for studying the approximate Riemann
solver, which is the basis of the numerical methods using Godunov type schemes.
We study the characteristic structure of the regularized reduced model \eqref{eq:ms_n_quasi} and 
have the following conclusion.
\begin{theorem} \label{thm:characteristicstructure}
    The characteristic fields corresponding to $\cs{N}{0}$ and $\cs{N}{N}$ are genuinely nonlinear.
\end{theorem}
\begin{proof}
    Denote the eigenvectors of $(\bDn)^{-1}\bMn\bDn$ corresponding to the eigenvalue $\cs{N}{k}$ by
    $\bR_k=(R_{k,0}, \dots,R_{k,N})^T$ and let
    \[
        \Delta^{(N)}_k := \nabla_{\bw} \cs{N}{k} \cdot \bR_k, \quad
        k=0,\dots,N.
    \]
    We only check whether $\Delta^{(N)}_k$ with $k=0$ and $N$ change their sign.
    Since the eigenvalues only depend on $\alpha$, we have
    \[
        \nabla_{\bw} \cs{N}{k} = \left(0,\pd{\cs{N}{k}}{\alpha},0,0,\dots,0\right)^T.
    \]
    The \cref{pro:speed} shows that $\pd{\cs{N}{k}}{\alpha}<0$ for any $\alpha\in(-1,1)$, 
    so we only need to check whether $R_{k,1}$ with $k=0,N$ change their sign.

    Since $\bMn$ is the Jacobian matrix of the orthogonal polynomial $\pln_k$, we denote the
    eigenvectors of $\bMn$ with respect to $\cs{N}{k}$ by $\br_k=(r_{k,0},\dots,r_{k,N})^T$, then 
    direct calculation yields
    \begin{equation*}
        r_{k,0} = \frac{\mKn_{1,1}}{\mKn_{0,0}},\quad r_{k,1}=\cs{N}{k}-\frac{\mKn_{1,0}}{\mKn_{0,0}}.
    \end{equation*}
    Notice that $\bR_k=(\bDn)^{-1}\br_k$ and the matrix $\bDn$ \eqref{eq:matrixDn} is a block lower
    triangle matrix, whose top-left block is 
    $\begin{pmatrix}
        \beta_0 &   \gamma_0f_0\\
        \alpha  &   -4f_0
    \end{pmatrix}$.
    Thus we can obtain after some calculations
    \begin{equation*}
        R_{k,1} = \frac{1}{-4\beta_0f_0-\alpha \gamma_0f_0}(-\alpha r_{k,0}+\beta_0r_{k,1})
        =\frac{\beta_0}{\text{det}(\bD_{11})} \left(\lambda^{(N)}_k-\frac{E_1}{E_0}\right).
    \end{equation*}
    \Cref{pro:speed} shows that $R_{k,1}$ with $k=0,N$ do not change their sign. This completes the
    proof.
\end{proof}

\begin{remark}
In \cref{thm:characteristicstructure}, we only study the characteristic fields corresponding to
$\cs{N}{0}$ and $\cs{N}{N}$. For other fields, we conjecture that each of other characteristic
fields is \emph{neither genuinely nonlinear nor linearly degenerate}. In the proof of
\cref{thm:characteristicstructure}, we have shown that we only need to check the sign of
$\cs{N}{k}-\frac{E_1}{E_0}$. \cref{fig:lambdaN} presents the profile of $\cs{N}{k}$ with $N=3,4$ and 
$\frac{E_1}{E_0}$. One can observe that all the eigenvalues expect $\cs{N}{0}$ and $\cs{N}{N}$ have 
an intersection point with $\frac{E_1}{E_0}$, so these characteristic fields are neither 
genuinely nonlinear nor linearly degenerate. 
But a rigorous proof is not easy.  We numerically verify it for $N$ ranging from $2$ to $200$ with
the help of high performance computing and obtain a positive result.
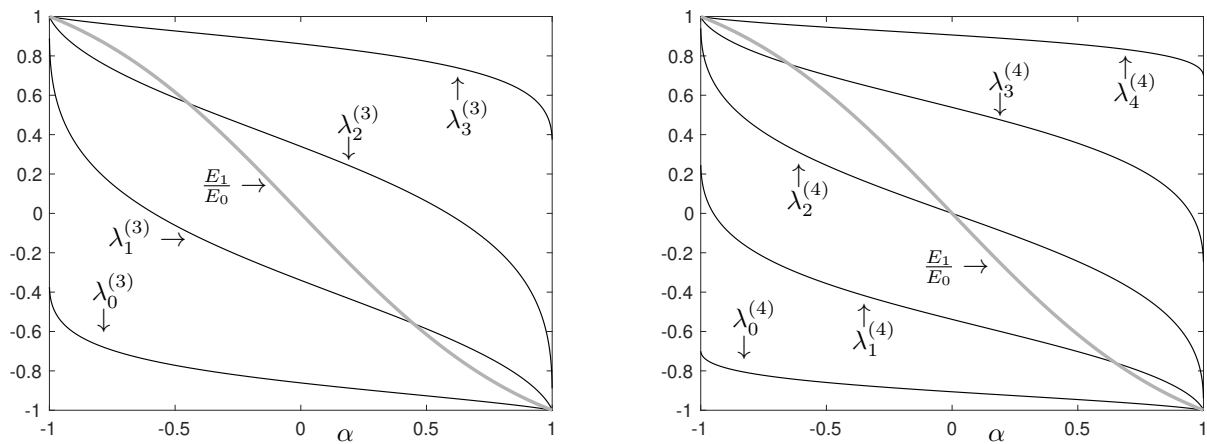
\begin{figure}[ht]
    \begin{overpic}[width=0.45\textwidth]{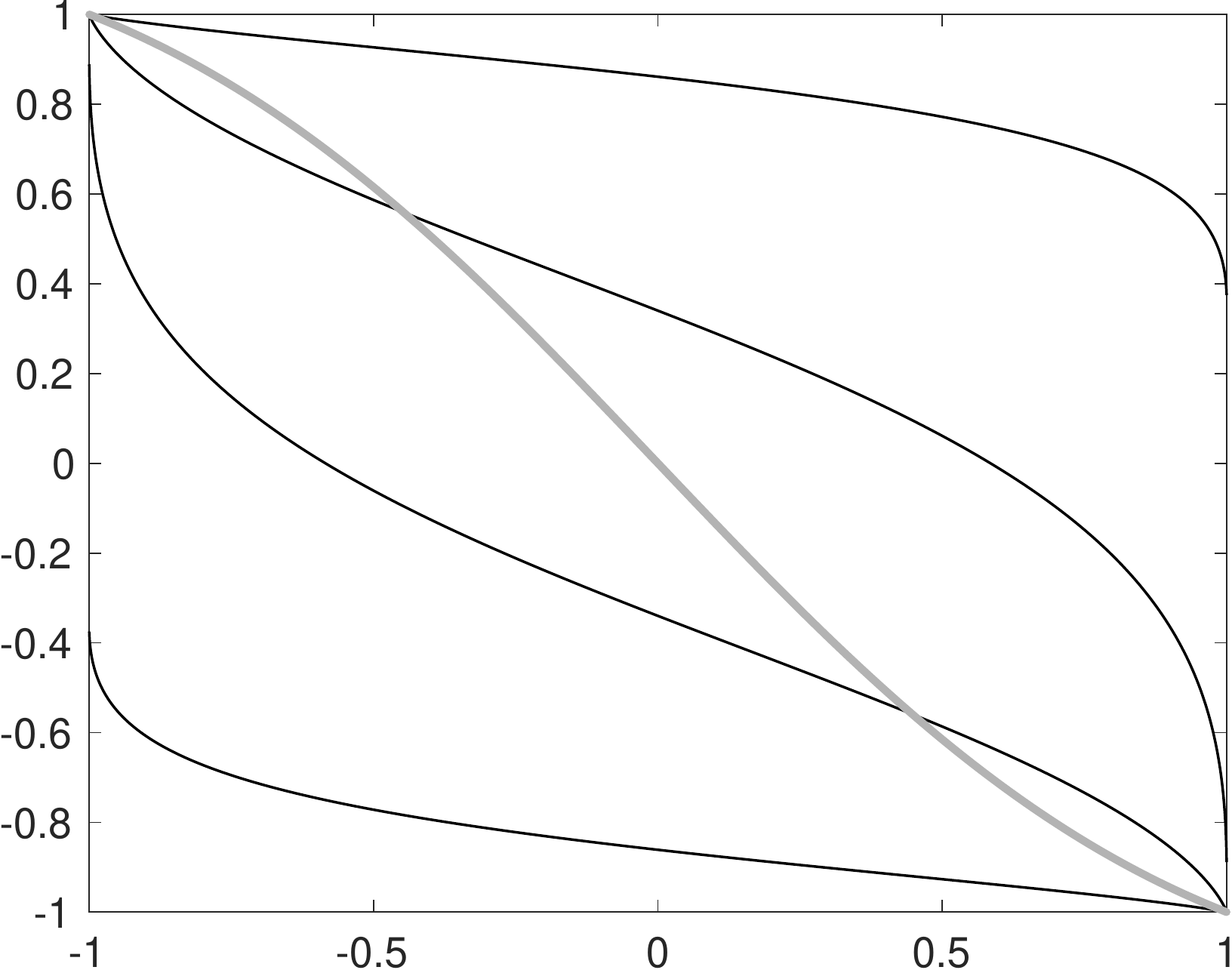}
        \put(15,25){$\cs{3}{0}$}
        \put(16,20){$\downarrow$}
        \put(18,35){$\cs{3}{1}\rightarrow$}
        \put(35,45){$\frac{E_1}{E_0}\rightarrow$}
        \put(60,56){$\cs{3}{2}$}
        \put(61,52){$\downarrow$}
        \put(80,57){$\cs{3}{3}$}
        \put(81,63){$\uparrow$}
        \put(60,-1){$\alpha$}
    \end{overpic}\qquad\qquad
    \begin{overpic}[width=0.45\textwidth]{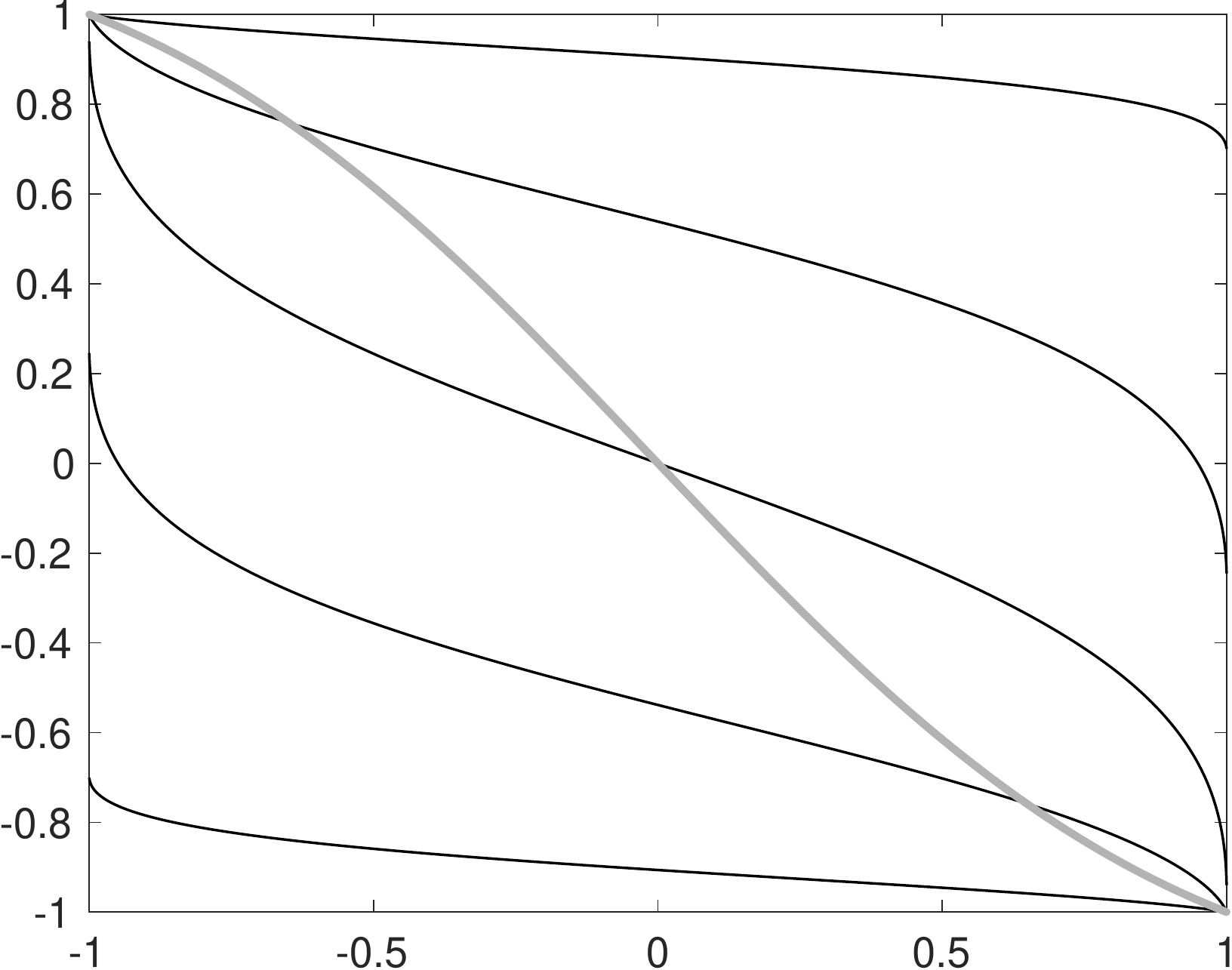}
        \put(13,20){$\cs{4}{0}$}
        \put(14,15){$\downarrow$}
        \put(35,16){$\cs{4}{1}$}
        \put(36,21){$\uparrow$}
        \put(23,42){$\cs{4}{2}$}
        \put(24,47){$\uparrow$}
        \put(48,30){$\frac{E_1}{E_0}\rightarrow$}
        \put(60,64){$\cs{4}{3}$}
        \put(61,60){$\downarrow$}
        \put(83,62){$\cs{4}{4}$}
        \put(84,67){$\uparrow$}
        \put(60,-1){$\alpha$}
    \end{overpic}
    \caption{\label{fig:lambdaN}Profile of $\cs{N}{k}$, $k=0,\dots,N$ with $N=3$ and $4$ and
    $\frac{E_1}{E_0}$ with respect to $\alpha\in(-1,1)$.}
\end{figure}
\end{remark}

%% file: num.tex
\section{Numerical Simulation}\label{sec:num}
The regularized reduced model \eqref{eq:regularized_ms} proposed in \cref{sec:ms} can be
reformulated as
\begin{equation}\label{eq:HMPN}
    \dfrac{1}{c} \pd{\bU}{t} + \pd{\bF(\bU)}{z} + {\bf R}(\bU)\pd{\bU}{z} = \bC,
\end{equation}
where $\bU=(E_0,E_1,\cdots,E_N)^T$, $\bF(\bU)=(E_1,E_2,\cdots,E_{N+1})^T$, ${\bf R}\pd{\bU}{z}
=(0,0,\cdots,-\mR_N)^T$, and $C_k=\int_{-1}^{1}\mu^k\mS(\mP I)\dd\mu$. 
Here $E_{N+1}$ is given by the moment closure of the \MPN model in \eqref{eq:MPNclosure} and $\mR_k$
is defined as $\eqref{eq:mR}$.
In this section, we investigate the numerical scheme for the
regularized reduced model
\eqref{eq:HMPN}, and perform numerical simulations on some typical examples to demonstrate
its numerical efficiency.

\subsection{Numerical scheme}\label{sec:scheme}
Denote the computational domain by $[z_l,z_r]$, which is discretized uniformly by $N_{\cell}$ cells.
The $i$-th mesh cell is $[z_{i-1/2}, z_{i+1/2}]$, $i=1,\dots,N_\cell$ with $z_{i+1/2}=z_l+i\Delta z$
and $\Delta z=\frac{z_r-z_l}{N_\cell}$. 
Let $\bU_i^n$ be the approximation of the solution $\bU$ on the $i$-th mesh cell at the $n$-th time
step $t_n$. 

To construct the numerical scheme for \eqref{eq:HMPN}, we split it into two parts: convection part
and the source part as
\begin{align}
    \label{eq:convection}
    \text{convection part:}\quad & \frac{1}{c}\pd{\bU}{t}+\pd{\bF(\bU)}{z}+\bR(\bU)\pd{\bU}{z}=0,   \\
    \label{eq:source}
    \text{source part:}\quad & \frac{1}{c}\pd{\bU}{t} = \bC.
\end{align}
Next we study the numerical scheme for the both parts.

\subsubsection{Source term}
The right hand side $\mS(I)$ denotes the actions by the background medium on the photons. 
Generally, it contains a scattering term, an absorption term, and an emission term, and has the form
\cite{Bru02,McClarren2008Semi} 
\begin{equation} \label{eq:sourceterm_initial}
    \mS(I) = \dfrac{1}{2}\sigma_s\int_{-1}^{1} I \dd\mu
    -\sigma_t I + \dfrac{1}{2}ac\sigma_a T^4 + \dfrac{s}{2},
\end{equation}
where $a$ is the radiation constant; $T(z,t)$ is the material temperature;
$\sigma_a(z,T)$, $\sigma_s(z,T)$ and $\sigma_t=\sigma_a+\sigma_s$ are the absorption, scattering,
and total opacity coefficients, respectively; and $s(z)$ is an isotropic external source.
The temperature is related to the internal energy $e$, whose evolution equation is
\begin{equation}\label{eq:internalenergy}
  \pd{e}{t} = \sigma_a \left(\int_{-1}^{1} I\dd\mu -acT^4\right).
\end{equation}
The relationship between $T$ and $e$ is problem dependent, and we will assign it in the numerical
examples when necessary.

Noticing the quartic term $a c\sigma_a T^4$ in $\mS(I)$ and the evolution equation of $e$
\eqref{eq:internalenergy}, we adopt the implicit Euler scheme on them as
\begin{equation}
    \frac{\bU_i^{n+1}-\bU_i^n}{c\Delta t} = \bC^{n+1}_i,\quad
    \dfrac{e_i^{n+1}-e_{i}^n}{\Delta t} = \sigma_{a,i}^{n+1} 
    \left( E_{0,i}^{n+1}-ac(T_{i}^{n+1})^4 \right).
\end{equation}
One can directly check that in the absence of any external source of radiation, i.e., $s = 0$, this
discretization satisfies the conservation of total energy as
\begin{equation}
  \frac{e^{n+1}_{i}-e^n_{i}}{\Delta t} + \frac{E^{n+1}_{0,i}-E^{n}_{0,i}}{c\Delta t} = 0.
\end{equation}

\subsubsection{Convection part}
The hyperbolic regularization in \cref{sec:regularization} modifies the governing equation of $E_N$
such that this equation may not be written into the conservative form. Thus, the classical Riemann
solvers for hyperbolic conservation laws can not be directly applied to solve
\eqref{eq:convection}. Here we adopt the DLM theory \cite{Maso} to deal with the non-conservative
terms. The key idea of the DLM theory is introducing a path $\Gamma(\tau;\cdot,\cdot)$,
$\tau\in[0,1]$ to connect two states $\bU^L$ and $\bU^R$ beside the Riemann problem such that
\begin{equation}\label{eq:path}
    \Gamma(0;\bU^L,\bU^R) = \bU^L,\quad
    \Gamma(1;\bU^L,\bU^R) = \bU^R.
\end{equation}
The path allows a generalization of the Rankine-Hugoniot condition to the non-conservative system
as
\begin{equation} \label{eq:generalizedRH}
  \bF(\bU^L) - \bF(\bU^R) + \int_0^1 [v_s{\bf I}-\bR(\Gamma(\tau;\bU^L,\bU^R))]
  \pd{\Gamma}{\tau}(\tau;\bU^L,\bU^R)\dd \tau = 0,
\end{equation}
if the two states $\bU^L$ and $\bU^R$ are connected by a shock with shock speed $v_s$.
Then the weak solution of the non-conservative system can be defined.
Readers can find more details of the constrained of path and the theory results in \cite{Maso}.
We then introduce the finite volume scheme in \cite{Rhebergen} to discretize the non-conservative
system \eqref{eq:convection}. This scheme can be treated as a non-conservative version of the HLL
scheme and has been successfully applied on the non-conservative models
\cite{Microflows1D,Qiao}.

Applying the finite volume scheme in \cite{Rhebergen} yields
\begin{equation}
  \label{eq:numericalscheme}
  \dfrac{\bU^{n+1}_i-\bU_i^n}{c\Delta t} + \dfrac{
  \hat{\bF}_{i+1/2}^{n}-\hat{\bF}_{i-1/2}^n}{\Delta z}
  +\dfrac{\hat{\bR}_{i+1/2}^{n-}-\hat{\bR}_{i-1/2}^{n+}}{\Delta z} = 0.
\end{equation} 
Here the flux $\hat{\bF}^n_{i+1/2}$ is the HLL numerical flux for the conservative term
$\pd{\bF(\bU)}{z}$, given by
\begin{equation} \label{eq:HLLflux_con}
  \hat{\bF}_{i+1/2}^{n} =
  \begin{cases}
      \bF(\bU_{i}^n), &  \lambda^L_{i+1/2}\geq 0,\\
      \dfrac{\lambda^R_{i+1/2}
      \bF(\bU_i^n)-\lambda^L_{i+1/2}\bF(\bU_{i+1}^n)+\lambda^L_{i+1/2}\lambda^R_{i+1/2}(\bU^n_{i+1}-\bU^n_i)}
      {\lambda^R_{i+1/2}-\lambda^L_{i+1/2}},  & \lambda^L_{i+1/2}<0<\lambda^R_{i+1/2},\\
      \bF(\bU_{i+1}^n),   &  \lambda^R_{i+1/2}\leq 0,\\
  \end{cases}
\end{equation}
where $\lambda^L_{i+1/2}$ and $\lambda^R_{i+1/2}$ are defined as 
\[
    \lambda^L_{i+1/2} = \min(\lambda^L_{i},\lambda^L_{i+1}),\quad 
    \lambda^R_{i+1/2} = \max(\lambda^R_{i},\lambda^R_{i+1}).
\]
Here $\lambda^{L}_{i}$ and $\lambda^R_{i}$ are the minimum and maximum characteristic speeds of
$\bU_{i}^{n}$, respectively. 
The flux $\hat{\bR}_{i+1/2}^{n\pm}$ is the special treatment of the finite volume scheme in
\cite{Rhebergen} for the non-conservative term $\bR(\bU)\pd{\bU}{z}$, given by
\begin{equation} \label{eq:HLLflux_nonconm}
  \hat{\bR}_{i+1/2}^{n-}=
  \begin{cases}
    0,  &  \lambda^L_{i+1/2}\geq 0,\\
    -\dfrac{\lambda^L_{i+1/2}\bg_{i+1/2}^n}
    {\lambda^R_{i+1/2}-\lambda^L_{i+1/2}},  & \lambda^L_{i+1/2}<0<\lambda^R_{i+1/2},\\
    \bg_{i+1/2}^n,   &  \lambda^R_{i+1/2}\leq 0,\\
  \end{cases}
\end{equation}
and 
\begin{equation} \label{eq:HLLflux_nonconp}
  \hat{\bR}_{i+1/2}^{n+}= 
  \begin{cases}
    -\bg^{n}_{i+1/2},    &  \lambda^L_{i+1/2}\geq 0,\\
    -\dfrac{\lambda^R_{i+1/2}\bg_{i+1/2}^n}
    {\lambda^R_{i+1/2}-\lambda^L_{i+1/2}},  & \lambda^L_{i+1/2}<0<\lambda^R_{i+1/2},\\
    0,  &  \lambda^R_{i+1/2}\leq 0,\\
  \end{cases}
\end{equation}
where  
\begin{equation} \label{eq:integrateR}
  \bg_{i+1/2}^n=\int_{0}^{1}{\bf R}(\Gamma(\tau;\bU_i^n,\bU_{i+1}^n))
  \pd{\Gamma}{\tau}(\tau;\bU_{i}^n,\bU_{i+1}^n)\dd \tau.
\end{equation}

Since the implicit scheme is adopted in the discretization of the source term, one can easily check
that the discretization is unconditionally stable. Thus the time step is constrained by the
convection term and complies with the CFL condition 
\begin{equation}
    \text{CFL} := \max_{i,k} |\cs{N}{k}(\bU_i^n)| \frac{\Delta t}{\Delta z} < 1.
\end{equation}
In all the tests in this paper, we set $\text{CFL} = 0.95$.
The \cref{cor:speed} indicates that the maximum speed is less than $1$, i.e.,
\begin{equation}
  \label{eq:maxspeedless1}
  \max_k |\lambda_k(\bU_i^n)| \leq 1.
\end{equation}
While for the \MPN model, as shown in \cref{sec:speed}, the inequality \eqref{eq:maxspeedless1} does
not hold, which limits the time step $\Delta t$.

\add{
\subsubsection{Path selection} \label{subsubsec:pathselection}
The remaining issue is the selection of the path $\Gamma(\tau;\cdot,\cdot)$ in \eqref{eq:path}. 
As is pointed out in \cite{abgrall2010comment}, for a given hyperbolic non-conservative system,
different path $\Gamma(\tau; \cdot, \cdot)$ would give different numerical results. Nevertheless,
many numerical tests have shown that for the non-conservative system reduced from kinetic equation, 
the selection of the path is not so critical \cite{Fan, Microflows1D, Qiao, Cai2018}. This motivates
us to study the reason behind.

Note that the smooth solution does not depend on the path, and the path only affects the way in
which the waves are damped and show no affects on the intrinsic constituent of the solution. For the
RTE, due to the existence of the source term, which may contain a scattering term, an absorption term,
and an emission term, its solution is usually smooth. 
Hence, the choice of the path is not essential if the solution approaches to the solution of the RTE
and is also smooth except for two cases, where the solution might not be smooth.
The first case is that 
subshocks appear in the solution. The choice of the path does make sense. However, in such case, 
the reduced model is inadequate to describe the physical process and the moment order $N$ has to
be increased. 
The other case is that the end time $\ctend$ is small. However, this solution has no
physical significance. The reduced model is designed to approximate the distribution, which is close
enough to the smooth functions, and thus it shows its ability to describe physics after the initial
layer. To sum up, the choice of the path is not essential in solving the reduced model
\eqref{eq:HMPN}.
}

\subsubsection{Boundary condition}

We adopt the method in \cite{MPN} to deal with the boundary condition. The ansatz of the \MPN model
provides an injective function from the moments $E_0,E_1,\cdots,E_N$
to the distribution function $\hat{I}$, which is stated in \cref{sec:MPNmodel},
thus we can construct the boundary condition of the reduced
model based on the boundary condition of the RTE.
Without loss of generality, we take the left boundary  as an example. 

On the left boundary ($z=z_l$), the specific intensity is given by 
\begin{equation}
  I^B(t,\mu) = \left\{
  \begin{aligned}
    &I(z=z_l,t,\mu),\quad &\mu<0,\\
    &I_{\text{out}}(t,\mu),\quad &\mu>0,
  \end{aligned}
  \right.
\end{equation}
where $I_{\text{out}}$ is the specific intensity outside of the domain, which depends on the
specific problem and the intensity inside the domain on the boundary $I(z=z_l,t,\mu)$. Here we list 
some of the common used boundary conditions and the choices of the intensity $I_{\text{out}}$, for
later use.

\begin{itemize}
  \item Infinite boundary condition:
\begin{equation}
  I_{\text{out}}(t,\mu) = I(z=z_l,t,\mu),\quad 0<\mu\leq 1.
\end{equation}
  \item Reflective boundary condition:
\begin{equation}
  I_{\text{out}}(t,\mu) = I(z=z_l,t,-\mu),\quad 0<\mu\leq 1.
\end{equation}
  \item Vacuum boundary condition:
\begin{equation}
  I_{\text{out}}(t,\mu) = 0, \quad 0<\mu\leq 1.
\end{equation}
  \item Inflow boundary condition:
\begin{equation}
  I_{\text{out}}(t,\mu) = I_{\text{inflow}}(t,\mu), \quad 0<\mu\leq 1,
\end{equation}
where $I_{\text{inflow}}$ is the specific intensity of the external inflow.
\end{itemize}

Furthermore, we replace the intensity $I(z=z_l,t,\mu)$ by the specific intensity constructed by the
moments in the cell near the left boundary. Precisely, 
\begin{equation}
  I(z=z_l,t,\mu) = \hat{I}\left(\bU(z=z_l,t);\mu\right),
\end{equation}
where $\bU(z=z_l,t)$ is the moments in the 1-st cell $[z_{1/2},z_{3/2}]=[z_l,z_l+\Delta z]$
at time $t$. Then one
can directly obtain the flux across the left boundary. Precisely, the $k$-th flux is given by 
\begin{equation}
  F^B_k = \int_{-1}^{1}I^B(t,\mu)\dd\mu = \int_{-1}^{0} \mu^{k+1} \hat{I}(\mu;\bU(z=z_l,t))\dd\mu 
  +\int_{0}^{1} \mu^{k+1}I_{\text{out}}(t,\mu)\dd\mu.
\end{equation}

\subsection{Numerical results}
In this subsection, \add{we perform simulations to validate the correctness of the numerical scheme
in \cref{sec:scheme}, and} also to demonstrate the utility and numerical efficiency of the reduced
model \eqref{eq:HMPN} by comparing with the \MPN model and the \PN model. Since the proposed model
is the hyperbolic version of the \MPN model, we call it the \emph{\HMPN model} hereafter.

\add{
\subsubsection{Verification on the path selection} \label{example:scheme}
As is discussed in \cref{subsubsec:pathselection}, the solution of the \HMPN model should not be
sensitive to the path selection $\Gamma(\tau;\cdot,\cdot)$ in dealing with the non-conservative
terms. Here we further verify it numerically.

Noticing the formula of the regularization term $\mR_k$ in \eqref{eq:mR} is depicted by $f_N$ and
$\alpha$, we define the path $\gamma(\tau;\bw^L,\bw^R)$ instead of $\Gamma(\tau;\bU^L,\bU^R)$, where
$\bw$ and $\bU$ are uniquely determined by each other. For the path $\gamma$, we select it in the
form
\begin{equation}\label{eq:pathchoices}
  \gamma(\tau; \bw^L, \bw^R) = \bw^L + \tau^k(\bw^R - \bw^L),\quad 0\leq \tau\leq 1.
\end{equation}
The following numerical experiment will show that the \HMPN model is not sensitive to the choice of 
the parameter $k$.
Moreover, the compound Simpson formula with $N_{\mathrm{intvl}}$ intervals is used to evaluate the
integral in \eqref{eq:integrateR}. 

We consider the Riemann problem with the initial value
\begin{equation}\label{eq:RiemannExample}
  I = \begin{cases}
    2I_0, & x\leq 0,\\ 
    I_0, & x> 0, 
  \end{cases}
\end{equation}
where $I_0$ is given by
\begin{equation}
  I_0(\mu) = a c \dfrac{w_0}{(1-0.08\mu-0.85\mu^2)^4},\quad -1\leq \mu\leq 1,
\end{equation}
with $w_0$ to be a constant such that $\frac{1}{a c}\int_{-1}^{1} I_0(\mu) \dd\mu = 1$.
The computational domain is set as $[-0.5,0.5]$ with $N_{\cell}=10000$ discretization cells, and the
end time is $\ctend=0.1$.

\begin{figure}[htb]
  \centering 
  \subfloat[$E_0$ for $N=2$]{
      \includegraphics[width=0.33\textwidth,height=0.16\textheight]{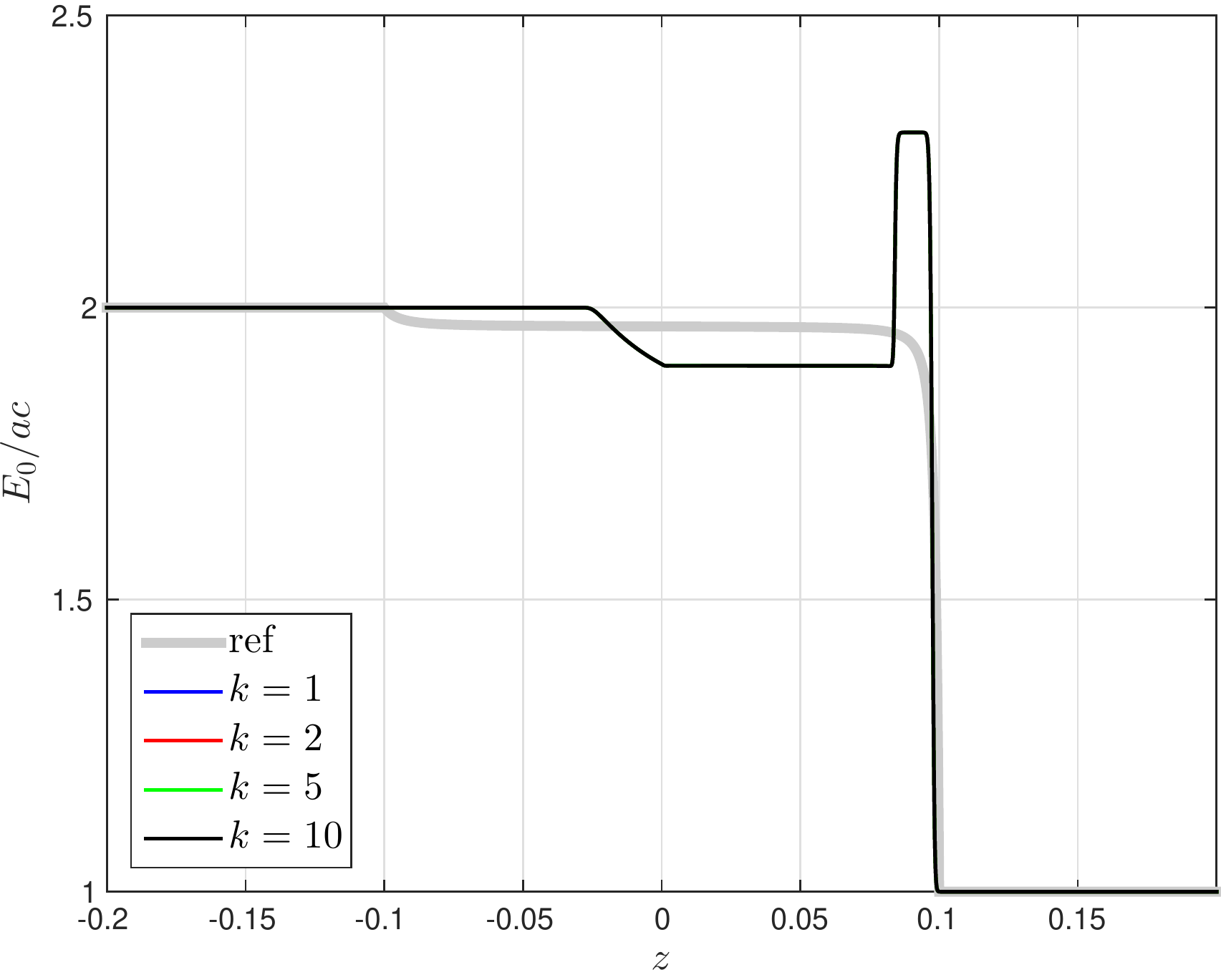}
  }
  \subfloat[$E_0$ for for $N=7$]{
      \includegraphics[width=0.33\textwidth,height=0.16\textheight]{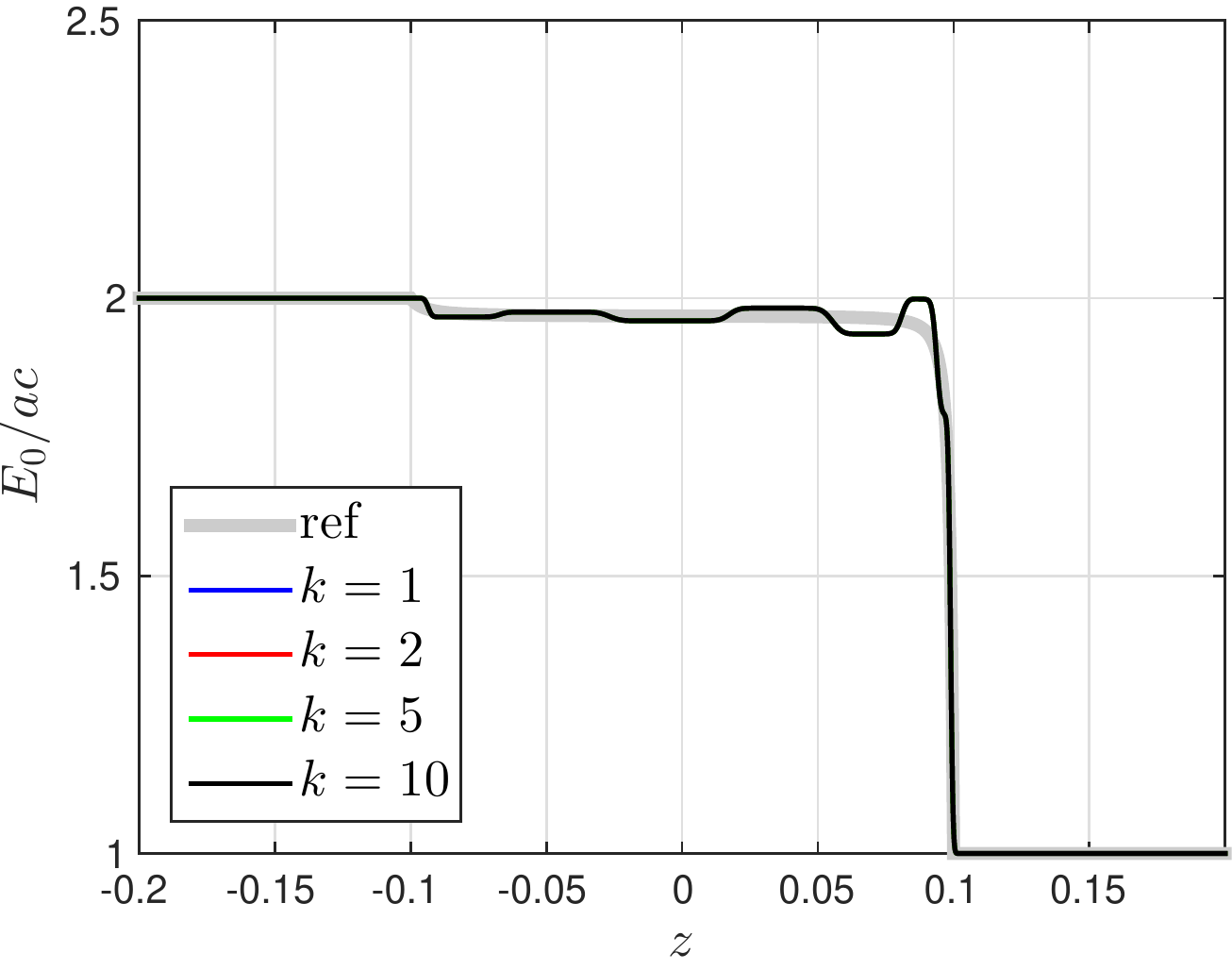}
  }
  \subfloat[$E_0$ for $N=12$]{
      \includegraphics[width=0.33\textwidth,height=0.16\textheight]{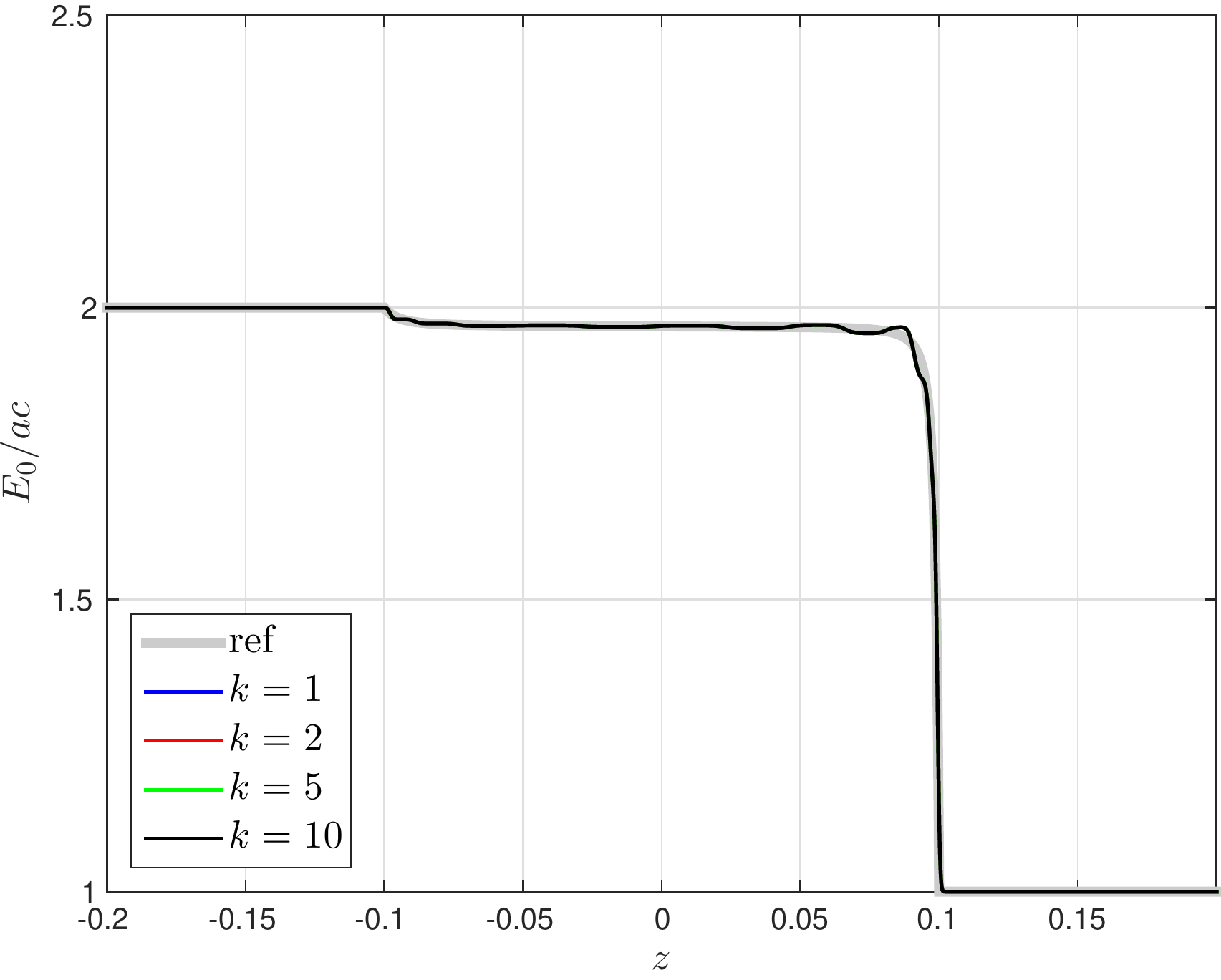}
  } \\
  \subfloat[Relative $L_2$ error of $E_0$ with respect to the solution of $k=1$]{
    \hspace{3mm}
    \begin{tabular}{|p{12mm}<{\centering}|p{0.2\textwidth}<{\centering}|p{0.31\textwidth}<{\centering}|p{0.3\textwidth}<{\centering}|}
      \hline 
      \diagbox{$k$}{$N$} & 2       & 7       & 12      \\ \hline
      2                      & 1.31e-8 & 2.32e-8 & 2.87e-6 \\ \hline
      5                      & 5.30e-7 & 1.03e-7 & 2.93e-6 \\ \hline
      10                     & 1.37e-5 & 1.55e-6 & 3.15e-6 \\ \hline
    \end{tabular}
  }
  \caption{\label{fig:pathselectiontest}
  Profiles of $E_0$ with $N_{\mathrm{intvl}}=10$ for different integral paths of the \HMPN model and
the relative $L_2$ error with respect to the solution of $k=1$.}
\end{figure}

\begin{figure}[htb]
  \centering 
  \subfloat[$E_0$ for $N=2$]{
      \includegraphics[width=0.33\textwidth,height=0.16\textheight]{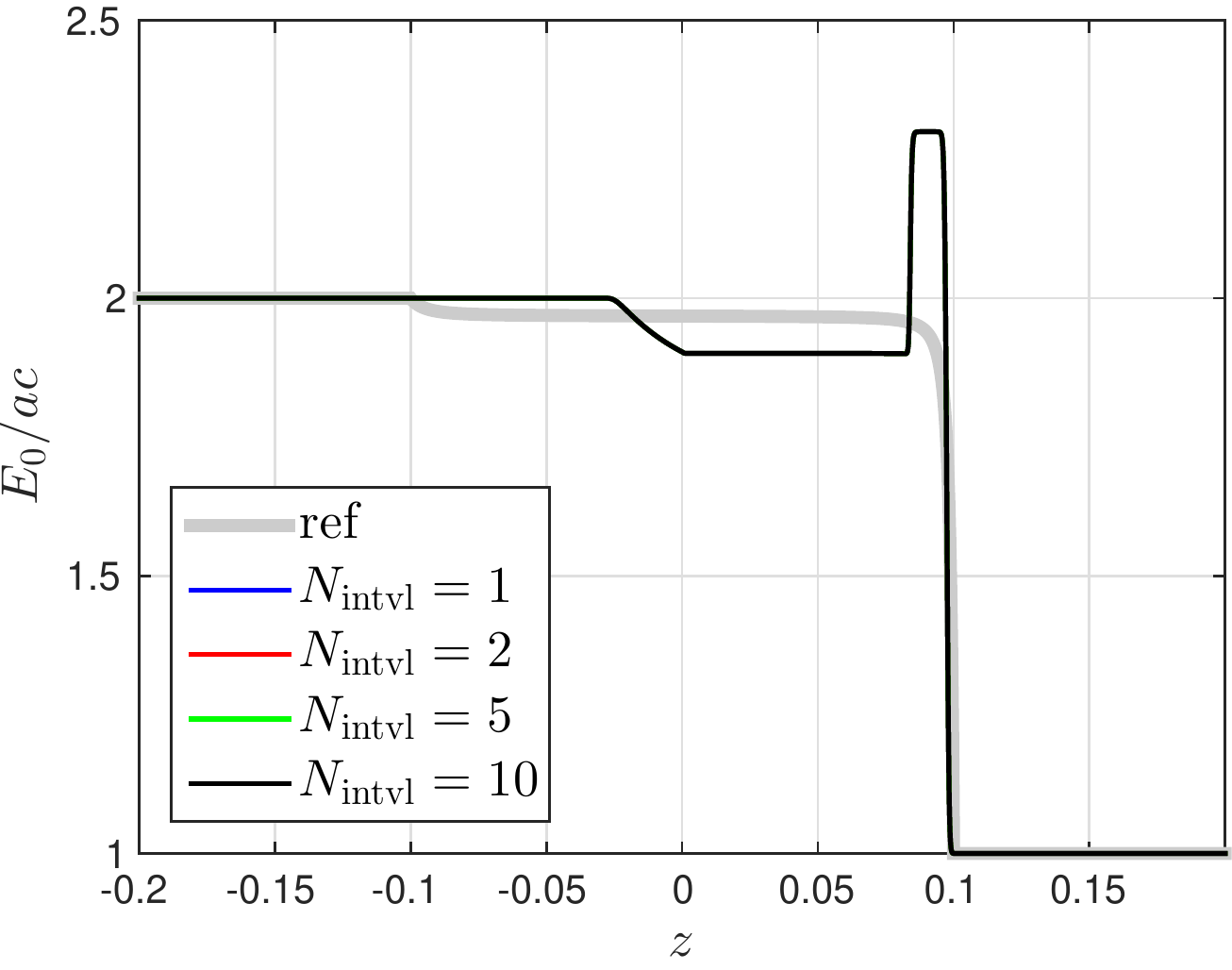}
  }
  \subfloat[$E_0$ for for $N=7$]{
      \includegraphics[width=0.33\textwidth,height=0.16\textheight]{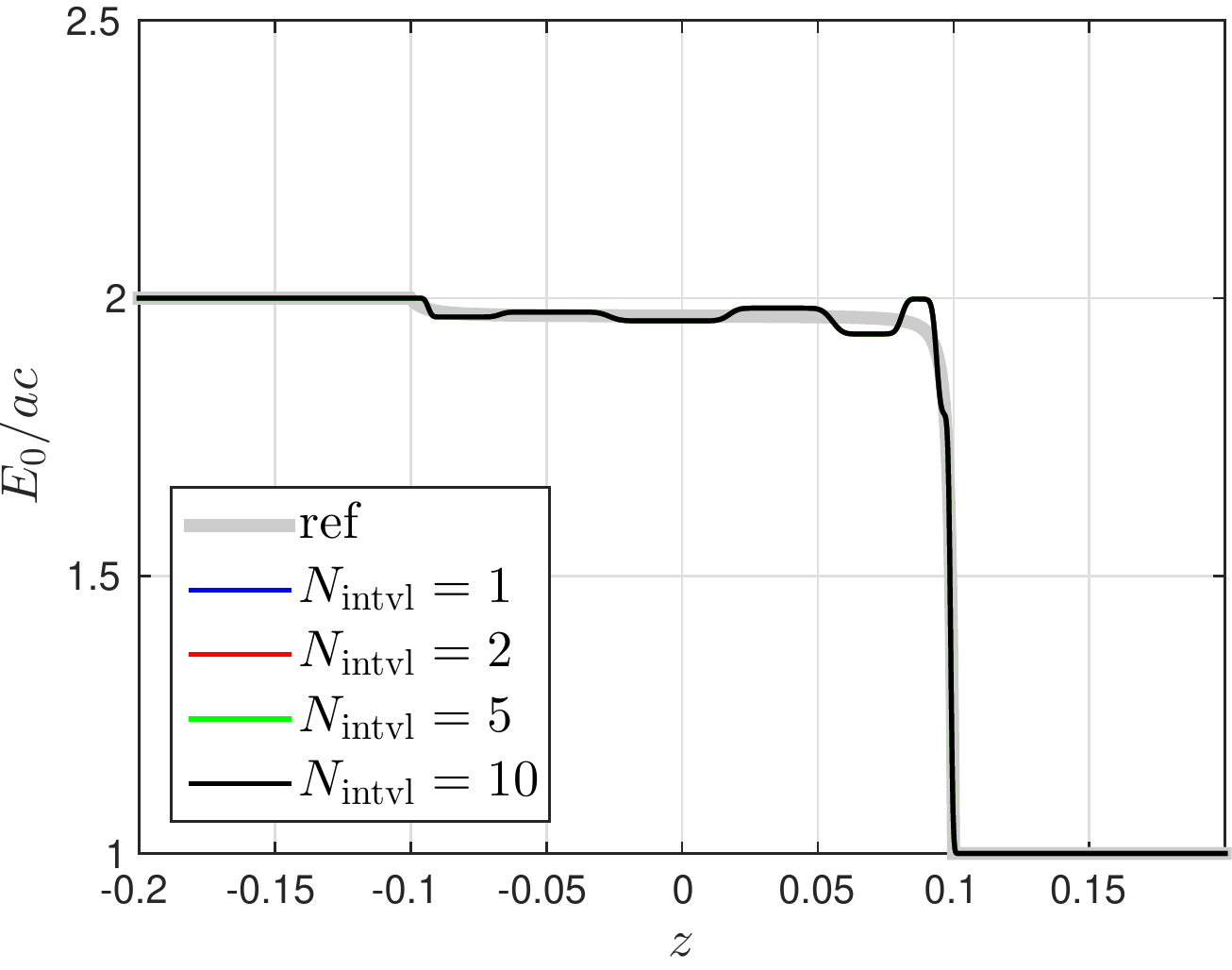}
  }
  \subfloat[$E_0$ for $N=12$]{
      \includegraphics[width=0.33\textwidth,height=0.16\textheight]{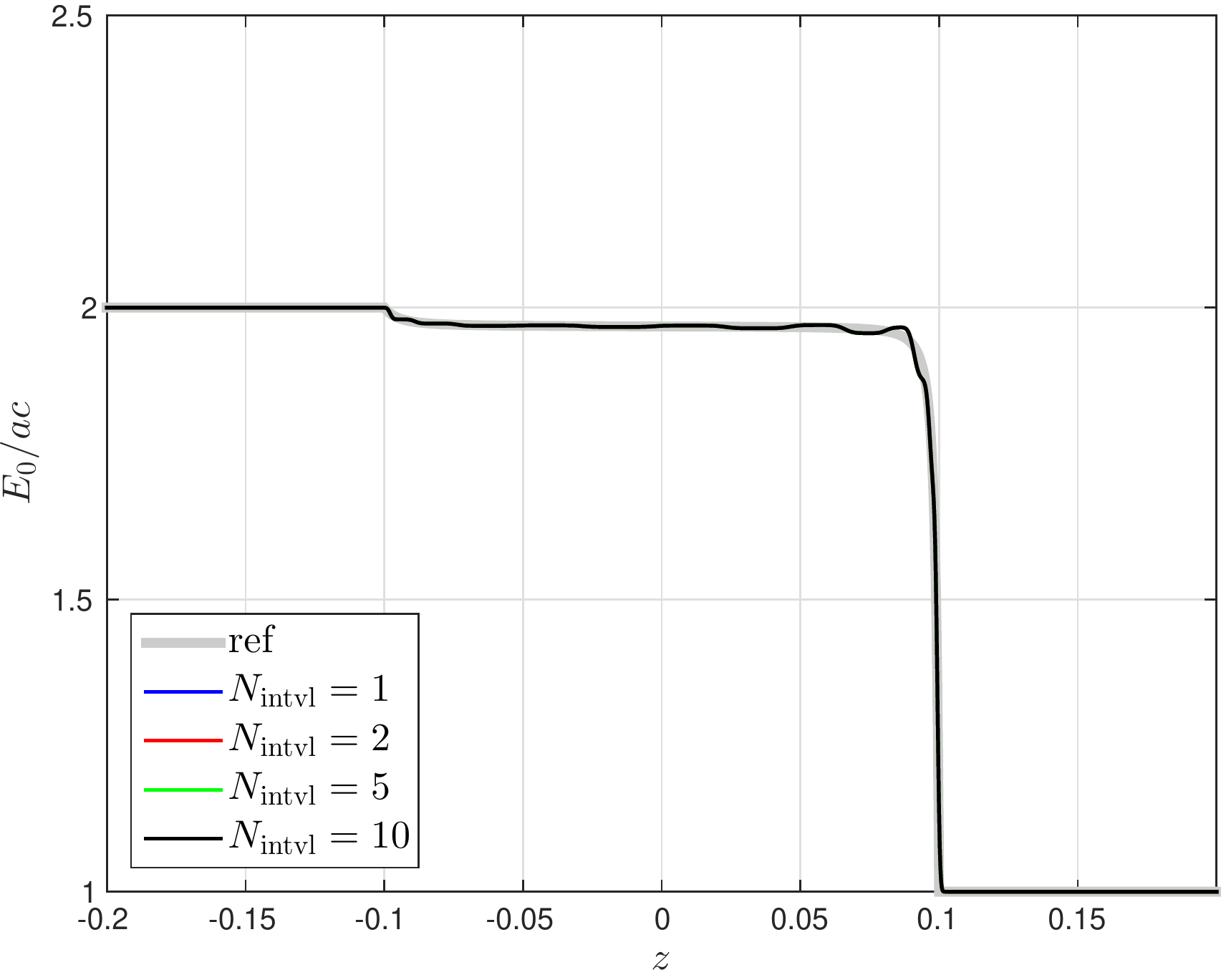}
  } \\
  \subfloat[Relative $L_2$ error of $E_0$ with respect to the solution of $N_{\mathrm{intvl}}=10$]{
    \hspace{6mm}
    \begin{tabular}{|p{1.5cm}<{\centering}|p{0.16\textwidth}<{\centering}|p{0.3\textwidth}<{\centering}|p{0.3\textwidth}<{\centering}|}
     \hline 
     \diagbox[width=18mm]{$N_{\mathrm{intvl}}$}{$N$} & 2       & 7       & 12      \\ \hline
     1                             & 3.73e-7 & 2.33e-8 & 2.85e-6 \\ \hline
     2                             & 7.39e-8 & 3.28e-8 & 2.74e-6 \\ \hline
     5                             & 1.54e-8 & 2.32e-6 & 2.87e-6 \\ \hline 
    \end{tabular}
  }
  \caption{\label{fig:integralformulatest} Profiles of $E_0$ with $k=1$ for different number of
  intervals in the compound Simpson formula of the \HMPN model and the relative $L_2$ error with
  respect to the solution of $N_{\mathrm{intvl}}=10$.}
\end{figure}

\Cref{fig:pathselectiontest} presents the profiles of $E_0$ of the \HMPN model with
$N_{\mathrm{intvl}}=10$ for different integral paths parameter $k=1$, $2$, $5$ and $10$, and the
relative $L_2$ error with respect to the solution of $k=1$. Here the relative $L_2$ error is defined
as 
\begin{equation}
  \mathrm{err} = \frac{\|E_0-E_0^{\mathrm{ref}}\|_{L^2([-0.5,
  0.5])}}{\|E_0^{\mathrm{ref}}\|_{L^2([-0.5, 0.5])}}.
\end{equation}
Clearly, it is hard to distinguish the solutions for different $k$ due to the negligible relative
error. This indicates that the path selection for this problem is indeed not sensitive.
Moreover, we would like to point out that the number of intervals $N_{\mathrm{intvl}}$ in the
compound Simpson formula is sufficient large. \Cref{fig:integralformulatest} gives the results for 
different $N_{\mathrm{intvl}}$ with $k=1$. The relative errors for different cases are all
negligible. Hence, in the following simulations, we always choose the path with $k=1$ and set
$N_{\mathrm{intvl}}=1$.

}
\subsubsection{Hyperbolicity validity} \label{example:hyperbolictest}
We have theoretically showed that the \MPN model with $N\geq 3$ is not globally hyperbolic
while the \HMPN model fixes the hyperbolic issue.  Now we construct examples to validate it. The first
example is a Riemann problem and the second one has a continuous initial value. In order to avoid the
disturbance of the interaction between photons and background, we consider the case that the right
hand side vanishes, i.e., $\mS=0$, then the RTE can be written as
\begin{equation} \label{eq:hyperbolicitytest}
  \dfrac{1}{c} \pd{I}{t} + \mu \pd{I}{z} = 0.
\end{equation} 

\paragraph{A Riemann problem}

\begin{figure}[htb]
  \centering 
  \subfloat[$E_0$ at $\ctend=0.0287$ for $N=4$]{
      \includegraphics[width=0.33\textwidth,height=0.16\textheight]{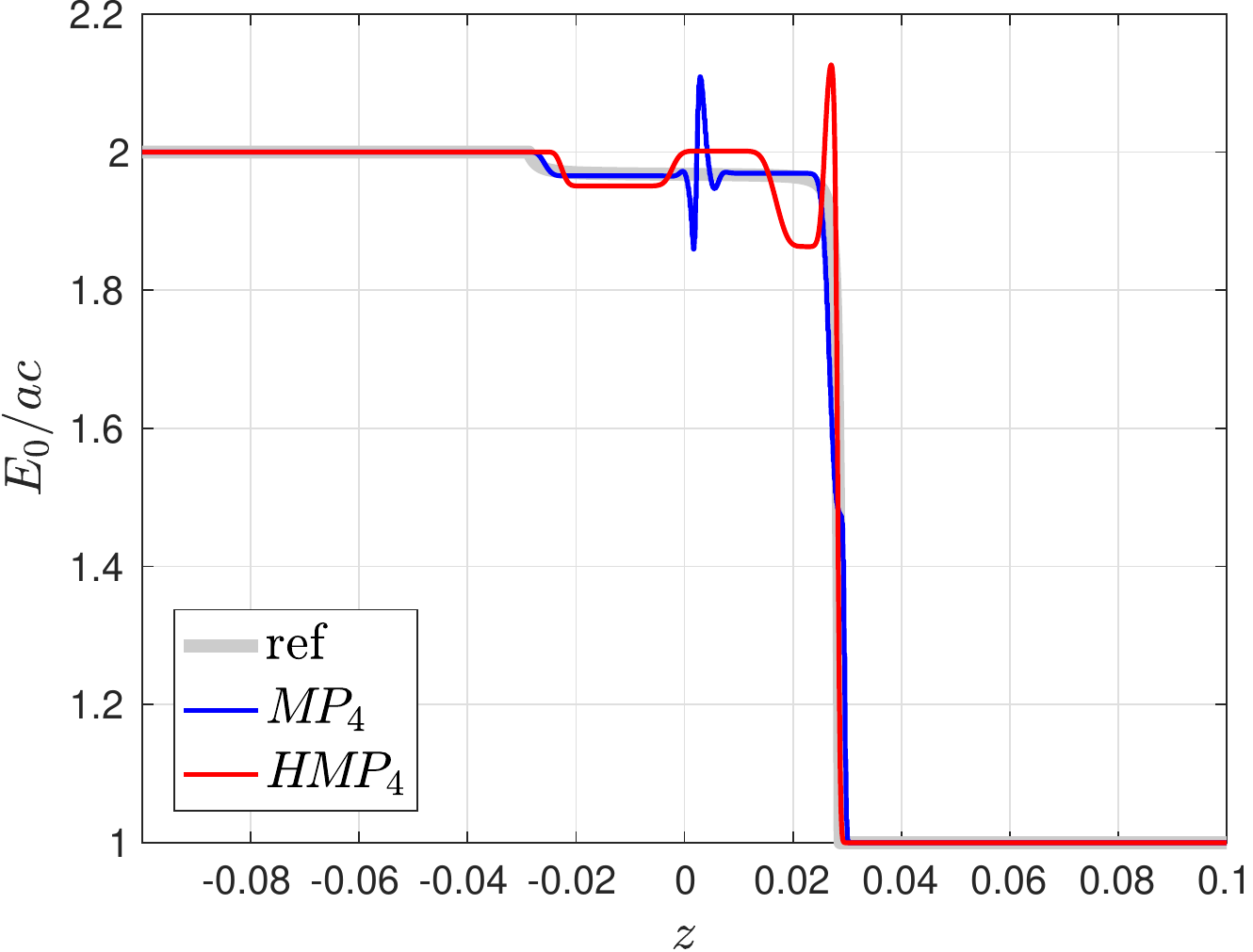}
  }
  \subfloat[$E_0$ at $\ctend=0.051$ for $N=6$]{
      \includegraphics[width=0.33\textwidth,height=0.16\textheight]{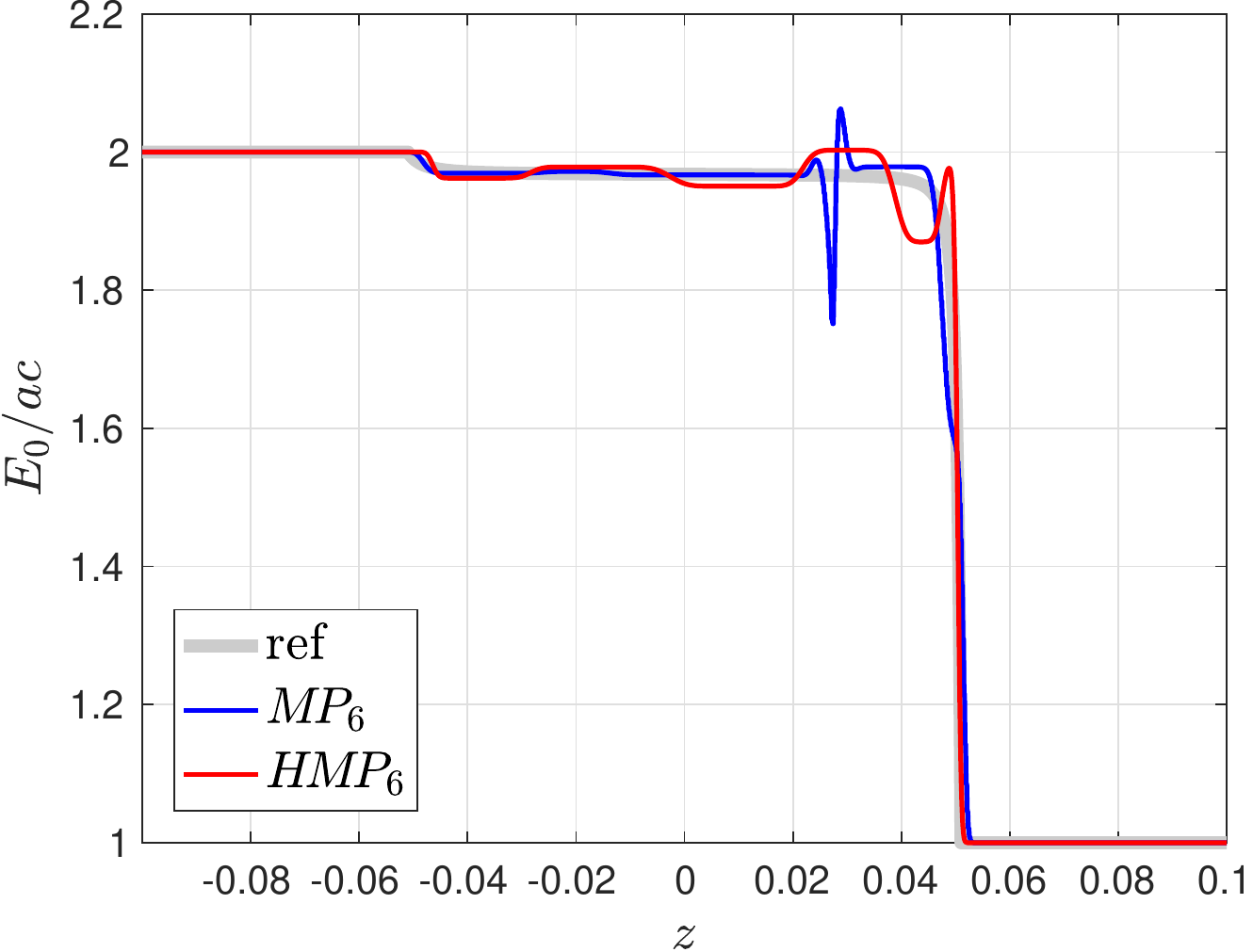}
  }
  \subfloat[$E_0$ at $\ctend=0.169$ for $N=8$]{
      \includegraphics[width=0.33\textwidth,height=0.16\textheight]{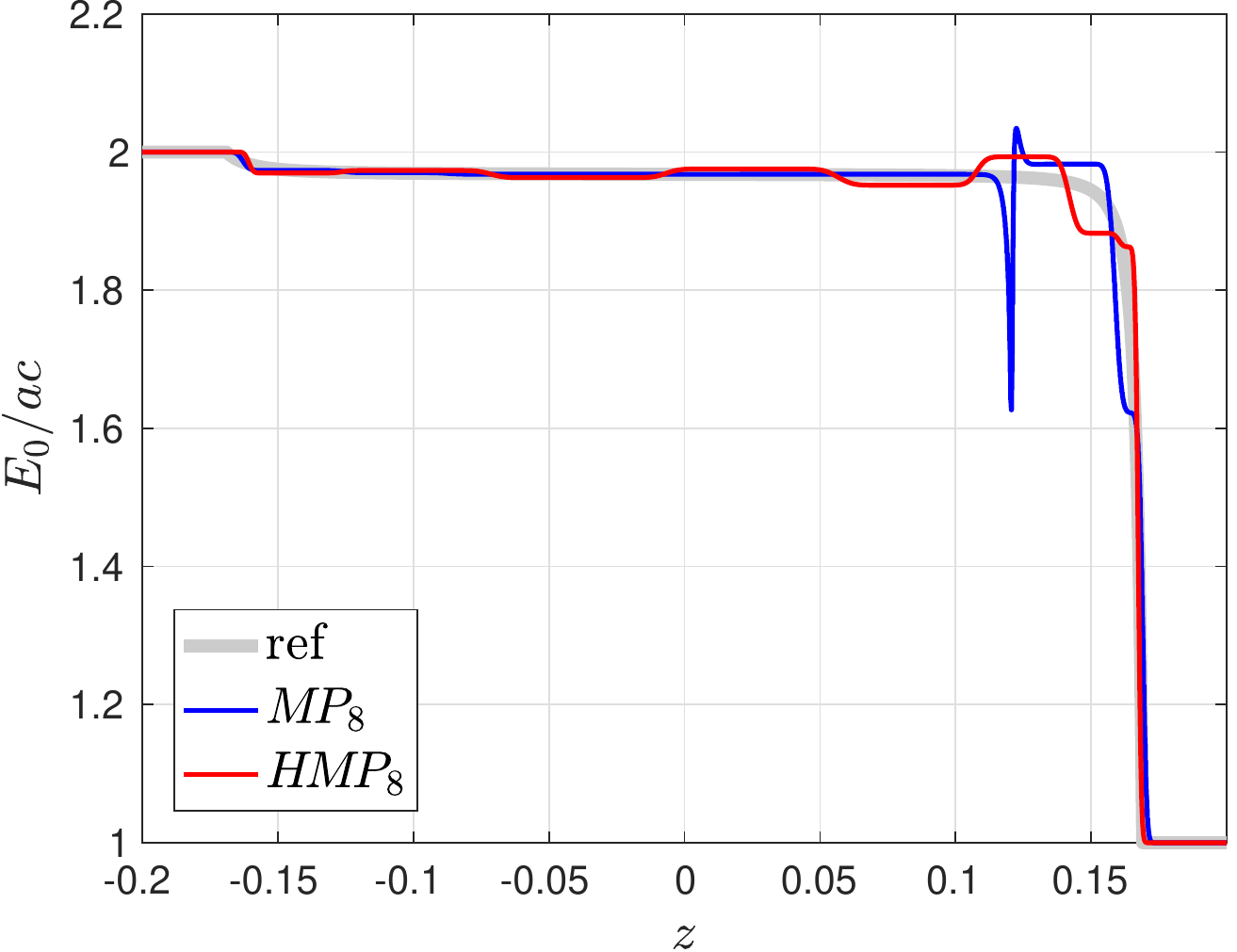}
  } \\
  \subfloat[$\frac{E_1}{E_0}$ at $\ctend=0.0287$ for $N=4$]{
      \includegraphics[width=0.33\textwidth,height=0.16\textheight]{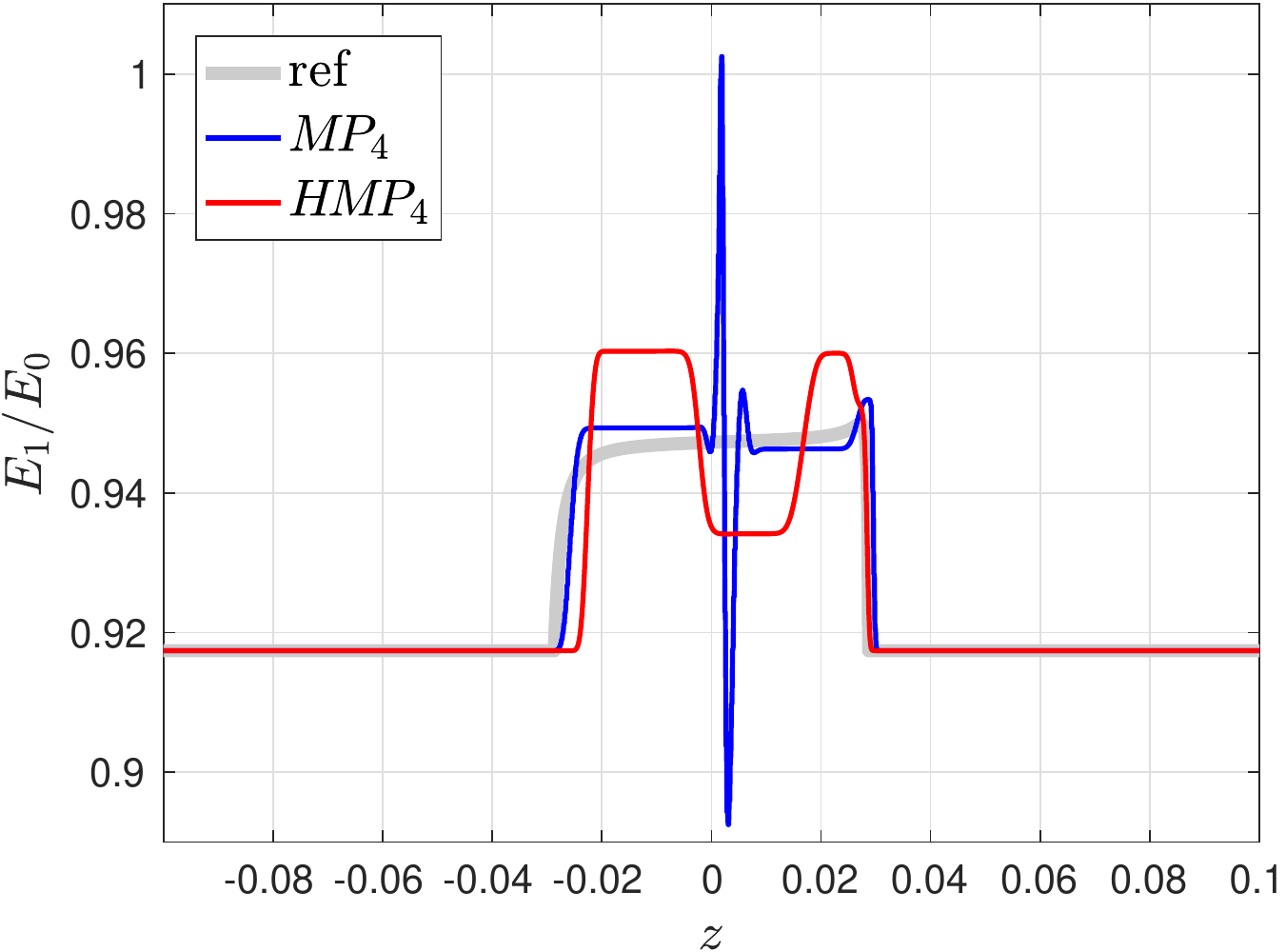}
  }
  \subfloat[$\frac{E_1}{E_0}$ at $\ctend=0.051$ for $N=6$]{
      \includegraphics[width=0.33\textwidth,height=0.16\textheight]{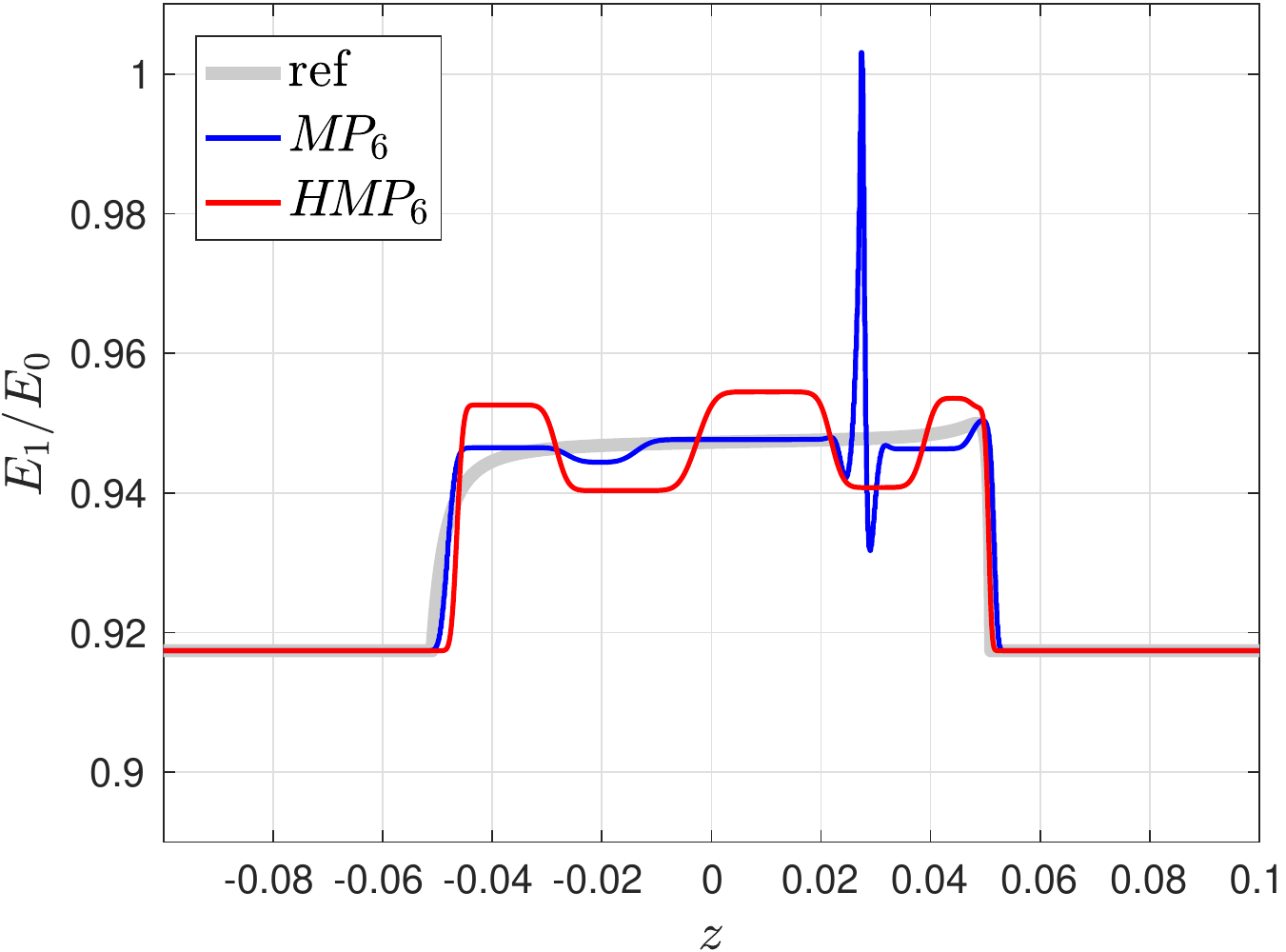}
  }
  \subfloat[$\frac{E_1}{E_0}$ at $\ctend=0.169$ for $N=8$]{
      \includegraphics[width=0.33\textwidth,height=0.16\textheight]{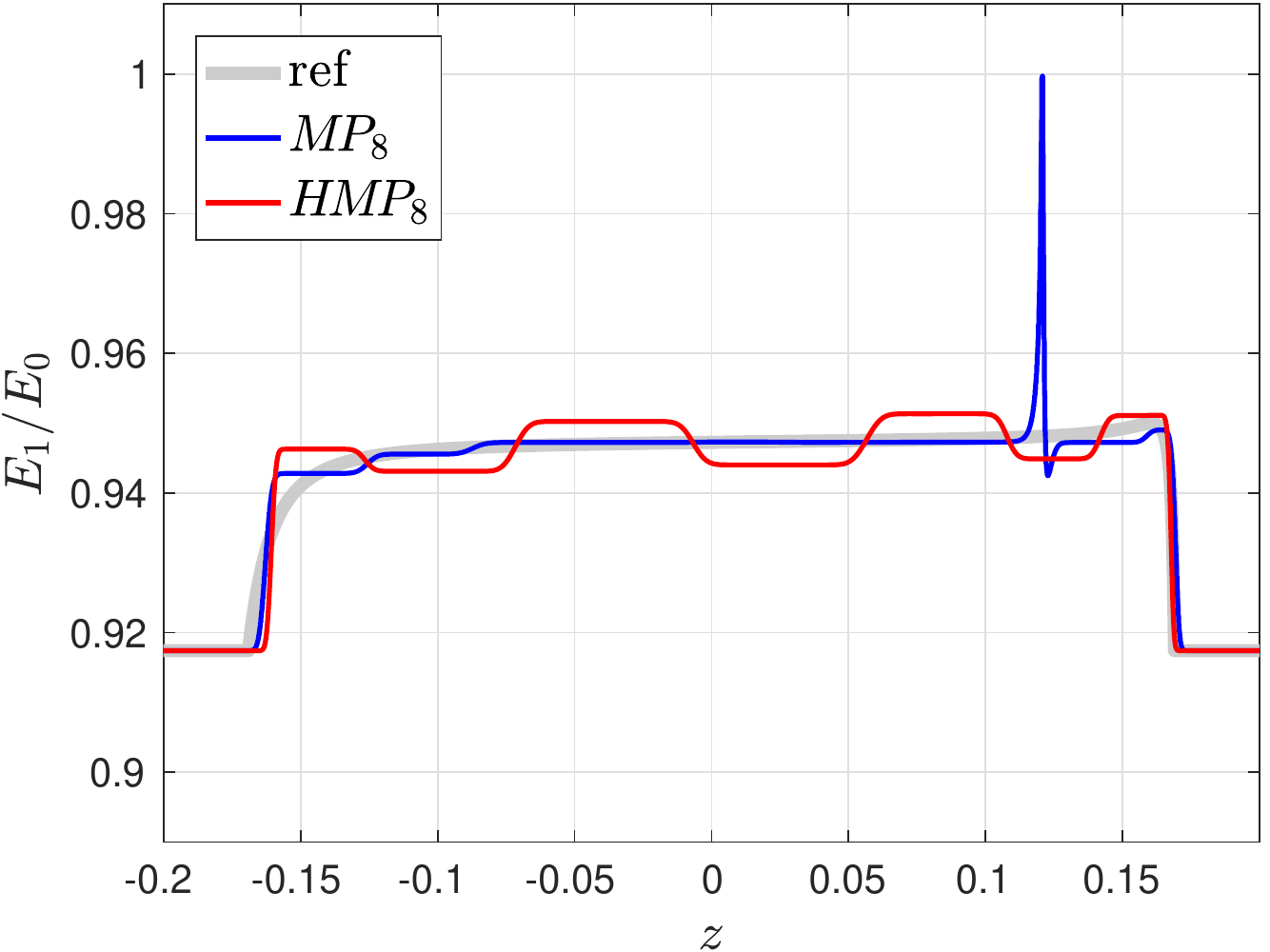}
  }
  \caption{\label{fig:Hyperbolicitytest_ex2}
  Profiles of $E_0$ and $\frac{E_1}{E_0}$ for the \MPN model, the \HMPN model and the analytical
  solution at specific end times for the Riemann problem.}
\end{figure}

\begin{figure}[htb]
  \centering 
  \subfloat[$E_0$ of \HMPN]{
      \includegraphics[width=0.45\textwidth,height=0.22\textheight]{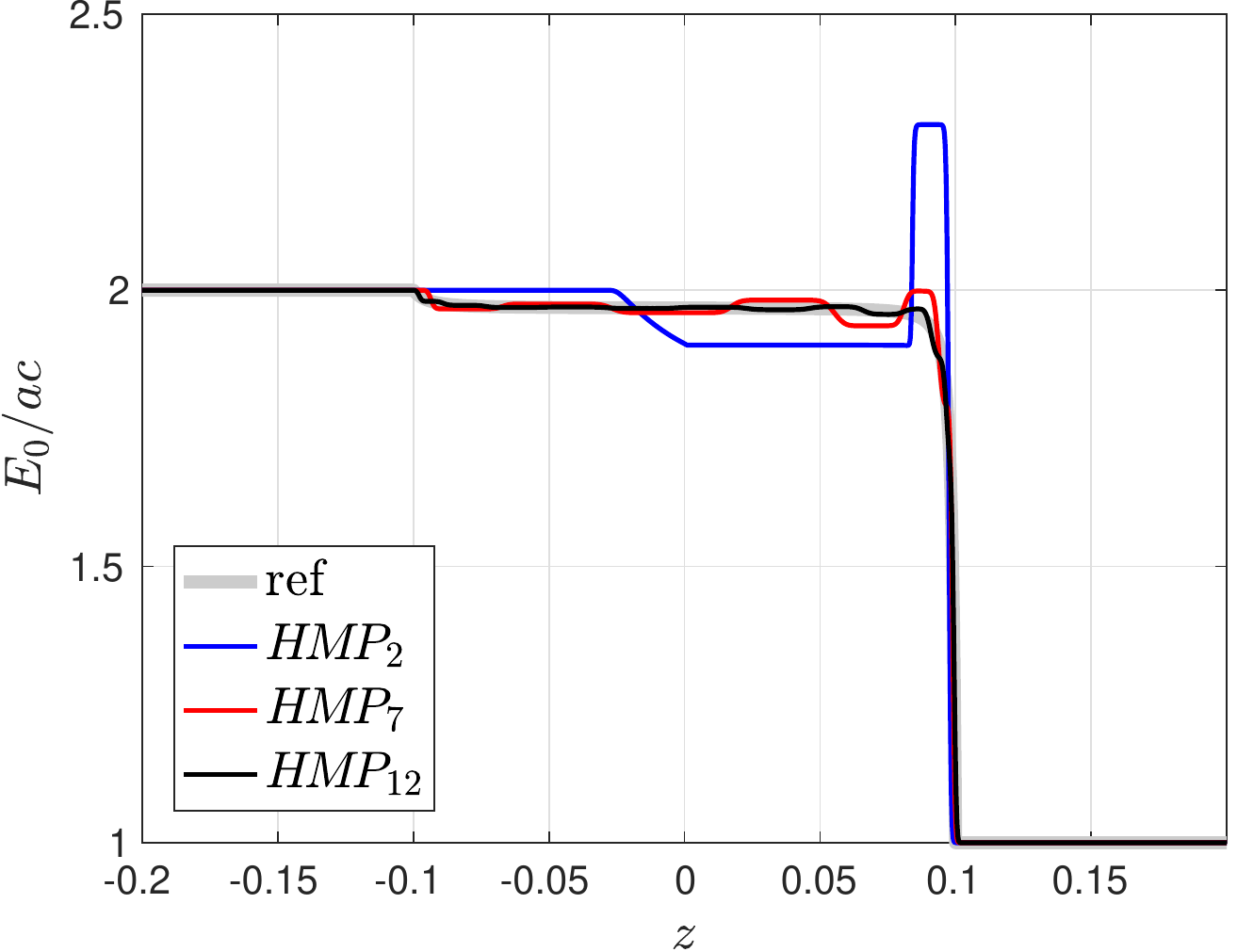}
  }
  \subfloat[$\frac{E_1}{E_0}$ of \HMPN]{
      \includegraphics[width=0.45\textwidth,height=0.22\textheight]{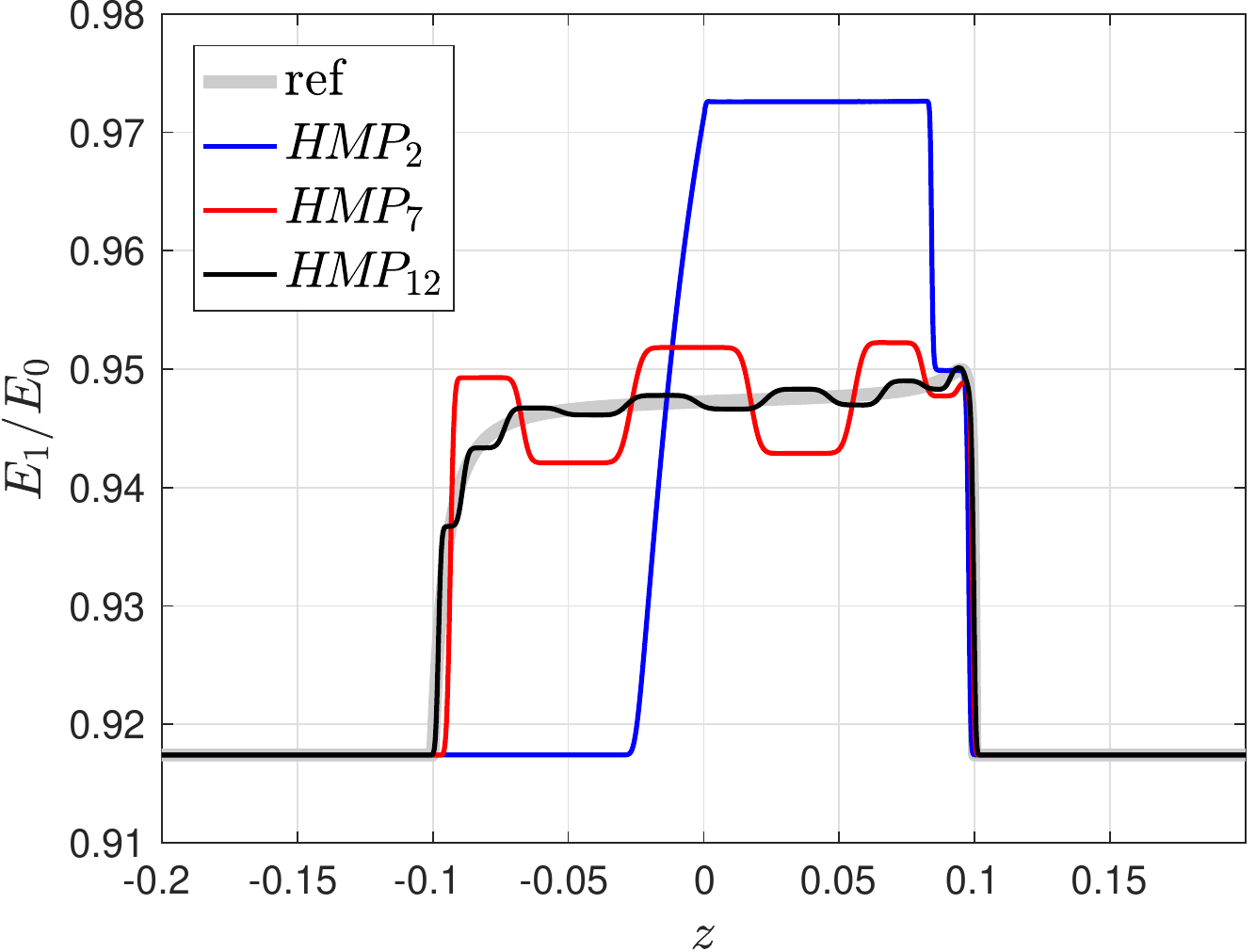}
  }\\
  \subfloat[$E_0$ of \PN]{
      \includegraphics[width=0.45\textwidth,height=0.22\textheight]{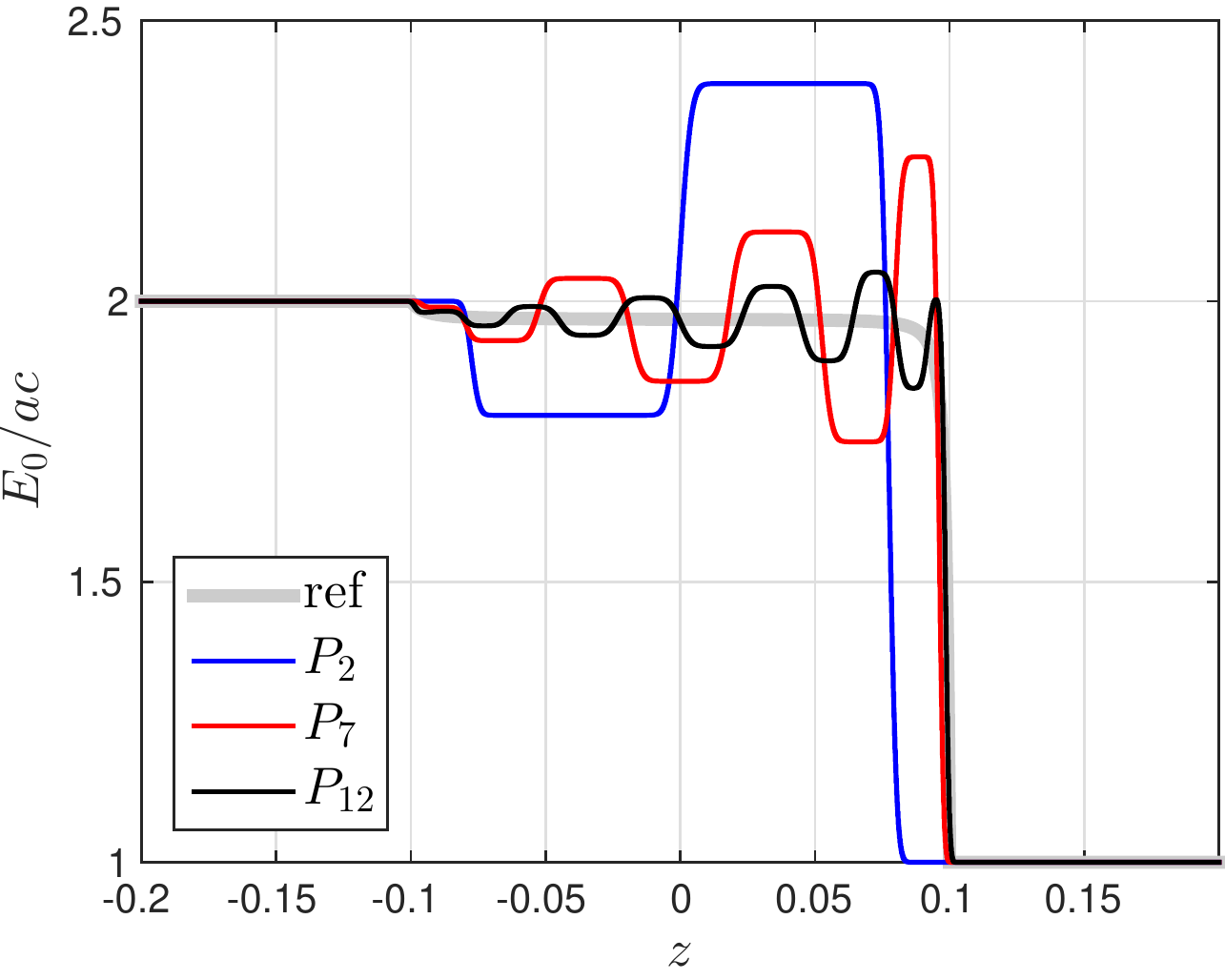}
  }
  \subfloat[$\frac{E_1}{E_0}$ of \PN]{
      \includegraphics[width=0.45\textwidth,height=0.22\textheight]{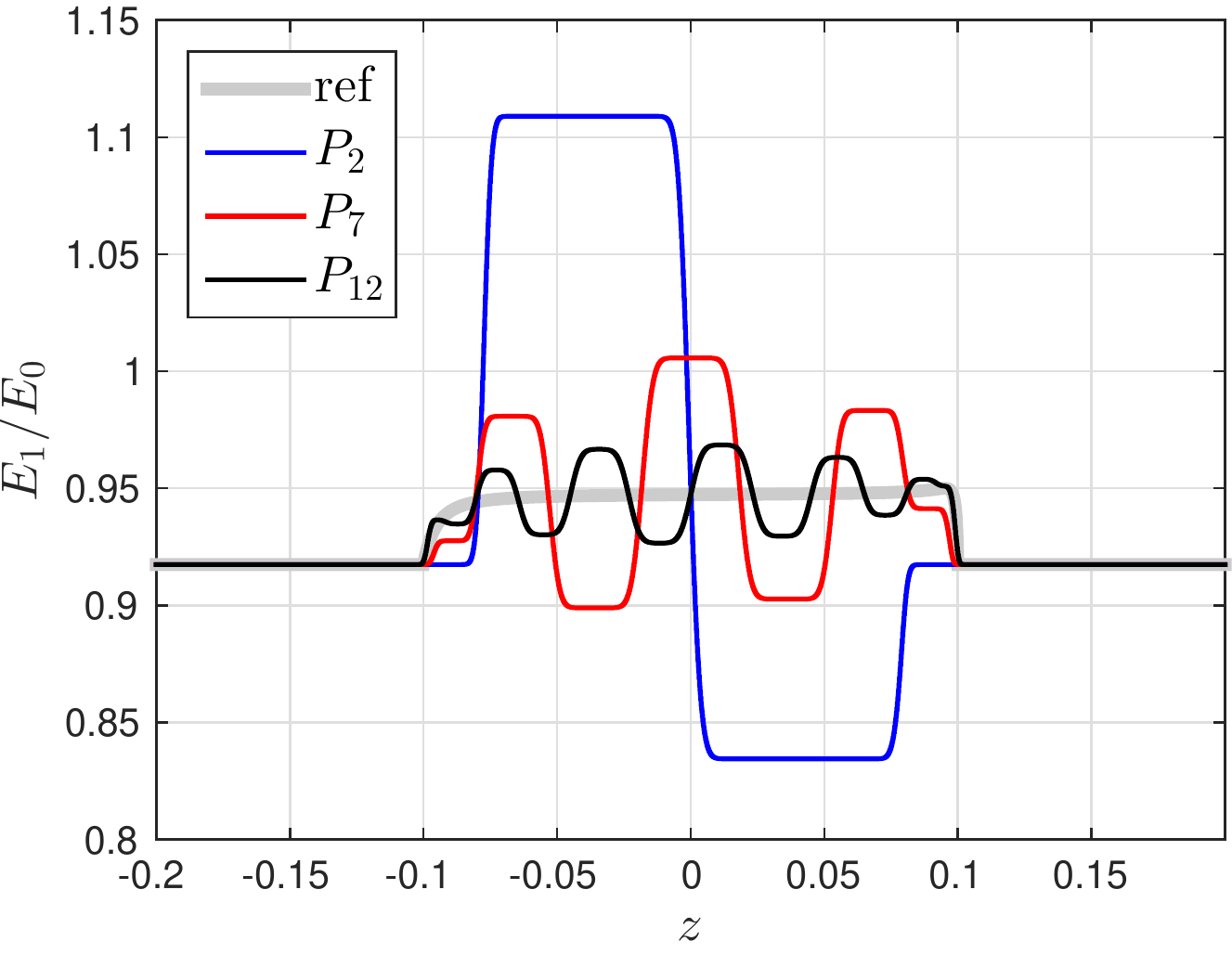}
  }
  \caption{\label{fig:Hyperbolicitytest_ex2_hmpn}
  Profiles of $E_0$ and $\frac{E_1}{E_0}$ for the \HMPN model \add{and the \PN model} 
  at the end time $\ctend=0.1$ for the
  Riemann problem.}
\end{figure}

We first use the Riemann problem in \eqref{eq:RiemannExample} to test the hyperbolicity.
The computational domain is set as $[-0.5,0.5]$ and the number of cells is $N_{\text{cell}}=10000$.
\Cref{fig:Hyperbolicitytest_ex2} presents the profiles of $E_0$ and $\dfrac{E_1}{E_0}$ for the \MPN
model, the \HMPN model, and the analytical solution (reference solution). The end time is determined
by the \MPN model when it blows up. Due to the loss of hyperbolicity, the \MPN model blows up in a
short time. The \HMPN does not suffer such an issue and gives reliable solutions thanks to the
hyperbolic regularization.

We also present the results of the \HMPN model \add{and the \PN model} at $\ctend=0.1$ for different $N$ in
\Cref{fig:Hyperbolicitytest_ex2_hmpn}. Clearly, as $N$ increases, the profiles of $E_0$ and
$\dfrac{E_1}{E_0}$ approach to the analytical solutions, \add{and the \HMPN model gets a better
approximation than the \PN model.} 

\paragraph{Continuous initial value}
We consider the problem with a continuous initial value as
\begin{equation}
  I = \begin{cases}
    6I_0,& x\leq -\frac{1}{10},\\
    (10-10x)I_0,& -\frac{1}{10} < x \leq \frac{1}{10},\\
    4I_0,& x>\frac{1}{10}.
  \end{cases}
\end{equation} 
where $I_0$ is given by 
\begin{equation}
  I_0(\mu) = \dfrac{1}{2}ac\left(\frac{1}{10}\delta(\mu) + \frac{9}{10}\delta(\mu-1)\right)
  ,\quad -1\leq \mu\leq 1.
\end{equation}


\begin{figure}[htb]
  \centering 
  \subfloat[$E_0$ at $\ctend=0.07$ with $N=4$]{
      \includegraphics[width=0.33\textwidth,height=0.16\textheight]{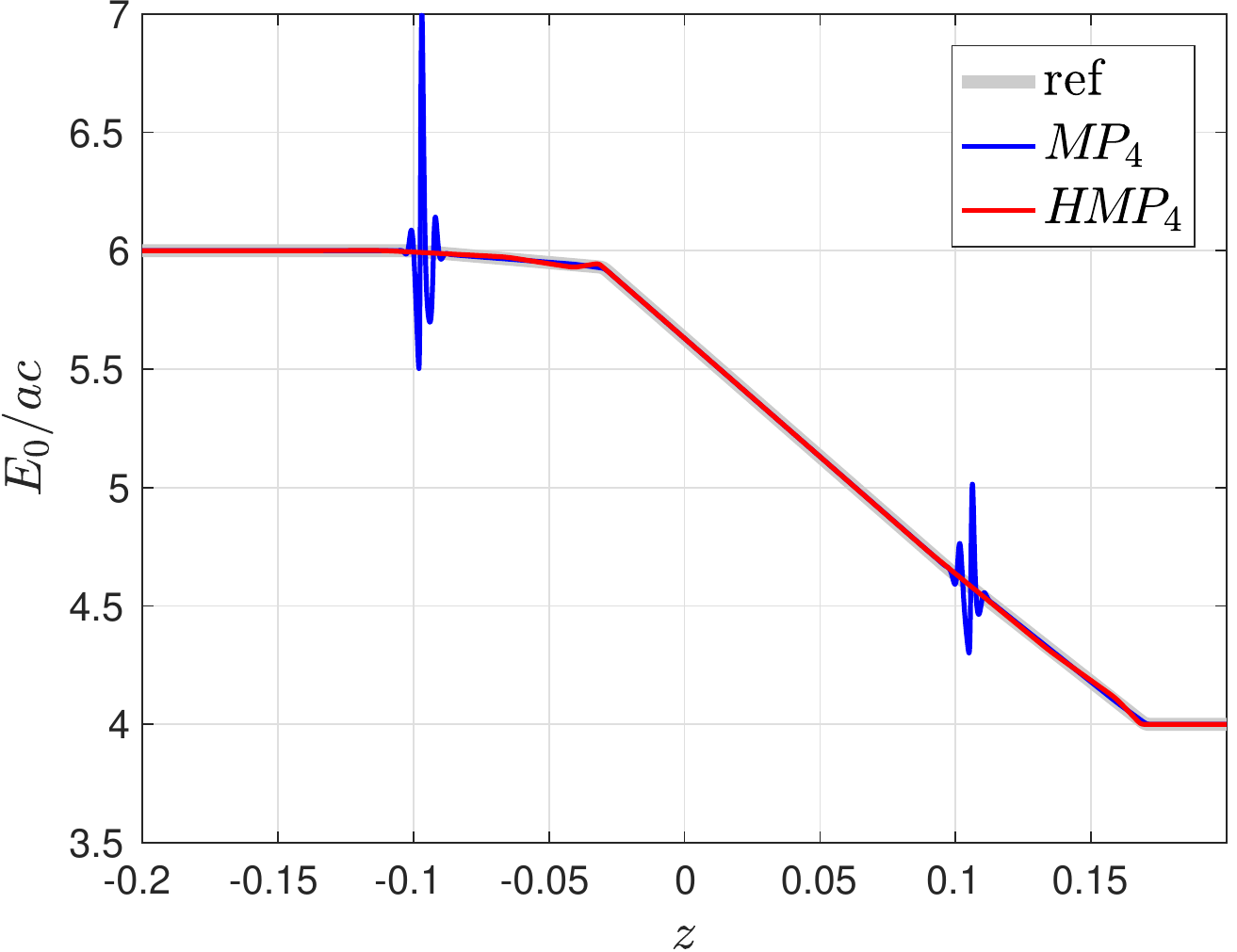}
  }
  \subfloat[$E_0$ at $\ctend=0.08$ with $N=6$]{
      \includegraphics[width=0.33\textwidth,height=0.16\textheight]{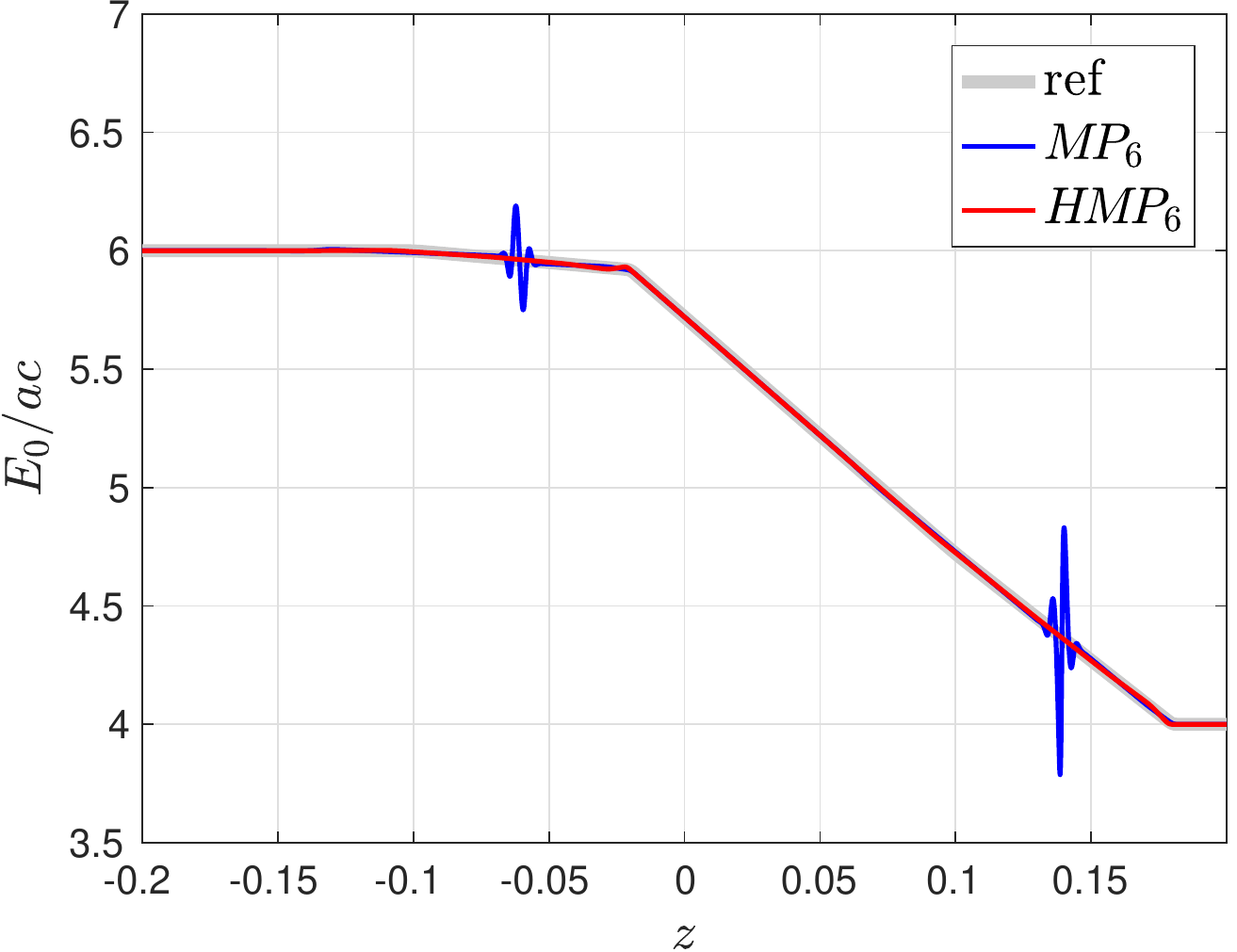}
  }
  \subfloat[$E_0$ at $\ctend=0.13$ with $N=8$]{
      \includegraphics[width=0.33\textwidth,height=0.16\textheight]{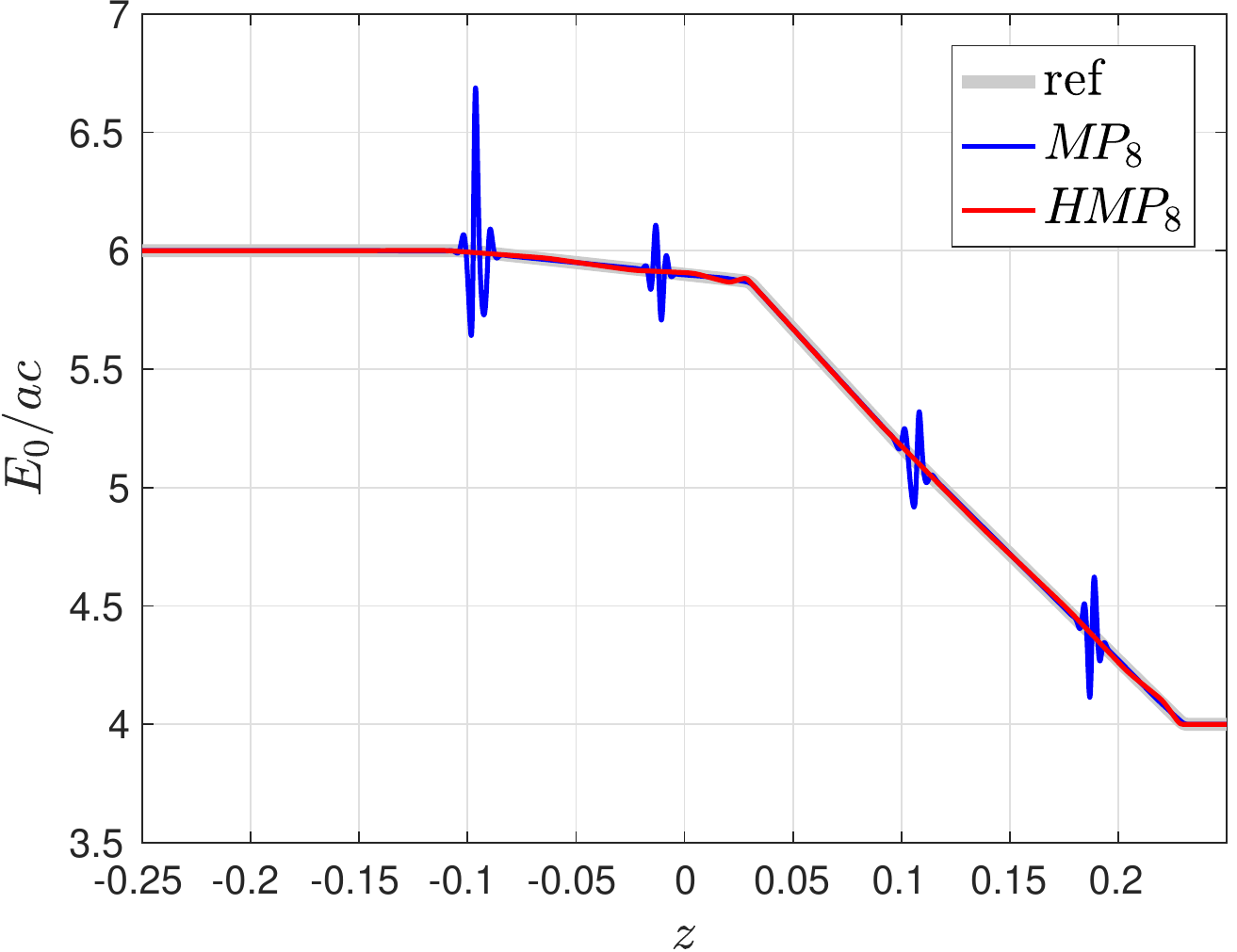}
  } \\
  \subfloat[$\frac{E_1}{E_0}$ at $\ctend=0.07$ with $N=4$]{
      \includegraphics[width=0.33\textwidth,height=0.16\textheight]{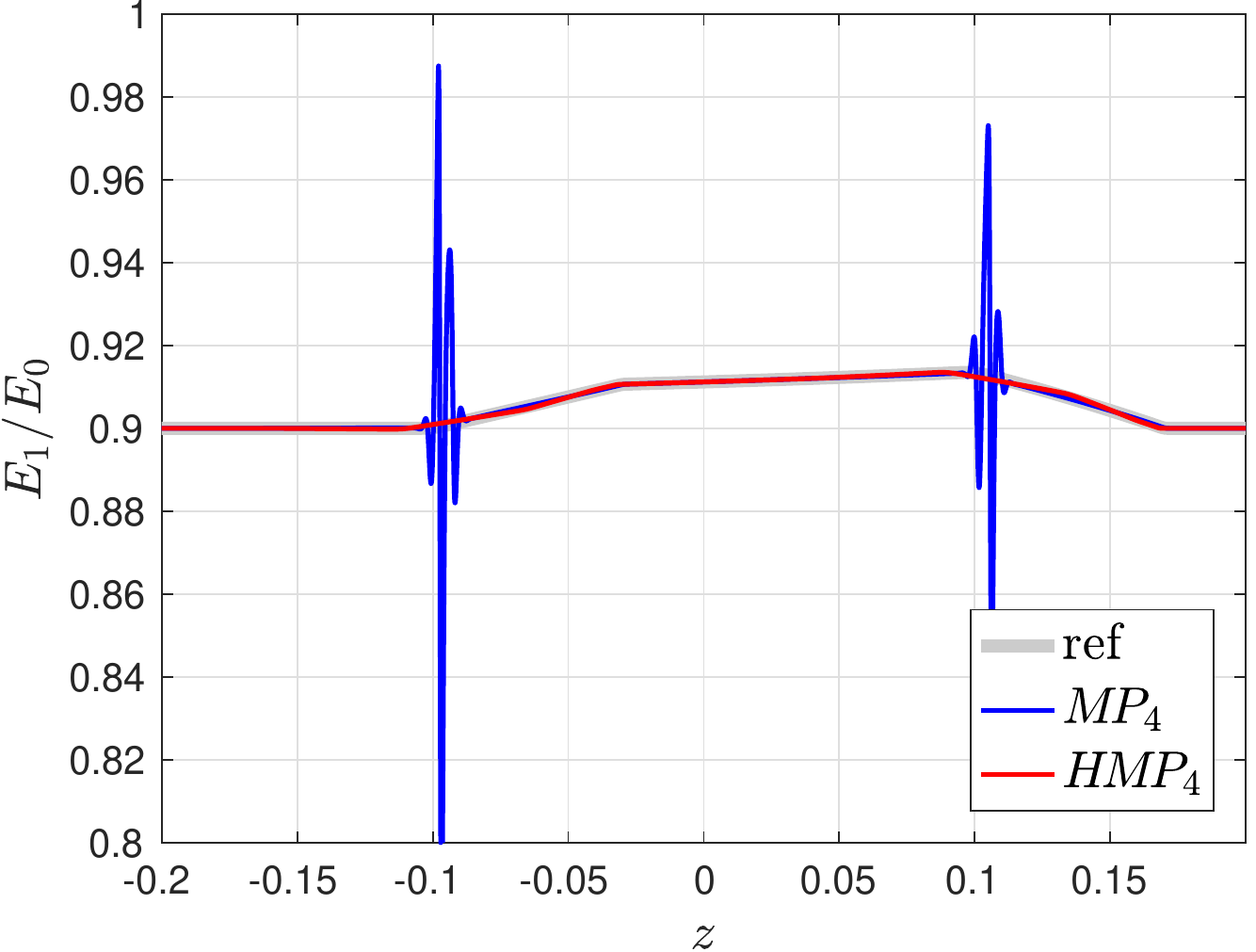}
  }
  \subfloat[$\frac{E_1}{E_0}$ at $\ctend=0.08$ with $N=6$]{
      \includegraphics[width=0.33\textwidth,height=0.16\textheight]{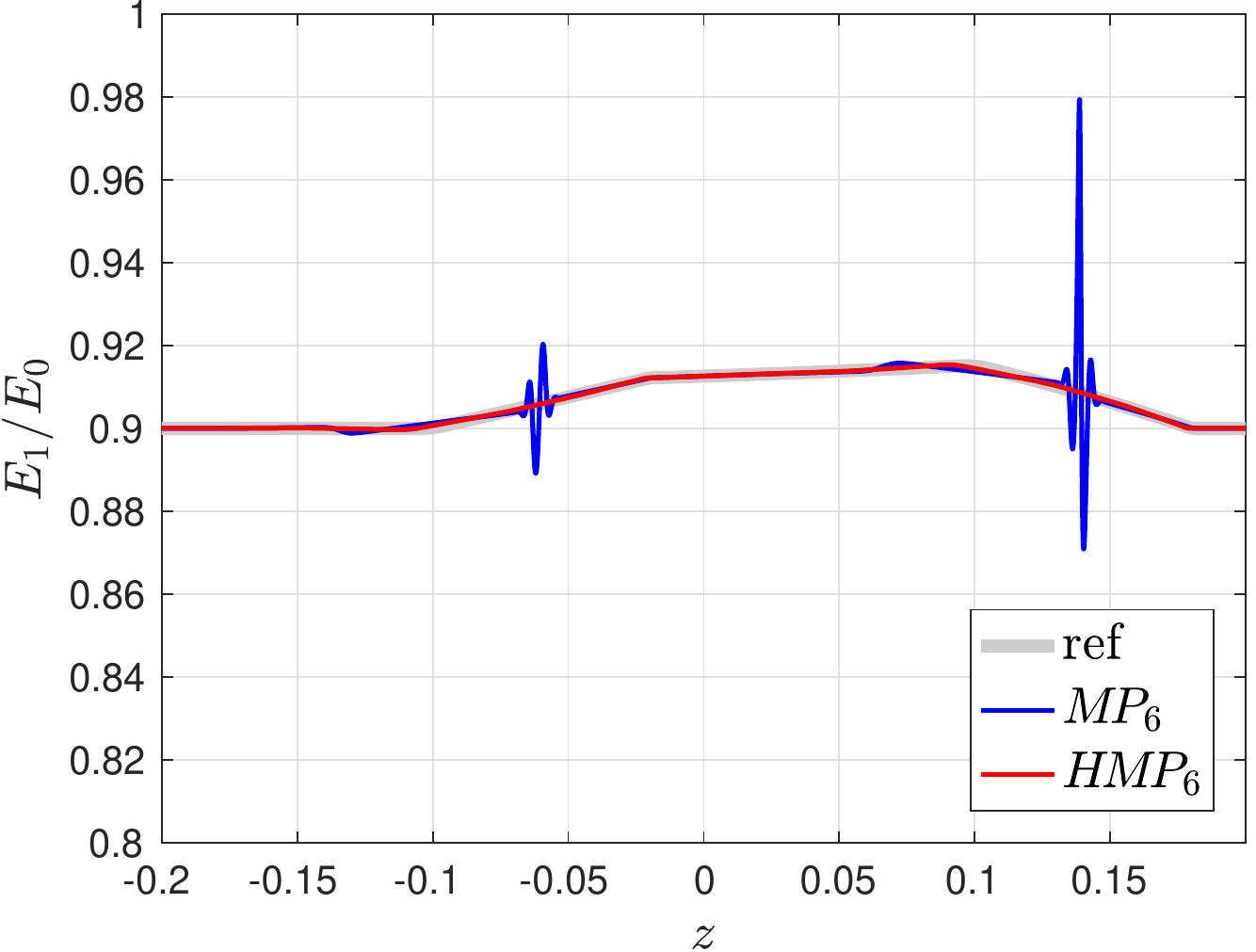}
  }
  \subfloat[$\frac{E_1}{E_0}$ at $\ctend=0.13$ with $N=8$]{
      \includegraphics[width=0.33\textwidth,height=0.16\textheight]{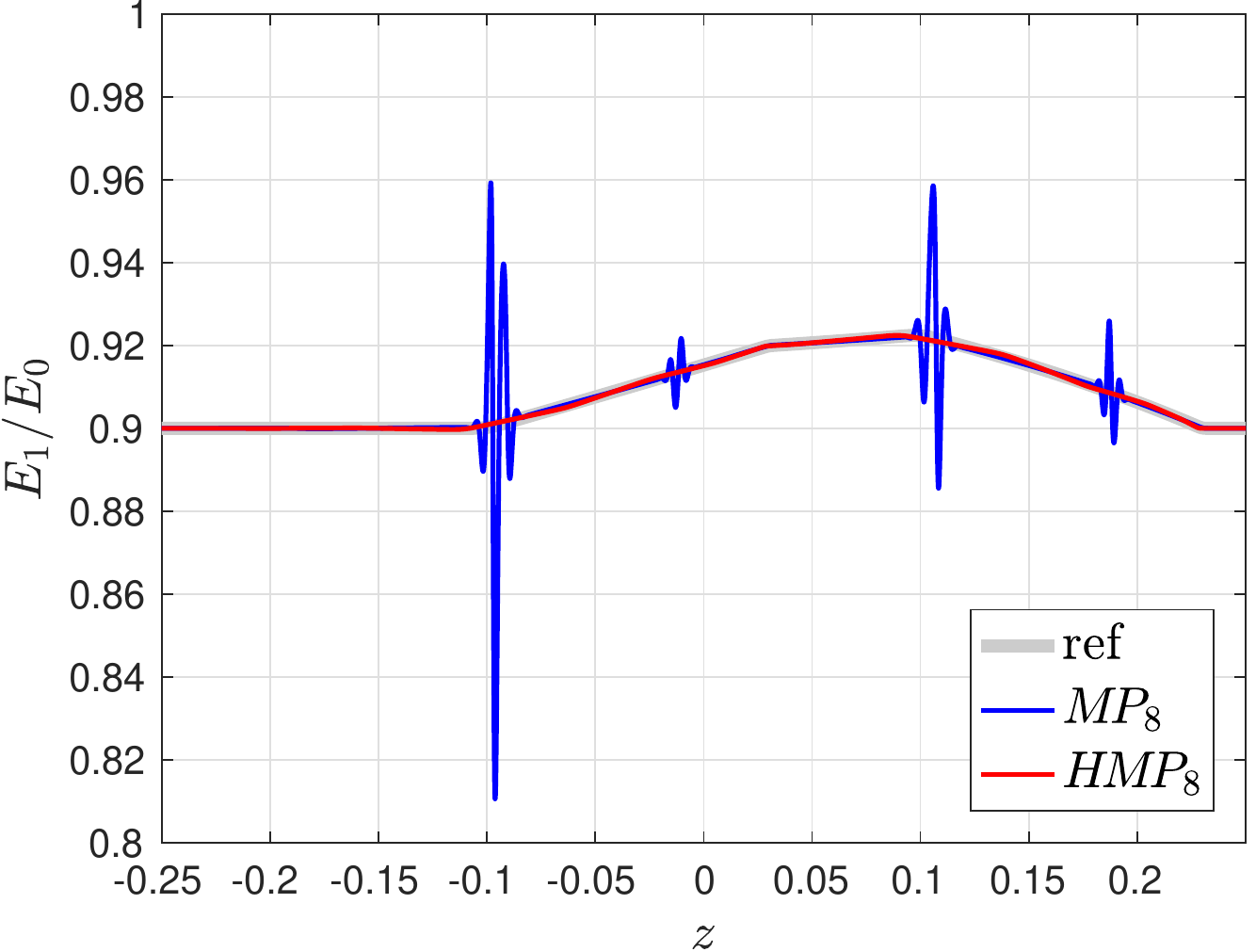}
  }
  \caption{ \label{fig:Hyperbolicitytest_ex1}
  Profiles of $E_0$ and $\frac{E_1}{E_0}$ for the \MPN model, the \HMPN model and the reference
  solution for the continuous initial value problem.}
\end{figure}

\begin{figure}[htb]
  \centering 
  \subfloat[$E_0$ of \HMPN]{
  \includegraphics[width=0.45\textwidth,height=0.22\textheight]{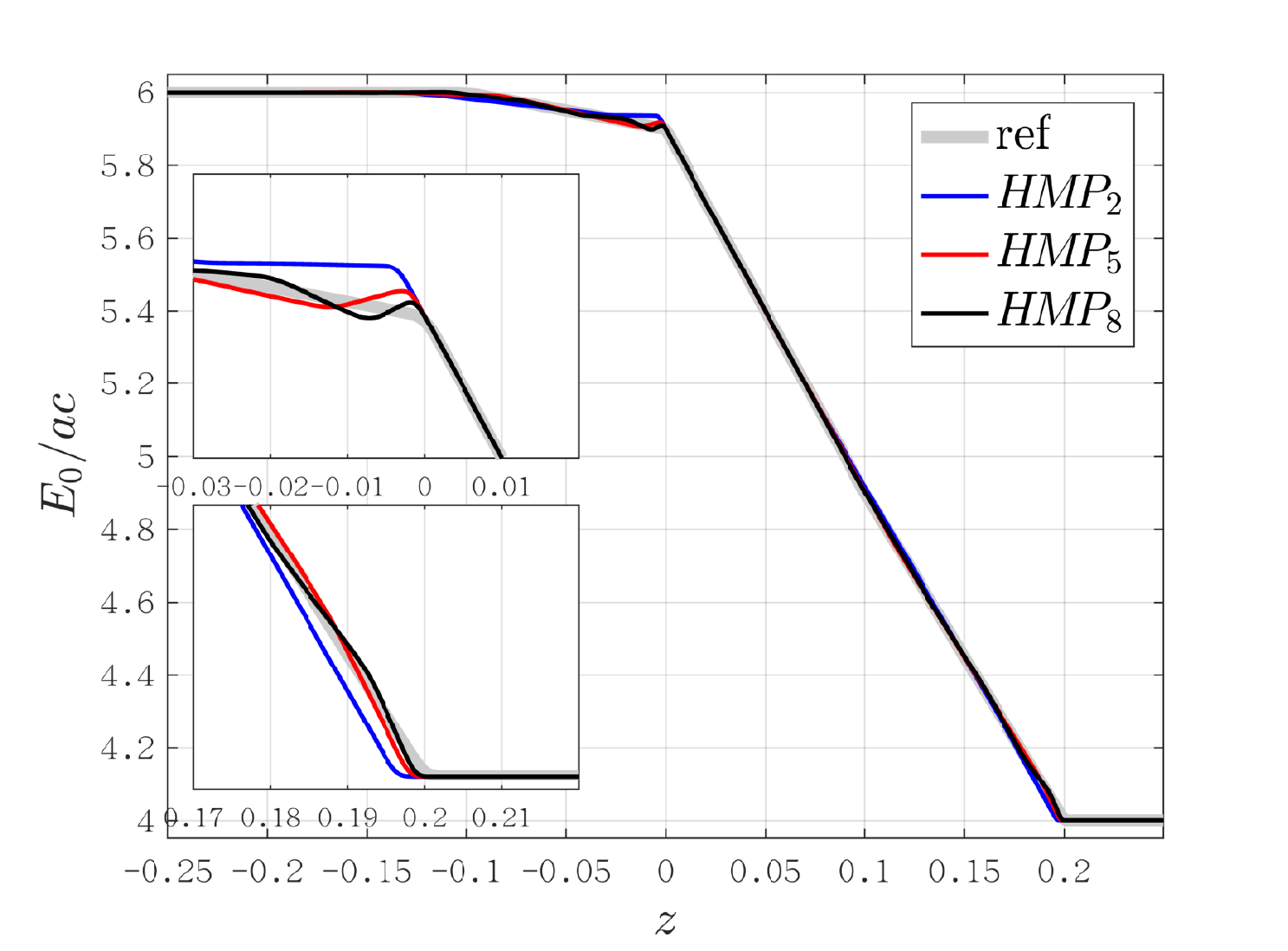}
  }
  \subfloat[$\frac{E_1}{E_0}$ of \HMPN]{
      \includegraphics[width=0.45\textwidth,height=0.22\textheight]{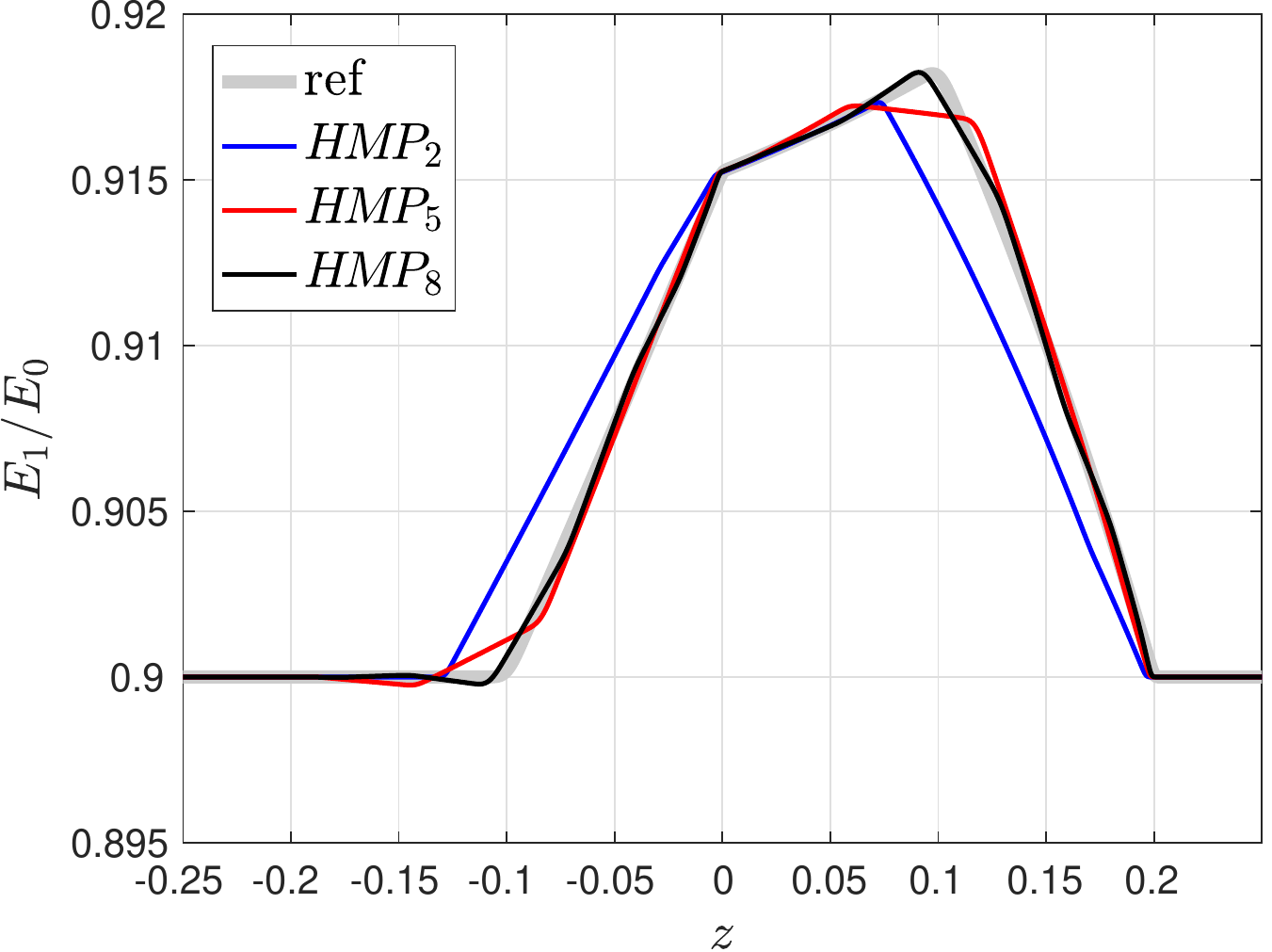}
  }\\
  \subfloat[$E_0$ of \PN]{
  \includegraphics[width=0.45\textwidth,height=0.22\textheight]{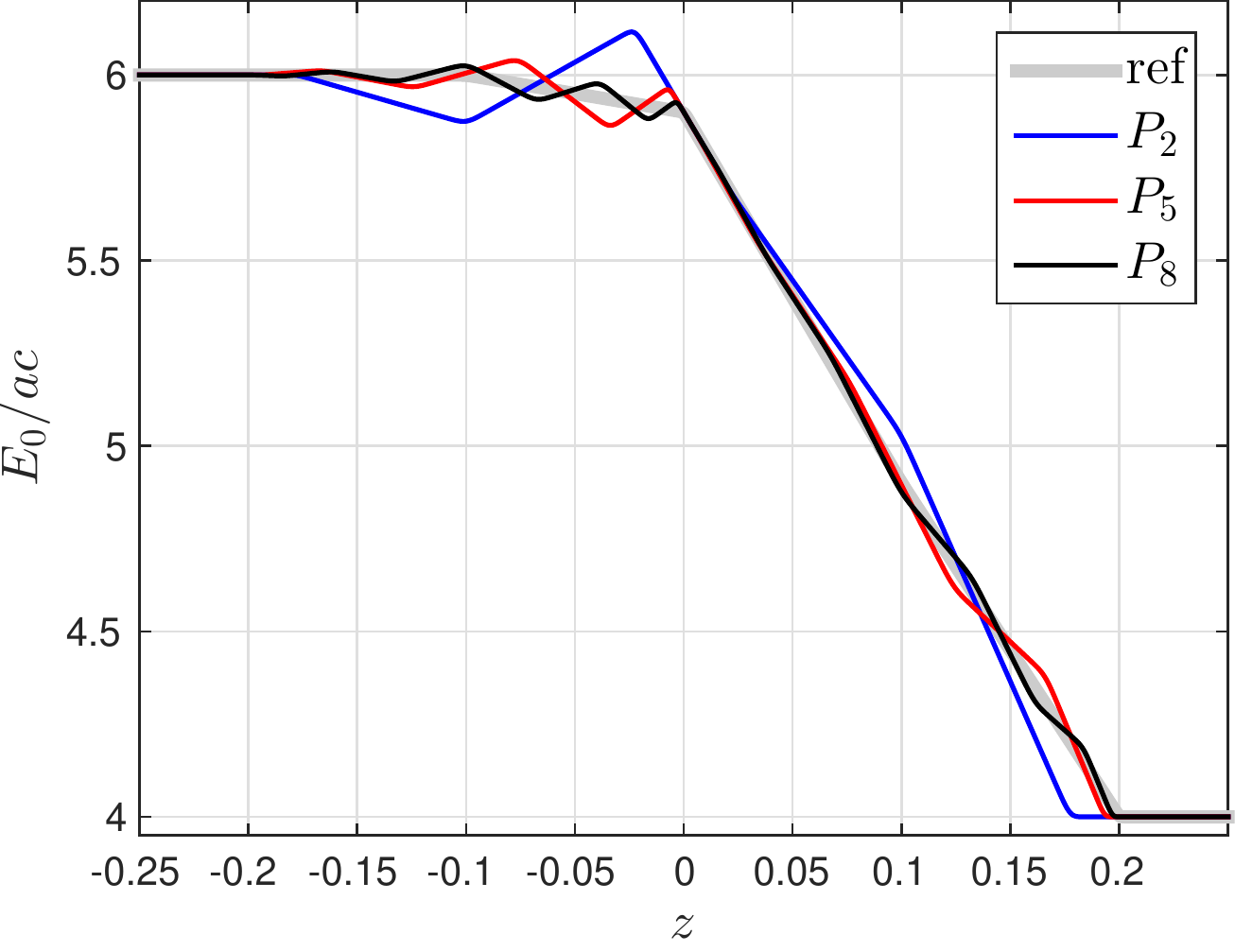}
  }
  \subfloat[$\frac{E_1}{E_0}$ of \PN]{
      \includegraphics[width=0.45\textwidth,height=0.22\textheight]{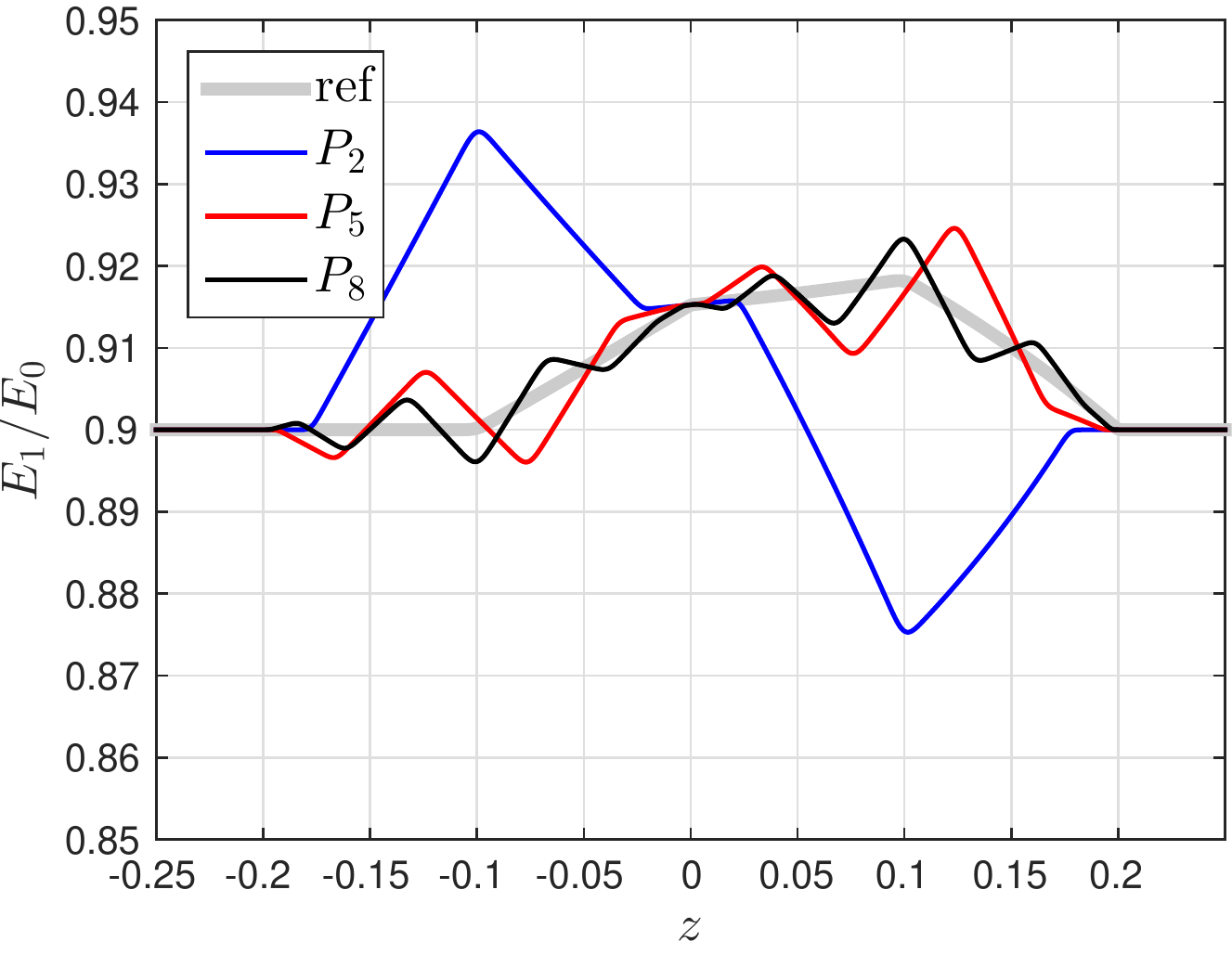}
  }
  \caption{\label{fig:Hyperbolicitytest_ex1_hmpn}
  Profiles of $E_0$ and $\frac{E_1}{E_0}$ for the \HMPN model \add{and the \PN model} 
  at the $\ctend=0.1$ for different
  $N$ for the continuous initial value problem.}
\end{figure}

The computational domain is also set as $[-0.5,0.5]$ and $N_{\cell}=10000$.
\Cref{fig:Hyperbolicitytest_ex1} presents the profiles of $E_0$ and $\frac{E_1}{E_0}$ of 
the \MPN model, the \HMPN model and the reference solution at specific end time. The reference
solution is calculated by the \PN model with $N=100$. The end time is determined by the \MPN model 
when it blows up. The results for the \HMPN model \add{and the \PN model} 
at $\ctend=0.1$ with different $N$ are presented
in \cref{fig:Hyperbolicitytest_ex1_hmpn}. 
One can see that the conclusions for the Riemann problem are also valid for the continuous initial
value, \add{and similarly, the approximation of the \HMPN model is more accurate than the \PN model.}

\subsubsection{Two-beam instability problem}
The two-beam instability problem is designed to test a closure's ability to handle multi-modal
distributions \cite{vikas2013radiation}. The maximum entropy model ($M_N$) yields unphysical shocks
\cite{brunner2001one,Hauck2011high} in this problem.
The computational domain is $[0,1]$;  
the absorption and scattering coefficients are $2$ and $0$, respectively; the
external source term $s=0$; and the coupling term with the background medium is neglected.
The inflow boundary condition are prescribed at the both boundaries, with
$I_{\text{inflow}}=\frac{1}{2}ac$, and the initial state is set as $I\big|_{t=0} = 10^{-8}ac$. 
We use $N_{\text{cell}}=10000$ cells to simulate this problem until the solution reaches the
steady state. The results of \add{the \PN model}, the \MPN model and the \HMPN model are presented in
\cref{fig:TwoBeams}. 

\begin{figure}[htb]
  \centering 
  \subfloat[$N=2$]{
      \includegraphics[width=0.33\textwidth,height=0.16\textheight]{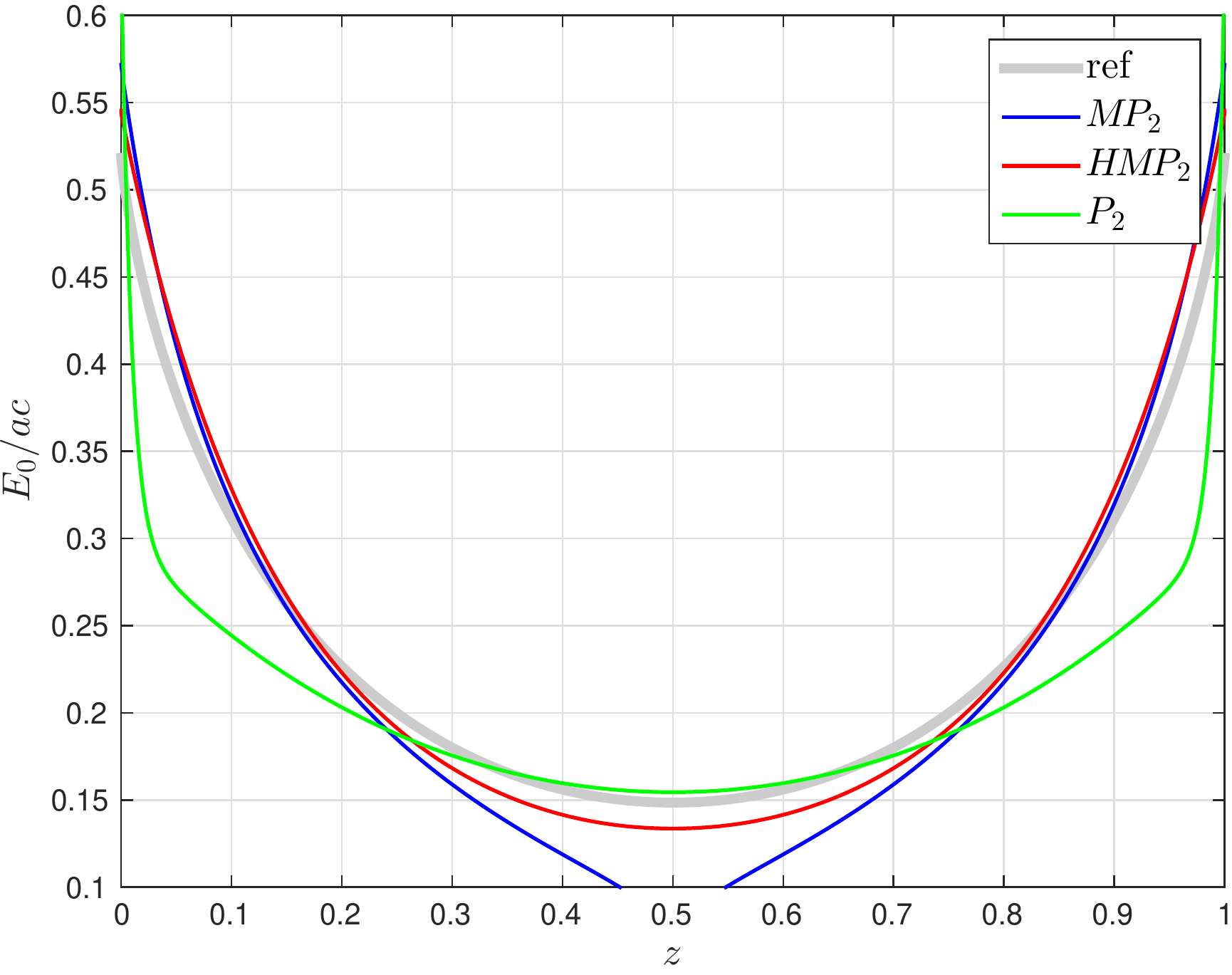}
  }
  \subfloat[$N=4$]{
      \includegraphics[width=0.33\textwidth,height=0.16\textheight]{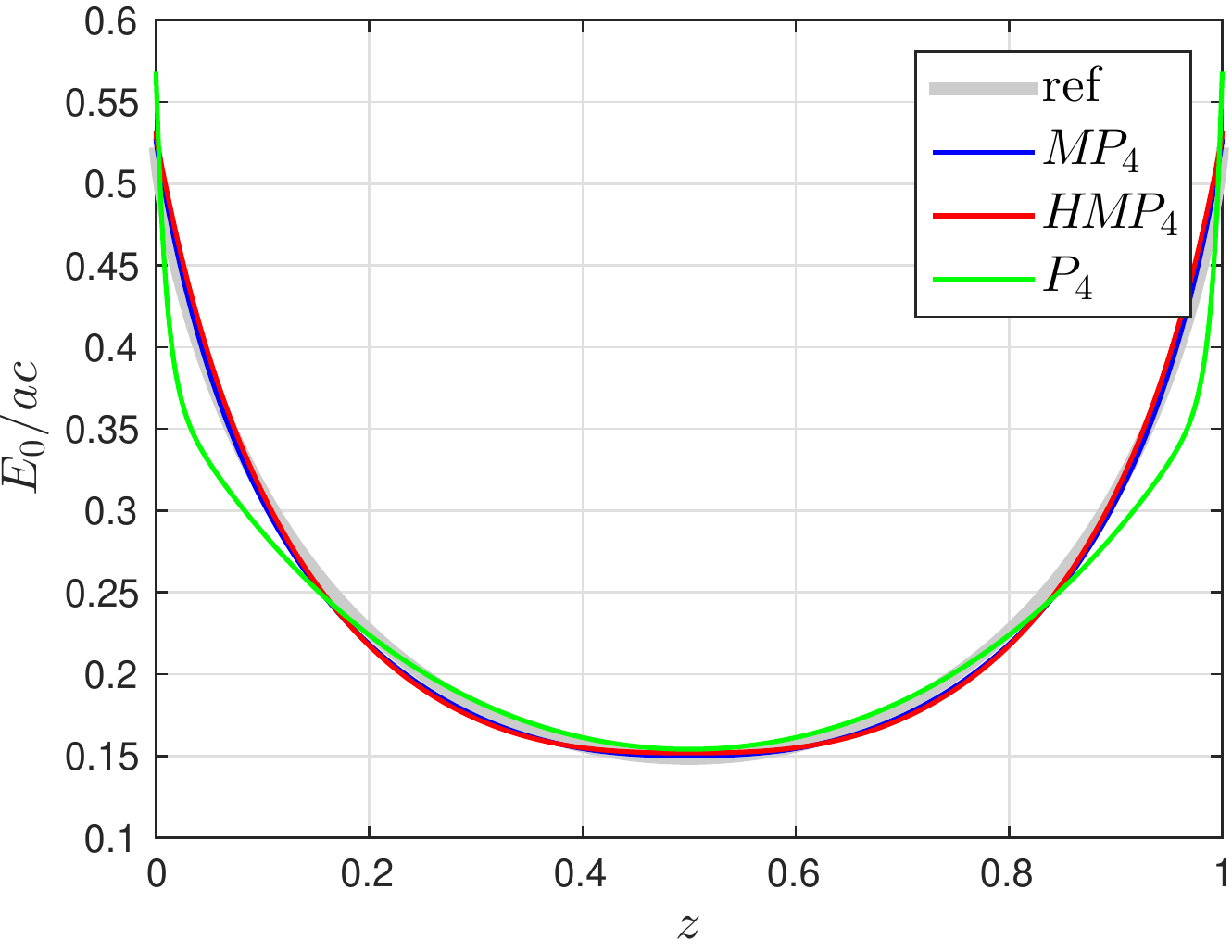}
  }
  \subfloat[$N=6$]{
      \includegraphics[width=0.33\textwidth,height=0.16\textheight]{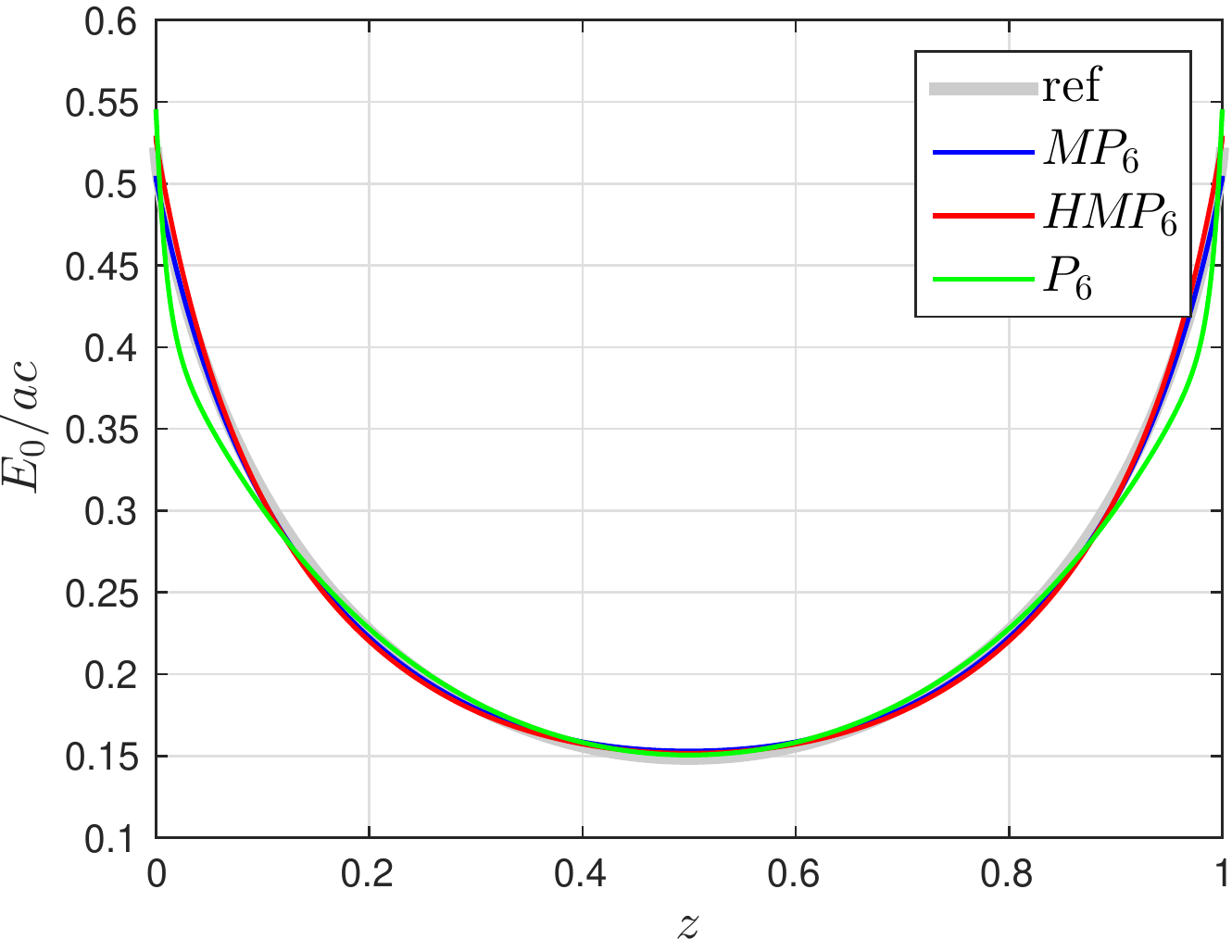}
  }
  \caption{\label{fig:TwoBeams}
  Profiles of $E_0$ for \add{the \PN model,} the \MPN model and the \HMPN model 
  for the two-beam instability problem.}
\end{figure}

For the case $N=2$, the \MPN model seems to give a result with unphysical shock while 
the profile of $E_0$ of the \HMPN model is smooth. When $N$ gets larger, the results of both models 
are close to each other and approach to the reference solution.
It is worth to point out that the characteristic speed of the \MPN model is greater than $1$,
so in the numerical simulation, the time step is smaller, thus its simulation is slower.



\subsubsection{Gaussian source problem}
This example simulates particles with an initial specific intensity that is a Gaussian distribution
in space \cite{Frank2012Perturbed, MPN}: 
\begin{equation}
  I_0(z,\mu) = \dfrac{ac}{\sqrt{2\pi\theta}} e^{-\frac{z^2}{2\theta}},\quad
  \theta=\dfrac{1}{100},\quad z\in(-L,L).
\end{equation}
Here $L$ is adopted as $\ctend+1$ such that no energy reaches the boundaries and we can set
vacuum boundary conditions at both boundaries. The external source term is zero, i.e., $s=0$, and
the absorption and scattering coefficients are $\sigma_a=0$ and $\sigma_s=1$, respectively, so the
material coupling term vanishes. 

\begin{figure}[htb]
    \centering 
    \subfloat[$N=2$]{
        \includegraphics[width=0.33\textwidth,height=0.16\textheight]{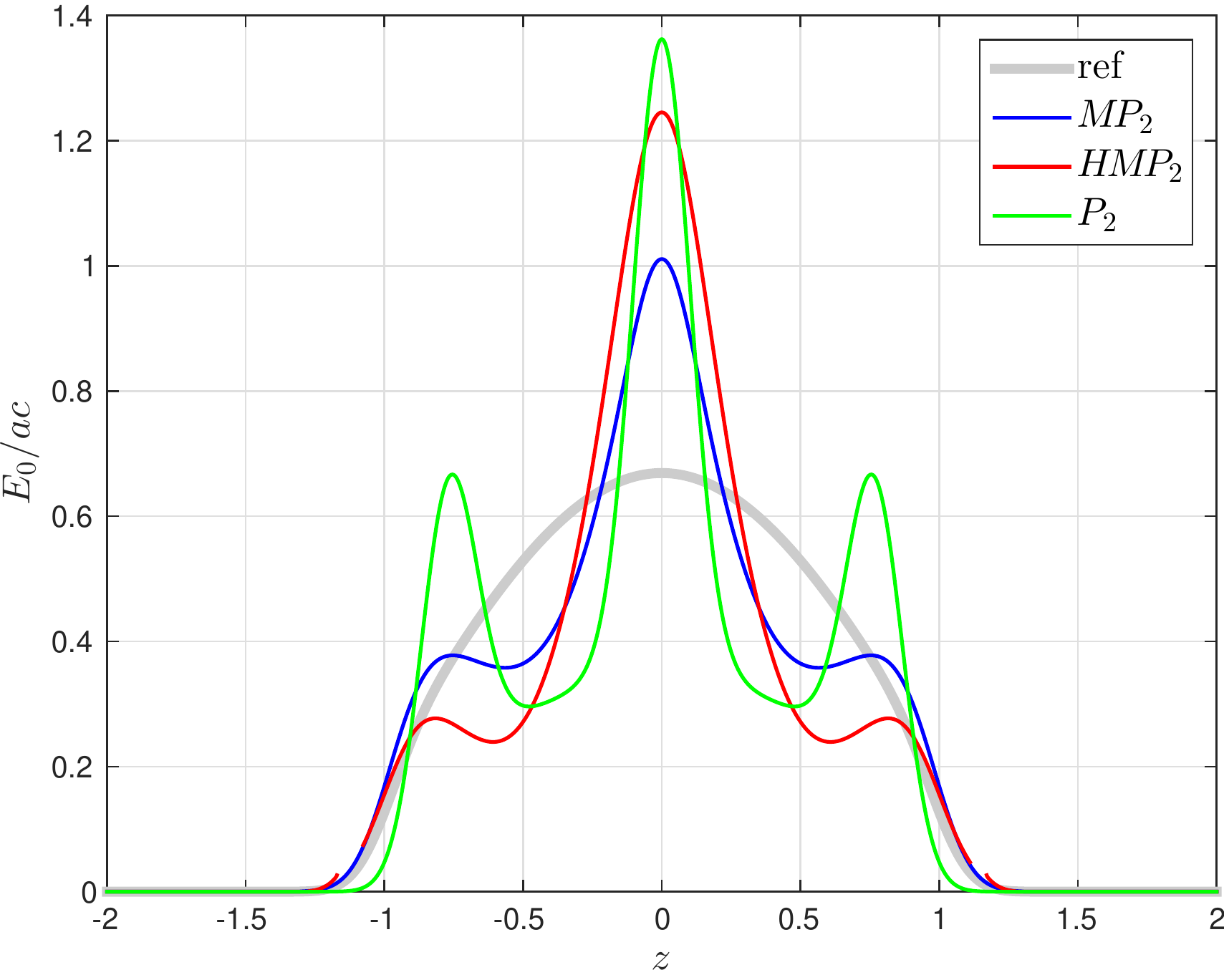}
    }
    \subfloat[$N=6$]{
        \includegraphics[width=0.33\textwidth,height=0.16\textheight]{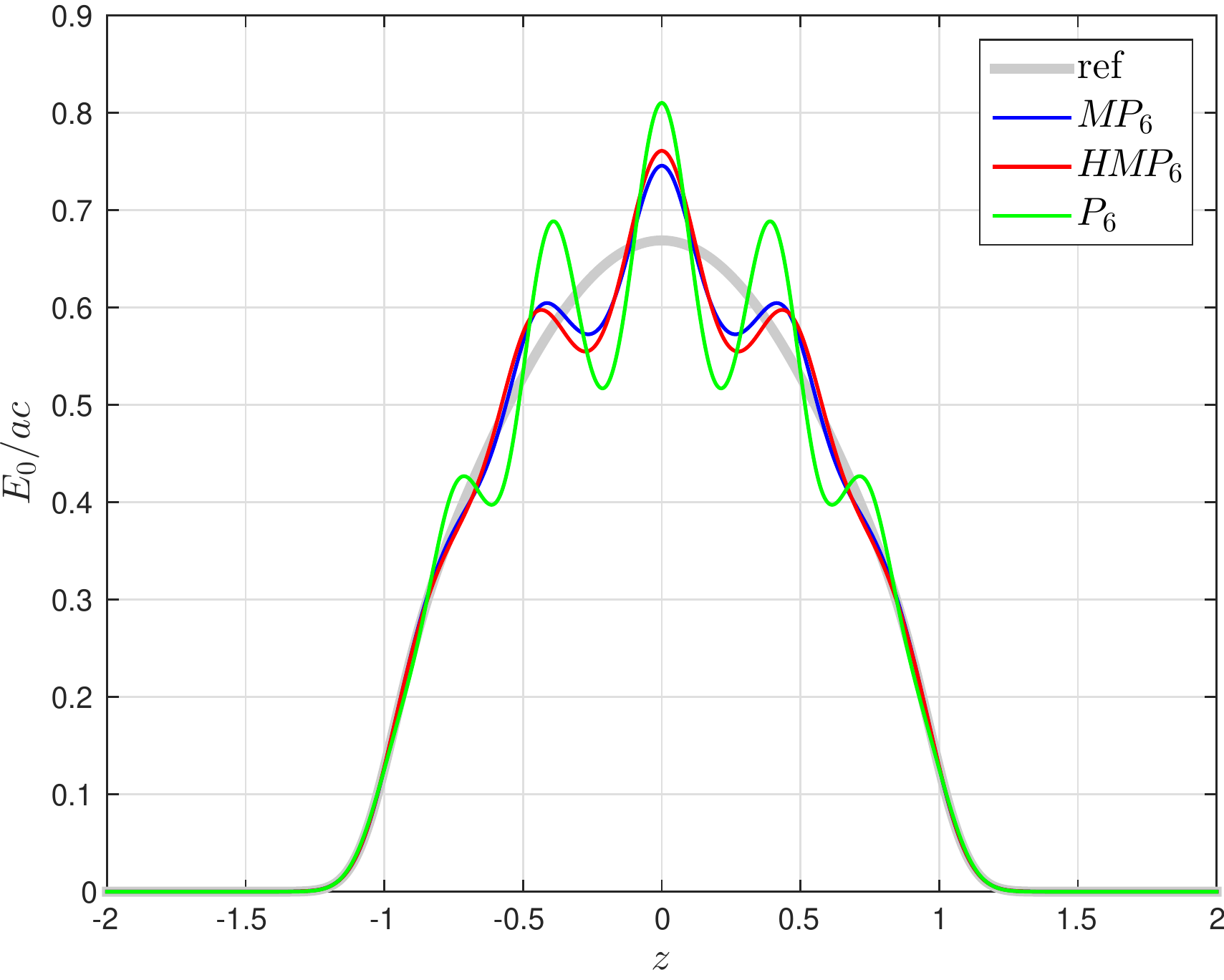}
    }
    \subfloat[$N=10$]{
        \includegraphics[width=0.33\textwidth,height=0.16\textheight]{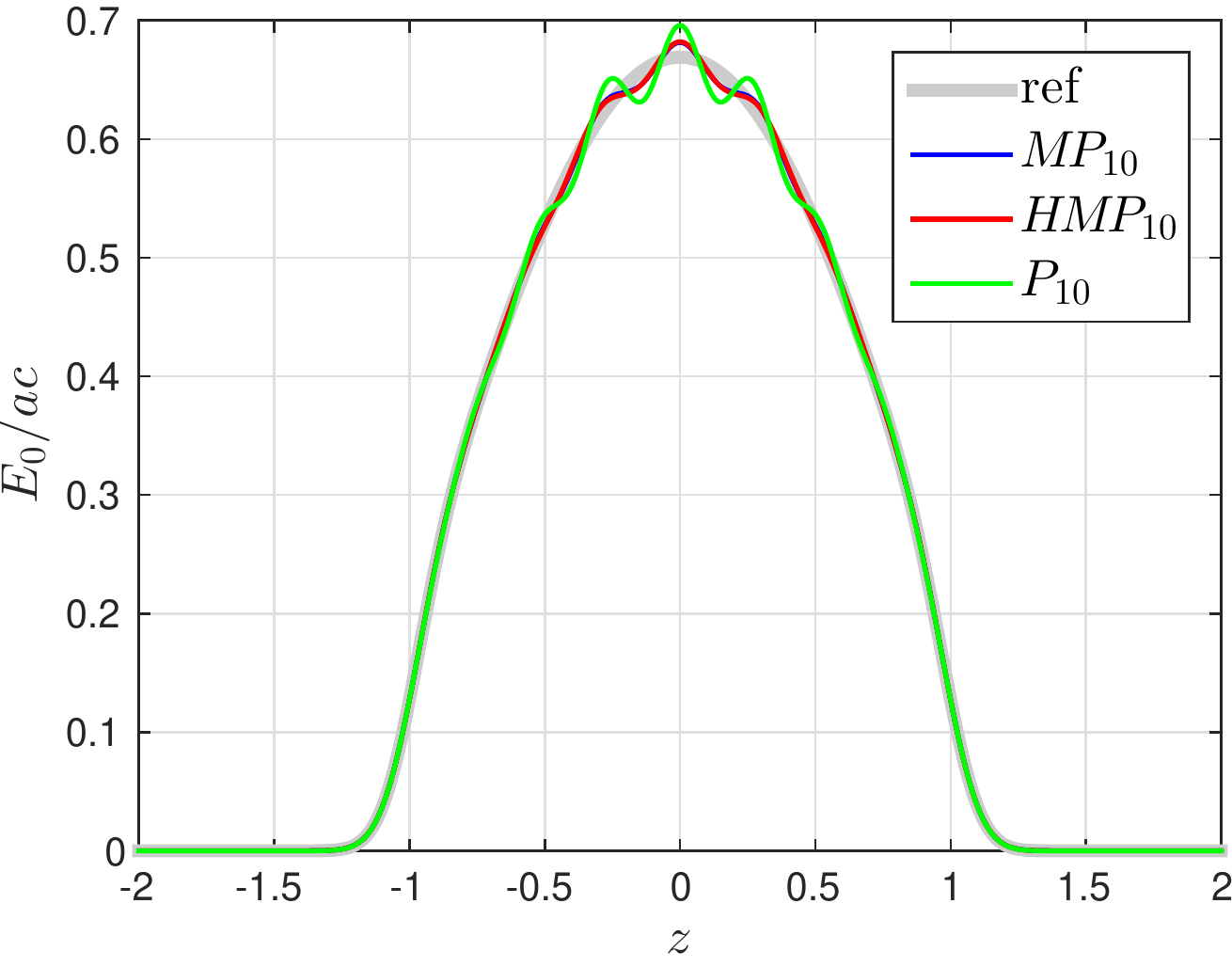}
    }
    \caption{\label{fig:GaussianSource}
    Profiles of $E_0$ for \add{the \PN model,} the \MPN model and the \HMPN model for the Gaussian source problem.}
\end{figure}

We simulate this problem until $\ctend=1$ using $N_{\cell}=10000$ cells.
\Cref{fig:GaussianSource} presents the results of \add{the \PN model,} the \MPN model, and the \HMPN model with $N=2$, 6 and
10, and the reference solution, which is calculated by the \PN model with $N=100$.

For such problem, both the \MPN model and the \HMPN model work well and their results are very
close. As the number of moments $N$ increases, the results of both models approach to the reference
solution. Particularly, if $N$ is large enough, for instance $N=10$, the results of the two
models almost coincide. For the problems where the \MPN model works, the \HMPN model also works well.

\subsubsection{Su-Olson problem}
The Su-Olson problem \cite{Olson2000Diffusion} is a non-equilibrium radiative transfer problem with
a material coupling term. 
The computation domain is $[0,30]$, and the absorption and scattering coefficients are $\sigma_a=1$
and $\sigma_s=0$, respectively. The external source term $s(z)$ is given by 
\begin{equation}
  s(z) = \begin{cases}
    ac, & 0\leq z\leq \frac{1}{2},\\
    0,  & \text{otherwise}.
  \end{cases}
\end{equation}
In this problem, the material coupling term plays an important role and the relationship between the
temperature $T$ and the internal energy $e$ is 
\begin{equation}\label{eq:eT}
  e(T) = aT^4.
\end{equation}

\begin{figure}[htb]
  \centering 
  \subfloat[$N=2$]{
      \includegraphics[width=0.33\textwidth,height=0.16\textheight]{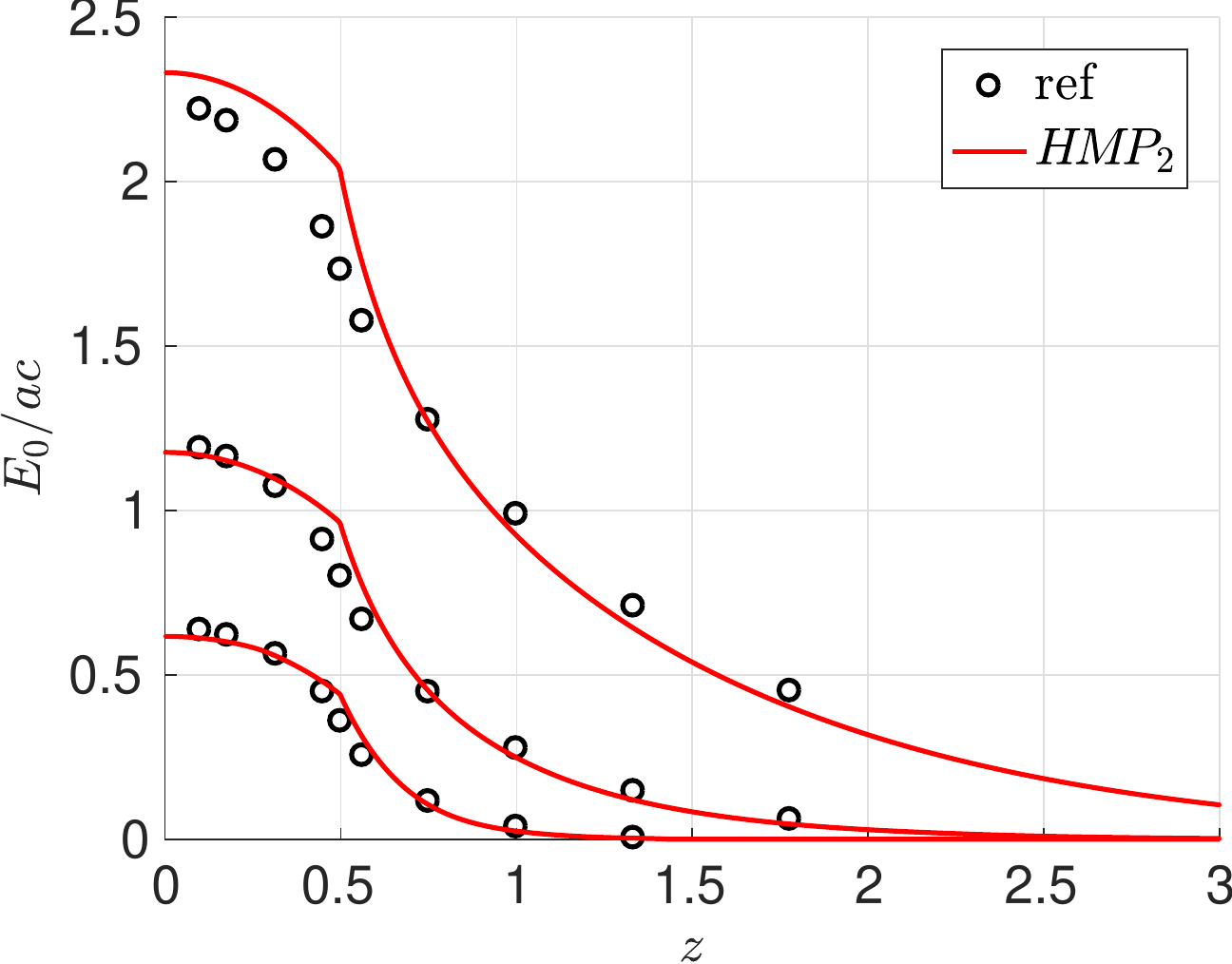}
  }
  \subfloat[$N=4$]{
      \includegraphics[width=0.33\textwidth,height=0.16\textheight]{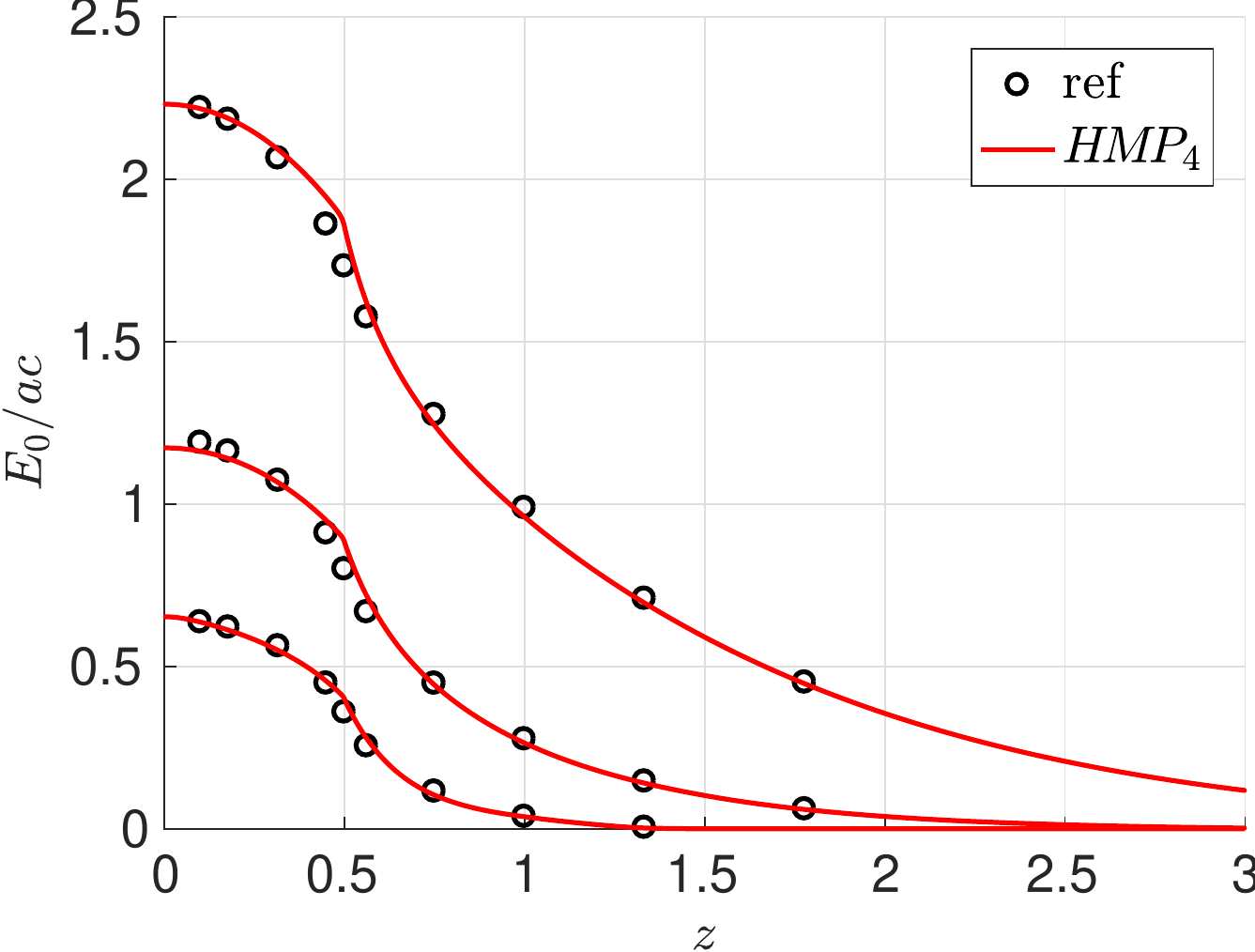}
  }
  \subfloat[$N=6$]{
      \includegraphics[width=0.33\textwidth,height=0.16\textheight]{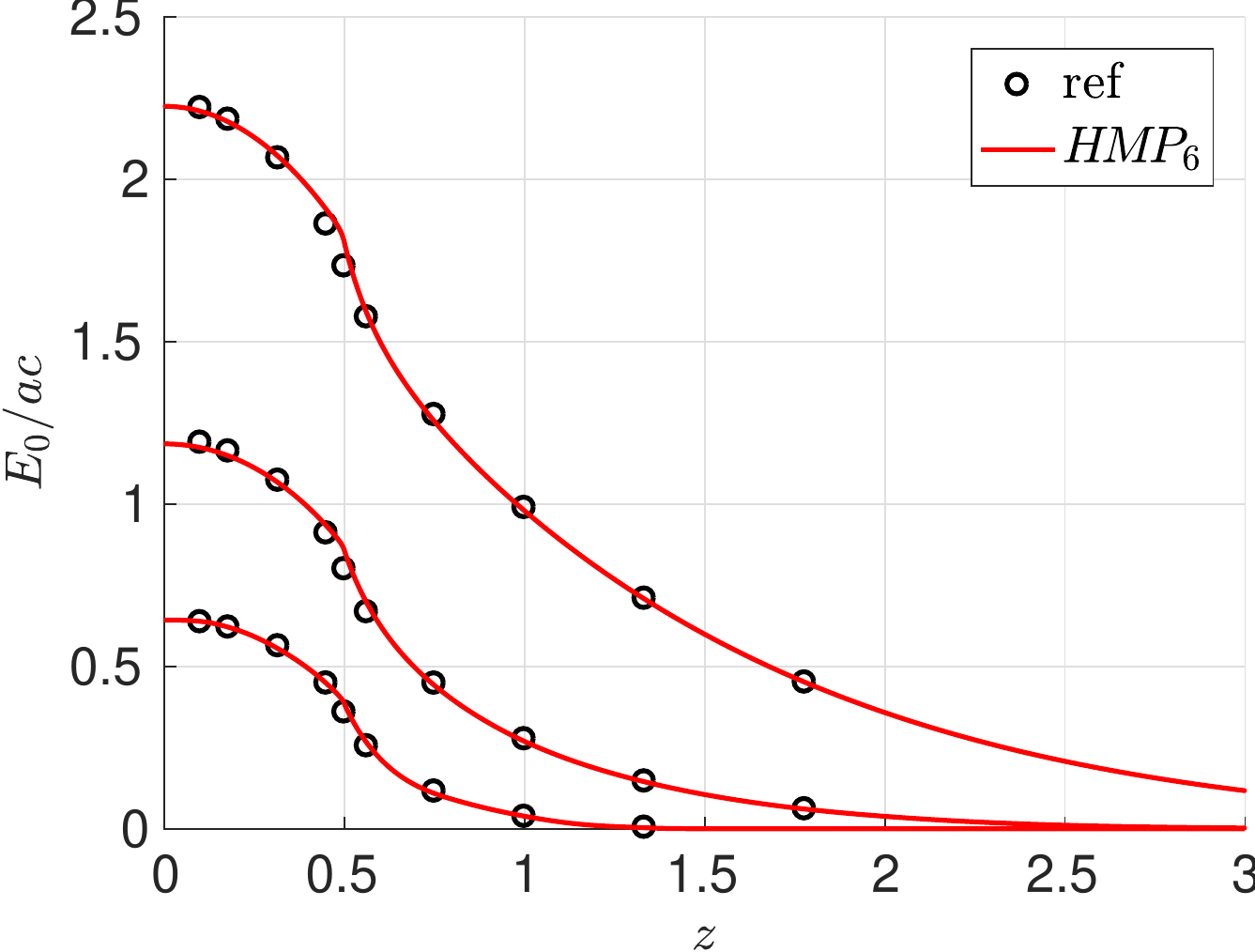}
  }
  \caption{\label{fig:SuOlson}
  Profiles of $E_0$ for the \HMPN model for the Su-Olson problem.
  The three groups of lines correspond to the end time $\ctend$ = 1, 3.16, and 10 (from bottom to
  top), respectively.}
\end{figure}

The reflective boundary condition is prescribed at the left boundary while the vacuum boundary
condition is prescribed at the right boundary. 
We use $N_{\text{cell}}=60000$ cells to simulate this problem till $\ctend=1$, 3.16, and 10.
The results of the \HMPN model and the semi-analytic solution taken form
\cite{su1997analytical} are presented in \cref{fig:SuOlson}.
The \HMPN model has a good agreement with the reference solution even when $N=2$ and gives a 
better approximation of the reference solution as $N$ increases.

\subsubsection{Anti-diffusive radiation flow}
The anti-diffusive flow \cite{mcclarren2010anti} is usually used to study the behaviour of radiative
shocks. The material consists of three parts and each part has its temperature. Precisely, 
the material temperature is given by 
\begin{equation}
  T^4 = 
  \begin{cases}
     a,  &z<0,\\
     1,  &0\leq z\leq z_0,\\
     b,  &z>z_0.
  \end{cases}
\end{equation}
In this problem, the parameters are set as $a=0.275$, $b=0.1$, $z_0=0.1$. The absorption and
scattering coefficients are $\sigma_a=1$ and $\sigma_s=0$, respectively, and the problem domain is
the whole space. The relationship between the temperature $T$ and the internal energy $e$ is
depicted by \eqref{eq:eT}. Then the analytical solution can be directly obtained as
\begin{itemize}
  \item When $0\leq z\leq z_0$,
    \begin{align}
      I(z,\mu)=\left\{
      \begin{aligned}
        & a e^{-z/\mu} + (1-e^{-z/\mu}),\quad & \mu>0,\\
        & b e^{-(z_0-z)/|\mu|} + (1-e^{-(z_0-z)/|\mu|}),\quad &\mu<0.
      \end{aligned}
      \right.
    \end{align}
  \item When $z<0$,
    \begin{align}
      I(z,\mu)=\left\{
      \begin{aligned}
        & a,\quad & \mu>0,\\
        & I(0,\mu) e^{z/|\mu|} +a (1-e^{z/|\mu|}),\quad &\mu<0.
      \end{aligned}
      \right.
    \end{align}
  \item When $z> z_0$,
    \begin{align}
      I(z,\mu)=\left\{
      \begin{aligned}
        & I(z_0,\mu) e^{-(z-z_0)/\mu} + b(1-e^{-(z-z_0)/\mu}),\quad & \mu>0,\\
        & b ,\quad &\mu<0.
      \end{aligned}
      \right.
    \end{align}

\end{itemize}
We simulate this problem with space step $h=1/200$ until steady state.
\Cref{fig:AntiDiffusive} presents the profiles of $E_0$ and $\frac{E_1}{E_0}$ of the \HMPN model and
the analytical solution.
Clearly, as $N$ increases, the solution of the \HMPN model convergents to reference solution very
fast, and when $N=8$, the \HMPN model is good enough to resolve the reference solution.

\begin{figure}[htb]
  \centering 
  \subfloat[$E_0$]{
      \includegraphics[width=0.45\textwidth,height=0.22\textheight]{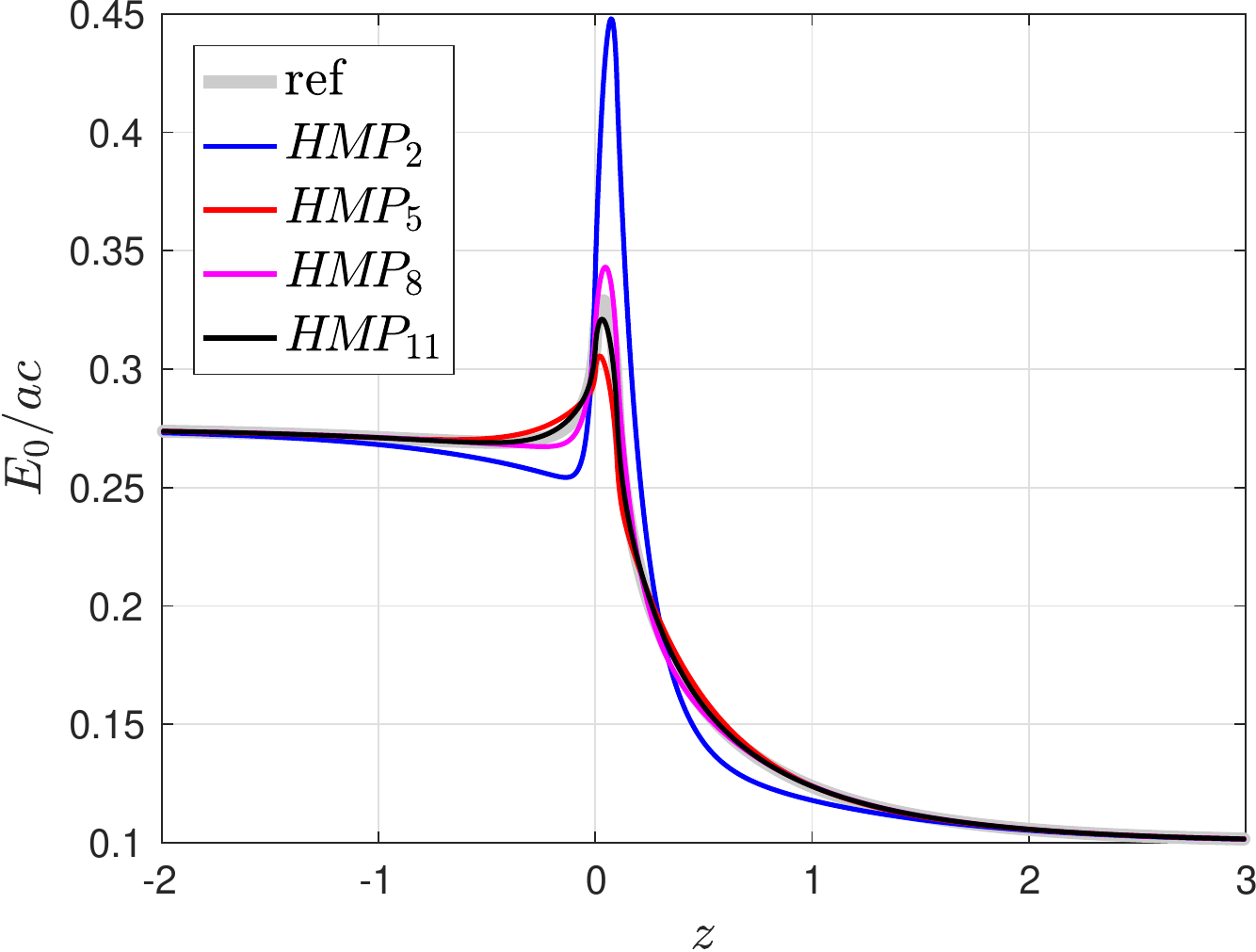}
  }
  \subfloat[$\dfrac{E_1}{E_0}$]{
      \includegraphics[width=0.45\textwidth,height=0.22\textheight]{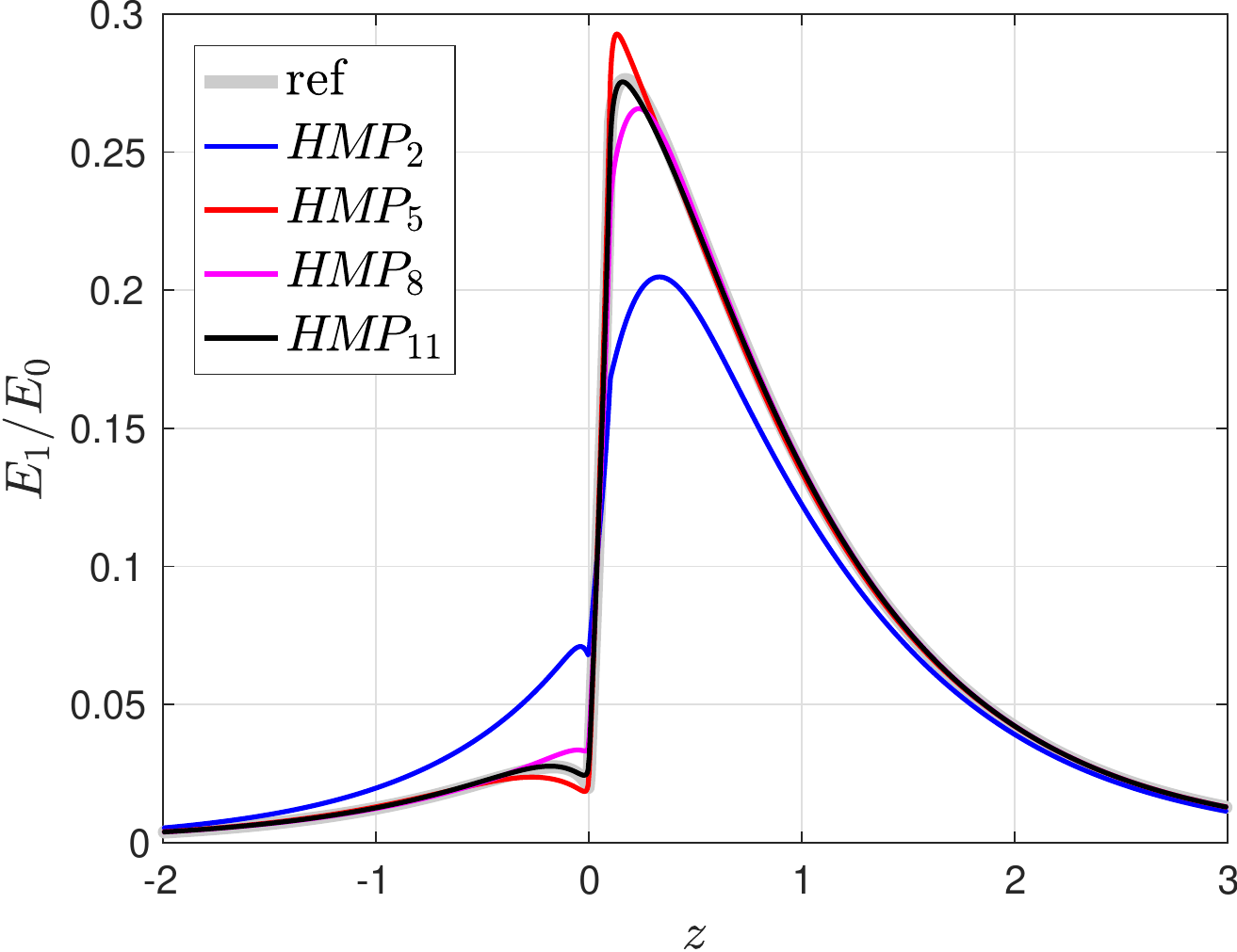}
  }
  \caption{\label{fig:AntiDiffusive}
  Profiles of $E_0$ and $\dfrac{E_1}{E_0}$ for the \HMPN model for the anti-diffusive radiation flow.
  }
\end{figure}

%% file: conclusion.tex
\section{Conclusion} \label{sec:conclusion}

We derived a new nonlinear model for RTE, which is a significant
progress than the \MPN model in \cite{MPN}. For the new model, not
only it is globally hyperbolic, but also some necessary physical
properties are preserved. Particularly, the regularization method in
this work was novel, and it extended the hyperbolic
regularization method in \cite{framework}.

The current work focuses on the globally hyperbolic moment system for the frequency-independent RTE
in slab geometry. Future work in this research area certainly contained: 1) an extension to the
three-dimensional case; 2) an extension to the frequency-dependent case; 3) the existence and
uniqueness of the solution of the new moment model.
Moreover, the regularization method in this work is worth further investigation. 
The novel regularization method generalizes the regularization method in \cite{Fan, Fan_new} and
takes more properties of the kinetic equation into account. It is also a future work to extend the
novel regularization method to a general framework, which reduces the kinetic equation to a globally
hyperbolic moment system by moment model reduction with maintaining physical properties of the
kinetic equation.

\section*{Acknowledgements}
The authors thank Dr. Weiming Li and Dr. Julian Koellermeier for valuable discussions.
The work of Y.F. is partially supported by the U.S. Department of Energy, Office
of Science, Office of Advanced Scientific Computing Research, Scientific
Discovery through Advanced Computing (SciDAC) program and the National Science
Foundation under award DMS-1818449.
The work of R.L. and L.Z. is partially supported by Science Challenge Project,
No. TZ2016002 and the National Natural Science Foundation of China (Grant No.
91630310 and 11421110001, 11421101).
